\documentclass[12pt,reqno]{amsart}
\usepackage{latexsym}
\usepackage{amssymb}
\usepackage{mathrsfs}
\usepackage{amsmath}
\usepackage{fancybox,color}
\usepackage{enumerate}
\usepackage[latin1]{inputenc}
\usepackage{soul} 
\usepackage[colorlinks=true, linkcolor=blue, citecolor=blue]{hyperref}

\usepackage{enumitem}

\usepackage{pgf,tikz}
\usetikzlibrary{patterns,arrows,calc,decorations.pathreplacing}

\newcommand{\ch}[1]{{\mbox{\raise 1pt\hbox{\Large$\chi$}}}_{\lower 1pt\hbox{$\scriptstyle #1$}}}

\usepackage[left=2.3cm,top=2.1cm,right=2.3cm]{geometry}

\def\1{\raisebox{2pt}{\rm{$\chi$}}}

\newtheorem{theorem}{Theorem}[section]
\newtheorem{corollary}[theorem]{Corollary}
\newtheorem{lemma}[theorem]{Lemma}
\newtheorem{claim}[theorem]{Claim}
\newtheorem{proposition}[theorem]{Proposition}
\theoremstyle{definition}
\newtheorem{definition}[theorem]{Definition}

\newtheorem{example}[theorem]{Example}

\theoremstyle{definition}
\newtheorem{remark}[theorem]{Remark}

\newcommand{\R}{{\mathbb R}}

\newcommand{\N}{{\mathbb N}}
\newcommand{\Z}{{\mathbb Z}}

\newcommand\diam{\operatorname{diam}}

\DeclareMathOperator{\Mu}{{Mu}}

\def\cprime{$'$}

\DeclareMathOperator*{\essinf}{ess\, inf}

{\catcode`p =12 \catcode`t =12 \gdef\eeaa#1pt{#1}}      
\def\accentadjtext#1{\setbox0\hbox{$#1$}\kern   
                \expandafter\eeaa\the\fontdimen1\textfont1 \ht0 }
\def\accentadjscript#1{\setbox0\hbox{$#1$}\kern 
                \expandafter\eeaa\the\fontdimen1\scriptfont1 \ht0 }
\def\accentadjscriptscript#1{\setbox0\hbox{$#1$}\kern   
                \expandafter\eeaa\the\fontdimen1\scriptscriptfont1 \ht0 }
\def\accentadjtextback#1{\setbox0\hbox{$#1$}\kern       
                -\expandafter\eeaa\the\fontdimen1\textfont1 \ht0 }
\def\accentadjscriptback#1{\setbox0\hbox{$#1$}\kern     
                -\expandafter\eeaa\the\fontdimen1\scriptfont1 \ht0 }
\def\accentadjscriptscriptback#1{\setbox0\hbox{$#1$}\kern 
                -\expandafter\eeaa\the\fontdimen1\scriptscriptfont1 \ht0 }
\def\itoverline#1{{\mathsurround0pt\mathchoice
        {\rlap{$\accentadjtext{\displaystyle #1}
                \accentadjtext{\vrule height1.593pt}
                \overline{\phantom{\displaystyle #1}
                \accentadjtextback{\displaystyle #1}}$}{#1}}
        {\rlap{$\accentadjtext{\textstyle #1}
                \accentadjtext{\vrule height1.593pt}
                \overline{\phantom{\textstyle #1}
                \accentadjtextback{\textstyle #1}}$}{#1}}
        {\rlap{$\accentadjscript{\scriptstyle #1}
                \accentadjscript{\vrule height1.593pt}
                \overline{\phantom{\scriptstyle #1}
                \accentadjscriptback{\scriptstyle #1}}$}{#1}}
        {\rlap{$\accentadjscriptscript{\scriptscriptstyle #1}
                \accentadjscriptscript{\vrule height1.593pt}
                \overline{\phantom{\scriptscriptstyle #1}
                \accentadjscriptscriptback{\scriptscriptstyle #1}}$}{#1}}}}

\newcommand{\iol}{\itoverline}

\def\1{\raisebox{2pt}{\rm{$\chi$}}}

\def\vint_#1{\mathchoice%
        {\mathop{\kern 0.2em\vrule width 0.6em height 0.69678ex depth -0.58065ex
                \kern -0.8em \intop}\nolimits_{\kern -0.4em#1}}%
        {\mathop{\kern 0.1em\vrule width 0.5em height 0.69678ex depth -0.60387ex
                \kern -0.6em \intop}\nolimits_{#1}}%
        {\mathop{\kern 0.1em\vrule width 0.5em height 0.69678ex
            depth -0.60387ex
                \kern -0.6em \intop}\nolimits_{#1}}%
        {\mathop{\kern 0.1em\vrule width 0.5em height 0.69678ex depth -0.60387ex
                \kern -0.6em \intop}\nolimits_{#1}}}
\def\vintslides_#1{\mathchoice%
        {\mathop{\kern 0.1em\vrule width 0.5em height 0.697ex depth -0.581ex
                \kern -0.6em \intop}\nolimits_{\kern -0.4em#1}}%
        {\mathop{\kern 0.1em\vrule width 0.3em height 0.697ex depth -0.604ex
                \kern -0.4em \intop}\nolimits_{#1}}%
        {\mathop{\kern 0.1em\vrule width 0.3em height 0.697ex depth -0.604ex
                \kern -0.4em \intop}\nolimits_{#1}}%
        {\mathop{\kern 0.1em\vrule width 0.3em height 0.697ex depth -0.604ex
                \kern -0.4em \intop}\nolimits_{#1}}}

\newcommand{\intav}{\vint}
\newcommand{\ud}[0]{\,\mathrm{d}}

\newcommand{\dist}{\operatorname{dist}}

 \title[Weak porosity on metric measure spaces]{Weak porosity on metric measure spaces}
\date{August 21, 2025}

\author[Carlos Mudarra]{Carlos Mudarra}

\address{Department of Mathematical Sciences, Norwegian University of Science and Technology, 7941 Trondheim, Norway} \email{carlos.mudarra@ntnu.no}

\makeatletter
\@namedef{subjclassname@2020}{{\mdseries 2020} Mathematics Subject Classification}
\makeatother



\begin{document}

\begin{abstract}
We characterize the subsets $E$ of a metric space $X$ with doubling measure whose distance function to some negative power $\dist(\cdot,E)^{-\alpha}$ belongs to the Muckenhoupt $A_1$ class of weights in $X$. To this end, we introduce the weakly porous sets in this setting, and show that, along with certain doubling-type conditions for the sizes of the largest $E$-free holes, these sets characterize the mentioned $A_1$-property. We exhibit examples showing the optimality of these conditions, and simplify them in the particular case where the underlying measure satisfies a qualitative annular decay property. 
In addition, we use some of these distance functions as a new and simple method to explicitly construct doubling weights in $\R^n$ that do \textit{not} belong to $A_\infty.$ 
\end{abstract}

\maketitle

\section{Introduction and main results}

The classes of Muckenhoupt $A_p$ weights play a central role in harmonic and geometric analysis, and in potential theory both in $\R^n$ and in metric measure spaces. Among other relevant properties, they characterize the weights $w$ for which the Hardy-Littlewood maximal function defines a bounded operator between the appropriate weighted $L^p$-spaces. Also, provided the underlying measure supports a $(1,1)$-Poincar\'{e} inequality and the weight $w$ belongs to the $A_p$ class, then the corresponding $w$-weighted $(1,p)$-Poincar\'{e} inequality holds true; see e.g., \cite{B01}. Moreover, they are crucial in understanding the structure of the functions of bounded mean oscillation and the classes of quasiconformal mappings. For background about the theory of $A_p$ weights and related issues in harmonic analysis and potential theory in Euclidean and in metric measure spaces, we refer to the monographs by A. Bj\"orn and J. Bj\"orn \cite{BBbook}, Heinonen, Kilpel\"ainen, and Martio \cite{HKMbook}, Kinnunen, Lehrb\"ack, and V\"ah\"akangas \cite{KLVbook}, and Str\"omberg and Torchinsky \cite{ST}. 

\smallskip

Our framework will be a complete metric space $(X,d)$ equipped with a doubling Borel regular measure $\mu.$ We are interested in the $A_p$-properties of the weights of the form $w_{\alpha,E}(x) = \dist(x,E)^{-\alpha},$ $x\in X,$ for subsets $E $ of $X.$ The fundamental starting point for the results in this paper is our recent work \cite{ALMV22}. There, we considered the class of \textit{weakly porous} sets in $\R^n$ equipped with the Lebesgue measure, and proved that this is precisely the class of sets $E$ for which $w_{\alpha,E}\in A_1(\R^n)$ for some $\alpha>0.$ Moreover, we introduced the concept of Muckenhoupt exponent for a set $E$, which permitted us to quantify the range of those $\alpha$ for which $w_{\alpha,E}\in A_1(\R^n).$ Furthermore, in \cite{ALMV22} we constructed examples of sets $E$ for which $\dist(\cdot,E)^{-\alpha} \in A_p \setminus A_1$ for every $\alpha>0$ and $p>1.$

Previously, Vasin \cite{V03} obtained the corresponding $A_1$-characterization for subsets $E$ contained in the one-dimensional torus $\mathbb{T}$. Also, in \cite{DILTV19}, Dyda, Ihnatsyeva, Lehrb\"{a}ck, Tuominen, and V\"{a}h\"{a}kangas studied the $A_p$-properties of distance functions to \textit{porous} sets, and applied their results to obtain certain weighted Hardy-Sobolev inequalities. See also Semmes' work \cite{S96} (in Section 4) for results on the \textit{strong $A_\infty$ condition} of some weights given by distance functions to uniformly disconnected sets. A detailed study of the Minkowski and Hausdorff dimensions of porous and locally porous sets can be found in the paper of Salli \cite{Sa91}. For results concerning the Assouad dimension and spectrum of porous and/or fractals sets, and their distorsion under quasiconformal maps, we refer, for instance, to the papers of Chrontsios-Garitsis and Tyson \cite{CT23}, Luukkainen \cite{L98}, and Fraser and Yu \cite{FY22}, and the references therein. Finally, see the monograph \cite{F2014book} by Falconer, for a thorough exposition of fractal geometry and their applications.

The Muckenhoupt $A_1(X)$ class of weights is the smallest among all the classes $A_p(X)$, with $p \geq 1.$ We recall that a weight $w$ (a nonnegative locally integrable function in $X$) belongs to $A_1(X)$ if there exists some constant $C>0$ for which
$$
\intav_B w(x) \ud \mu(x) \leq C \essinf_{x\in B} w(x), \quad \text{for every ball} \quad B \subset X.
$$
This $A_1$-condition characterizes the weights $w$ for which the maximal function $M$ defines a bounded operator between the weighted spaces $M: L^1(X, w \ud \mu) \to L^{1,\infty}(X, w \ud \mu).$ 

\smallskip

The main goal of this paper is to obtain characterizations of the property $w_{\alpha,E}\in A_1(X)$ for some $\alpha>0$ similar to those in \cite{ALMV22}, when the underlying space is a metric space $(X,d,\mu)$ with a doubling measure. In order to solve this problem, we consider the following size of the \textit{maximal $E$-free holes}, namely, if $B=B(x,R)$, we set
$$
h_E(B) := \sup \lbrace 0<r \leq 2R \: : \: \text{there exists } y\in X \text{ with } B(y,r) \subset B(x,R) \setminus E \rbrace.
$$
Then we define the \textit{weakly porous sets} of $X$ as those $E \subset X$ for which there are two constants $c, \delta \in (0,1)$ such that every open ball $B=B(x,R) \subset X$ contains  finitely many non-overlapping open balls $\lbrace B_i=B(x_i,r_i) \rbrace_{i=1}^N $ containing no point of $E$ with
$$
\delta h_E(B) \leq r_i \leq 2R, \quad i=1,\ldots,N, \quad \text{and}  \quad  \sum_{i=1}^N \mu(B_i) \geq c \mu(B).
$$
Unlike in the case where $X=\R^n,$ the weak porosity does not provide a characterization by itself, and certain doubling-type behavior for the quantity $h_E$ is needed, that is, $h_E(2B) \lesssim h_E(B)$ for all balls $B.$ Along with this doubling property, the condition of weak porosity provides a complete characterization, and our first main result is the following.

\begin{theorem}\label{maintheoremallspaces}
Let $E\subset X$ be a subset. Then, the following statements are equivalent:
\begin{enumerate}[label=\normalfont{(\Alph*)}]
\item \label{maintheoremconditionwp} $E$ is weakly porous and $h_E$ satisfies the doubling condition.
\item \label{maintheoremconditionA1} There exists $\alpha>0$ for which $\dist(\cdot, E)^{-\alpha} \in A_1(X)$.
\end{enumerate}
\end{theorem}

The proof of this theorem is given in Section \ref{sectionsufficiencymaintheorem}. Notice that, for positive powers $\beta>0,$ the weight $v=\dist(\cdot, E)^{\beta} $ belongs to $A_1(X)$ only when $E$ is dense in $X$, i.e., when $v$ is identically zero. The definitions of a weakly porous set and the doubling condition for $h_E$ are restated in Definitions \ref{definitionweakporballs} and \ref{definitionofdoubling} below.

When $X=\R^n$ is equipped with the Lebesgue measure, we proved in \cite{ALMV22} that $h_E$ always satisfies the doubling condition for weakly porous sets $E$, and thus the second assumption in Theorem \ref{maintheoremallspaces}\ref{maintheoremconditionwp} is unnecessary in this setting. However, there are connected metric spaces with doubling measure containing subsets $E$ that are weakly porous and yet $h_E$ is not doubling. This will be shown in this paper via Example \ref{examplenotdoublinghole} in Section \ref{sectiondoublingholesexamples}, where we further prove that for such weakly porous set $E,$ the functions $\dist(\cdot,E)^{-\alpha}$, $\alpha>0,$ do not even define doubling weights.

In addition, we give an example of a set $E \subset \R$ for which all the functions of the form $\dist(\cdot,E)^{-\alpha},$ with $\alpha <1,$ define weights in $\R,$ the maximal $E$-free hole $h_E$ is doubling, and still $E$ is not weakly porous. Thus we show that the properties of weak porosity and doubling maximal holes are independent of each other in the general setting. Remarkably, this set $E$ is so that, for all $\alpha<1,$ $\alpha \neq 0,$ $\dist(\cdot,E)^{-\alpha}$ defines a doubling weight that does not belong to any $A_p(\R)$ class, thus providing a new family of doubling weight functions that are not in $A_\infty(\R)$. The first example of such weights was by Fefferman and Muckenhoupt \cite{FeMu74}.

\begin{theorem}\label{thm:exampleDoublingNotAinfty}
Endow $\R$ with the usual distance and the Lebesgue measure. Then there exists $E \subset \R$ with doubling maximal holes $h_E$, and so that, for all $ \alpha <1,$ $\alpha \neq 0,$ the function $w_\alpha:=\dist(\cdot, E)^{-\alpha} $ is a doubling weight in $\R$ that does \emph{not} belong to $A_\infty(\R)$.

In particular, $E$ is not weakly porous. 
\end{theorem}

The proof of this theorem is given in Section \ref{sectionbigcounterexample}. This theorem has an obvious generalization to $\R^n,$ by replacing $E \subset \R$ with $  E \times \R^{n-1} \subset \R^n$; see Remark \ref{rem:generalizationR^nExample}.

We will reformulate Theorem \ref{thm:exampleDoublingNotAinfty} with more details in Example \ref{exampledoublignholesdoesnotimplywp}, but here we offer a brief description of the set $E.$ We begin with the two-point set $E_0= \lbrace 0, 1 \rbrace,$ and define $E_1$ by \textit{chaining}, from left to right, the set $E_{0},$ a $1/2$-contraction of a copy of $E_{0}$, and finally another copy of $E_{0}.$ With the same procedure we define recursively $E_2,E_3,\ldots,$ and let $E^+$ be the union of all the $E_n$'s. It is precisely this \textit{palindromic} distribution of the various copies of $E_0$ which allows to show the doubling properties of the theorem. The desired set $E$ is merely the symmetrization of $E^+$ about the origin.

\smallskip

It is also worth noticing the beautiful example by Wik \cite[p. 253]{W89} of a doubling weight $w$ in $\R$ with $w \notin A_\infty(\R).$ Wik's example can be seen as a distance function to a set $F$ constructed by chaining bounded sets $\lbrace F_n \rbrace_n$, where each $F_n$ contains in turn infinitely many copies of (arbitrarily) small $2^{-k}$-contractions of the previous $F_{n-1}.$ The initial set $F_0$ is also a two-point set. As in our construction, these sets are conveniently chained in a symmetric way. The set $E$ in Theorem \ref{thm:exampleDoublingNotAinfty} is, perhaps, simpler, as each of the main \textit{blocks} $E_n$ of our construction is a finite set whose smallest gap has size $2^{-n}.$ Also, the advantage of our Theorem \ref{thm:exampleDoublingNotAinfty} is that, besides showing that $h_E$ is doubling, it tells us that \textit{all} the powers $ \lbrace \dist(\cdot, E)^{-\alpha} \rbrace_{\alpha \in (-\infty,1) \setminus \lbrace 0 \rbrace },$ are doubling weights outside the $A_\infty(\R)$ class.

\medskip

Going back to the general setting, if our metric measure space $(X,d,\mu)$ additionally satisfies a weak version of the standard annular decay property \eqref{measureannulardecayproperty}, the situation changes and the doubling condition for $h_E$ is a consequence of the weak porosity of $E$. The resulting $A_1$-characterization is therefore as clean as in $\R^n.$ 

\begin{theorem}\label{maintheoremannular}
Assume that $(X,d,\mu)$ satisfies the annular decay property \eqref{measureannulardecayproperty}, and let $E \subset X$ be a subset. Then, the following statements are equivalent:
\begin{enumerate}[label=\normalfont{(\Alph*')}]
\item\label{maintheoremannularwpcondition} $E$ is weakly porous.
\item There exists $\alpha>0$ for which $\dist(\cdot, E)^{-\alpha} \in A_1(X)$.
\end{enumerate}
\end{theorem}

For the proof of our results, we will formulate a condition of \textit{dyadic weak porosity} for subsets of $X$; see Definition \ref{definitiondyadicweakporosity}. This formulation is given in terms of a system of \textit{adjacent dyadic grids} in doubling metric spaces, which provide partitions of the space at all scales, into sets that enjoy many of the good properties of the usual dyadic grids of cubes in $\R^n.$ They have proven to be instrumental in harmonic and geometric analysis in metric spaces; as shown, for instance, by the classical and recent works of Christ \cite{C90}, Hyt\"onen and Kairema \cite{HK12}, and K\"aenm\"aki, Rajala and Suomala \cite{KRS12}. We additionally prove some estimates for the measures of the $\varepsilon$-boundaries of dyadic sets. These estimates were obtained by Christ \cite{C90} for the corresponding Christ's (open) cubes, as well as by Hyt\"onen and Kairema \cite{HK12} for the half-open cubes in a probabilistic form. 

\smallskip

Using this dyadic formulation of weak porosity, the idea of proof of Theorem \ref{maintheoremallspaces} is inspired by that in the Euclidean setting from \cite{ALMV22}. We decompose each dyadic set into carefully chosen $E$-free descendants that conveniently depend on the parameters of weak porosity of the set $E$. By means of a dyadic iteration and employing the assumption  that the maximal $E$-free holes $h_E$ is doubling, we obtain a discrete and quantitative lower estimate for the sizes of the maximal holes in terms of the maximal holes of those dyadic descendants; see Lemma \ref{keylemmasufficiency}. Finally, we prove the $A_1$-estimates over every dyadic set of the systems, making use of the mentioned estimates for the measures of the $\varepsilon$-boundaries. 

\smallskip

Dealing with systems of dyadic sets in our proofs will lead us to several technicalities. However, there are two main reasons for wanting to use them. 1) There is a suitable partial order between the dyadic sets that allows to derive many properties for the size of the largest (dyadic) $E$-free holes, as well as the key discrete lower estimate for the maximal holes of a dyadic set in terms of the ones of its descendants. 2) Even if we were able to obtain such a key estimate for some balls at various scales, it would not be possible to estimate integrals of the type $\intav_B \dist(\cdot, E)^{-\gamma}$ in terms of $\textrm{rad}(B)^{-\gamma}$, e.g., when $B$ is close to $E.$ This is mainly due to the fact that there is no control on the measure of an $\varepsilon$-neighborhood of $\partial B$ within $B,$ as the measure of a ball may be concentrated near the boundary in arbitrary metric spaces with doubling measure. However, the same integral averages replacing balls $B$ with a dyadic set $Q$ can be estimated from above by a power of the \textit{length} of $Q,$ if one uses conveniently the quantitative (with power-type rate) estimates we will prove for the measure of the $\varepsilon$-boundaries of the dyadic sets of the systems; see Sections \ref{sufficiencyestimatesclosecubes} and \ref{sectionestimatesnonfree}.

 \smallskip

As concerns the proof of Theorem \ref{maintheoremannular}, we first show a lemma of rescaling (Lemma \ref{rescalinglemma}) that essentially allows us to utilize $E$-free admissible balls with very small radii in the definition of weak porosity. Together with the assumption \eqref{measureannulardecayproperty}, this enables us to obtain certain absolute continuity of the maximal $E$-free holes of balls with respect to their radii, and this leads us to the doubling property for $h_E;$ see Lemma \ref{lemmaannulardecayweaklyporousdoubling}. Then, Theorem \ref{maintheoremannular} follows as a consequence of Theorem \ref{maintheoremallspaces}.

 \smallskip

The structure of the paper is as follows. In Section \ref{sectionpreliminaries}, we collect some basic properties of Muckenhoupt weights and also explain and prove the properties of the systems of dyadic sets we will be using. In Section \ref{sectiondefinitionswporosity}, we define the concepts of weak porosity and dyadic weak porosity, along with the doubling conditions for the size of the largest free holes associated with a set, and prove the relevant relationships. In Section \ref{sectionecessitymaintheorem}, we give the proof of the implication \ref{maintheoremconditionwp}$\implies$\ref{maintheoremconditionA1} in Theorem \ref{maintheoremallspaces}. Sections \ref{sectiondiscreteinequalityA1} and \ref{sectionsufficiencymaintheorem} are devoted to the proof of the converse implication of Theorem \ref{maintheoremallspaces}, Lemma \ref{keylemmasufficiency} being one of the key ingredients. 
Then, in Section \ref{sectionsufficiencymaintheorem}, we prove the desired $A_1$-estimates for the weight $\dist(\cdot,E)^{-\alpha}$. In Section \ref{sectionproofmainthmannularabsolute}, we prove Theorem \ref{maintheoremannular}, and we also show how a variant of the dyadic weak porosity condition, called \textit{absolute dyadic weak porosity}, provides another characterization of the sets $E$ for which $\dist(\cdot,E)^{-\alpha} \in A_1(X)$ in the general setting; see Corollary \ref{corollaryabsolutedyadicwp}. In Section \ref{sectiondoublingholesexamples}, we give two examples that prove that the properties of weak porosity and doubling maximal free holes do not imply each other, and also construct new doubling weights that are not in $A_\infty(\R),$ thus proving Theorem \ref{thm:exampleDoublingNotAinfty}. Finally, in Section \ref{sectionMuexponent}, we define the Muckenhoupt exponent associated with a set and we use it to determine the exponents $\alpha$ for which $\dist(\cdot, E)^{-\alpha} \in A_p$, for a given $1\leq p < \infty.$

\medskip

\textbf{Acknowledgements.} I am very grateful to Juha Lehrb\"ack and Antti V\"ah\"akangas for their valuable comments and suggestions on earlier versions of this manuscript. I also thank the referees for several suggestions that improved the exposition of the paper. Finally, I acknowledge financial support from the Academy of Finland via the project ``Geometric Analysis'' (grant No. 323960) and from the Research Council of Norway via the project ``Fourier Methods and Multiplicative Analysis'' (grant no. 334466).

\section{Preliminary concepts and systems of dyadic sets}\label{sectionpreliminaries}
The underlying space we will be working on is a complete metric space $(X,d).$ For every $x\in X$ and $r>0,$ the open ball centered at $x$ and with radius $r>0$ is denoted by
$$
B(x,r):=\lbrace y\in X \, : \, d(x,y) <r \rbrace.
$$
We will denote the distance between any two subsets $E,F \subset X$ by $\dist(E,F):= \inf\lbrace d(x,y) \, : \, x\in E,\,y\in F \rbrace,$ and the distance from a point $p\in X$ to a subset $E$ by $\dist(p,E):=\inf\lbrace d(p,x) \, : \, x\in E \rbrace.$

In addition, we endow $(X,d)$ with a complete Borel regular measure $\mu,$ satisfying the following standard doubling condition: the existence of a constant $C_\mu>0$ for which
\begin{equation} \label{Cmudoublingmeasure}
0<\mu\left( B(x,2r) \right) \leq C_\mu \,\mu \left( B(x,r) \right)< \infty, \quad \text{for all} \quad x\in X, \, r>0.
\end{equation}
Notice that then $0<\mu(B)<\infty$ for every ball $B.$ Condition \eqref{Cmudoublingmeasure} implies that $X$ is separable, and also proper, meaning that every closed and bounded set is compact. 

Since our metric space $(X,d)$ supports a doubling measure \eqref{Cmudoublingmeasure} $\mu$, the space $(X,d)$ is geometrically doubling, meaning that there exists $N\in \N$ so that for every $x\in X$ and $r>0,$ the ball $B(x,r)$ can be covered by at most $N$ balls $B(x_1, r/2), \ldots, B(x_N, r/2).$ The smallest of those $N\in \N$ is called the geometric doubling constant, and is  related to our constant $C_\mu.$

We will also use the notation
\[
 \vint_F w \ud \mu  = \frac 1 {\mu( F)} \int_F w(x) \ud\mu(x)
\]
for the mean value integral over a measurable set $F\subset X$ with $0<\mu(F)<\infty$. Also, in some cases we will denote by $w(F)$ the integral $\int_F w \ud \mu,$ for any weight $w$ and $F\subset X$ a measurable subset.

A weight $w$ in $X$ is any nonnegative and locally integrable function of $X.$ We next remind the reader the definition of an $A_p$ weight, for $1\leq p < \infty.$

 \begin{definition}\label{d.A1}
A weight $w$ in $X$ belongs to the Muckenhoupt class $A_1$ if
there exists a constant $C \geq 0$ such that
\begin{equation}\label{e.a_1}
 \vint_B w \ud\mu  \le C\, \essinf_{x\in B} w(x),
\end{equation}
for every ball $B \subset X$. The smallest possible constant $C$ in~\eqref{e.a_1} is called the $A_1$ constant of $w$, and it is denoted by $[w]_{A_1}$.

Also, for $1<p< \infty,$ a weight belongs to the Muckenhoupt class $A_p(X)$ if
\begin{equation}\label{e.A_pd}
[w]_{A_p}:=\sup_{B \subset X} \left( \vint_B w \ud\mu \right) \biggl(\vint_B w^{\frac{1}{1-p}}\ud\mu\biggr)^{p-1} < \infty,
\end{equation}
where the supremum is taken over all balls $B$ of $X.$ 
\end{definition}

The inclusions $A_p(X) \subset A_q(X)$ hold for any $1\leq p \leq q <\infty.$

A straightforward consequence of the H\"older inequality in combination with \eqref{Cmudoublingmeasure} is that a weight $w$ belonging to any of these $A_p$-classes satisfies the following \textit{doubling property}:
$$
w(B(x,2R)) \leq C(p,[w]_{A_p},C_\mu) w(B(x,R))\, \quad \text{for every} \quad x\in X, \, R>0.
$$
This fact will be used in Section \ref{sectiondoublingholesexamples}. A less elementary property of $A_p$ weights is their well-known \textit{self-improvement}; see \cite[Lemma 2.12]{KM21} in combination with \cite[Theorem I.15]{ST}. We state it together with an additional property involving powers of Muckenhoupt weights, the proof of which is identical to that of \cite[Lemma 2.3]{ALMV22}. 

\begin{proposition}\label{propositionselfimprovement}
{\em The following two properties hold.
\begin{enumerate}
\item If $w\in A_p(X)$ for $1\leq p < \infty,$ there exists $s=s(p,[w]_{A_p},C_\mu)>1$ so that $w^s\in A_p(X).$ 
 
\item If $w\in A_1(X),$ and $s\geq 0$ is so that $w^s \in \bigcup_{p >1} A_p(X),$ then $w^s\in A_1(X)$ as well.
\end{enumerate}}
\end{proposition}

\medskip

A fundamental tool in the proof of our results is the systems of adjacent dyadic sets. These systems can be constructed in any geometrically doubling metric (or even quasi-metric) space. They provide partitions of the space $X$ into mutually disjoint sets at every scale, and enjoy very similar properties to those of the dyadic cubes in $\R^n.$ The Christ's construction \cite{C90} provided an almost everywhere-covering of $X,$ via open sets satisfying these properties. More recent developments of this technique provided a finite collection of systems of dyadic \textit{half-open} sets, with the additional properties that their union cover \textit{all points} of $X$, and every ball of radius $r$ is contained in a dyadic set of one of these systems that has \textit{length} approximately $r.$ Our main reference here will be the paper by Hyt\"onen and Kairema \cite{HK12}, where a probabilistic approach is used to obtain further properties for this construction that lead to sharp weighted norm inequalities for the maximal function operator. It is also worth mentioning the work by K\"aenm\"aki, Rajala, and Suomala \cite[Theorem 2.1]{KRS12}, where similar constructions of dyadic sets were developed and applied to prove the existence of doubling measures in geometrically-doubling metric spaces.

\medskip

The following result is contained in \cite[Theorem 4.1--Remark 4.13]{HK12} for the more general setting of (geometrically) doubling quasi-metric spaces.

\begin{proposition}\label{theoremexistencedyadicgrids}
Fix $0< \theta < 1/1000$. For every $k\in \Z$ and $t=1,\ldots,T=T(\theta, C_\mu) \in \N,$ there exist a countable set of points $\{z_\alpha^{k,t}\mathbin{:}\alpha\in\mathcal{A}_k\} \subset X,$ and a collection $\mathcal{D}_t = \{ Q_\alpha^{k,t}\mathbin{:} \alpha\in \mathcal{A}_k\}$ of Borel subsets of $X$ such that
\begin{enumerate}[label=\normalfont{(\arabic*)}]
\item $X = \biguplus_{\alpha\in \mathcal{A}_k}Q_\alpha^{k,t}$ for every $k\in \Z$ and every $t=1\,\ldots,T$; 
\item If $Q_{\alpha}^{k,t}, Q_\beta^{l,t} \in \mathcal{D}_t$ and $l \geq k,$ then
$$
\textit{either } \: Q_\beta^{l,t} \subset Q_\alpha^{k,t} \:  \text{ or } \: Q_{\alpha}^{k,t} \cap Q_\beta^{l,t} =\emptyset,
$$
for every $t=1,\ldots,T$; 
\item\label{thm:dyadics.subper} for every $Q_\alpha^{k,t}\in\mathcal{D}_t$ we have 
\[ B(z_\alpha^{k,t}, a \theta^k) \subset Q_\alpha^{k,t} \subset B(z_\alpha^{k,t},A\theta^k),\]
where $0<a<A< \infty$ are absolute constants;
\item\label{thm:dyadics.containingball} for every ball $B=B(x,r)$ there exists a cube $Q\in \mathcal{D}:=\bigcup_{t=1}^T \mathcal{D}_t $ such that $B\subset Q $ and $Q  = Q_\alpha^{k, t}$ for some indices $\alpha$ and $t$, and $k$ being the unique integer such that $\theta^{k+2}<r\leq\theta^{k+1}$. 
\end{enumerate} 
\end{proposition}

To each of the collections $\mathcal{D}_t$ we will refer as a \textit{grid} or \textit{dyadic grid} of $\mathcal{D}.$ And each set $Q_\alpha^{k,t}$ of these collections will be called a \textit{dyadic cube} or a \textit{dyadic set}.

If $Q \in \mathcal{D},$ then $Q=Q_\alpha^{k,t}$ for some $k\in \Z$, $\alpha \in \mathcal{A}_k$, $t\in \lbrace 1,\ldots, T\rbrace.$ We will say that \textit{the generation of} $Q$ is $g(Q)=k,$ and that $z=z_\alpha^{k,t}$ is \textit{the reference point of} $Q=Q_\alpha^{k,t}.$

It may occur that two cubes $Q, Q' \in \mathcal{D}$ belong to different grids $\mathcal{D}_t,$ $\mathcal{D}_{t'}$ and still $Q=Q'$ as sets of points. In this case, $Q$ and $Q'$ are considered different cubes. Whenever we are given a cube $Q \in \mathcal{D},$ we understand that $Q$ is in a fixed grid $\mathcal{D}_t,$ and so its reference point and its generation are taken with respect to the grid $\mathcal{D}_t.$

If $Q_0 \in \mathcal{D},$ then $Q_0\in \mathcal{D}_t$ for some $t,$ and by $\mathcal{D}(Q_0)$ we denote the family of all the cubes $Q$ \textbf{in the same grid} $\mathcal{D}_t$ that are contained in $Q_0$ and such that $g(Q) \geq g(Q_0).$ Thus we can informally say that $\mathcal{D}(Q_0)$ is the collection of all the \textit{dyadic descendants} of $Q_0.$ Also, note that if $Q \in \mathcal{D}(Q_0),$ then $\mathcal{D}(Q) \subset \mathcal{D}(Q_0)$. We include the lower bound $g(Q) \geq g(Q_0)$ in the definition of $\mathcal{D}(Q_0)$, because in certain metric spaces $X$, (e.g., when $X$ is bounded), a dyadic set $Q_0$ may very well coincide with all its ancestors, leading us to generations unbounded from below. By the properties of the dyadic systems we have
$$
Q_0=\bigcup \lbrace  Q \in \mathcal{D}(Q_0) \, : \,  g(Q)=g(Q_0)+l \rbrace \quad \text{for every} \quad l\in \N,
$$
and the union is disjoint, for each $Q_0\in \mathcal{D}.$

\smallskip
It is important to notice that, even when two cubes $Q$, $Q'$ in a common grid $\mathcal{D}_t$ are equal as sets of points, their generations may differ. The generation function is monotonic only when the inclusions between the cubes are proper. This is clarified in the following remark. We also include an additional basic property that we will take into account in several arguments in the paper.

\begin{lemma}\label{observationsdyadicsets}
Let $P,P'\in \mathcal{D}_t$ be cubes in the same grid $\mathcal{D}_t.$ The following holds
\begin{enumerate}
\item[(i)] If $P \subsetneq P',$ then $g(P') < g(P).$
\item[(ii)] If $g(P)=g(P')$ and $P\cap P' \neq \emptyset$, then they are exactly the same cube, meaning that, not only $P=P'$ as sets, but also their labels $(\alpha, g(P))$ and $(\beta, g(P'))$ coincide.
\item[(iii)] If $P \in \mathcal{D}(P')$ and $l=g(P)-g(P') \geq 1,$ then there exists a (unique) chain of cubes $P=Q_0 \subset Q_1 \subset \cdots \subset Q_{l-1} \subset Q_l=P'$ so that $Q_{i-1} \in \mathcal{D}(Q_{i})$ and $g(Q_i)=g(Q_{i-1})-1$ for each $i=1,\ldots,l.$ 
\end{enumerate}
\end{lemma}
\begin{proof}
\item[(i)] If $P=Q_\alpha^{k,t},$ $P'=Q_\beta^{l,t},$ and $l \geq k,$ then either $Q_{\alpha}^{k,t} \cap Q_\beta^{l,t} =\emptyset$ or $Q_\beta^{l,t} \subset Q_\alpha^{k,t}$ by Proposition \ref{theoremexistencedyadicgrids}(2). Because $P \subsetneq P',$ this is a contradiction. Consequently, $l <k.$

\item[(ii)] Let $P=Q_\alpha^{k,t},$ $P'=Q_\beta^{k,t}$. By Proposition \ref{theoremexistencedyadicgrids}, $Q_\alpha^{k,t}$ and $Q_\beta^{k,t}$ coincide as sets of points. Now, if $\alpha \neq \beta,$ then $Q_\beta^{k,t} \subset X \setminus B(z_\alpha^{k,t},  a \theta^k)$; see \cite[Lemma 3.6]{HK12}. This contradicts the inclusion $Q_\alpha^{k,t} \supset B(z_\alpha^{k,t}, a \theta^k).$

\item[(iii)] By the comments subsequent to Proposition \ref{theoremexistencedyadicgrids}, $P'$ can be written as a disjoint union of cubes $Q \in \mathcal{D}(P')$ with $g(Q)=g(P)-1.$ And, by Proposition \ref{theoremexistencedyadicgrids}, the cube $P$ must be contained in one (and only one) of these cubes, say $Q_1.$ If $l=1,$ then $g(Q_1)=g(P') $ and by (ii), $Q_1$ and $P'$ are the same cube, and thus we obtain the desired chain of cubes. If $l>1,$ we write $P'$ as the disjoint union of cubes $Q \in \mathcal{D}(P')=$ with $g(Q)=g(Q_1)-1,$ where one (and only one) of these cubes, say $Q_2$, must contain $Q_1.$ Repeating this procedure, we obtain the chain $P=:Q_0 \subset Q_1 \subset \cdots   \subset Q_{l-1} \subset Q_l:=P',$ with the desired properties.

\end{proof}

The following result gives a quantitative (with power-type rate) estimate for the measure of the $\varepsilon$-boundary of our dyadic cubes, and it is crucial for our purposes. This property was proved by Christ \cite[Theorem 11]{C90} for the corresponding open cubes. Also, Hyt\"onen and Kairema \cite{HK12} stated and proved a probabilistic version for the cubes of Proposition \ref{theoremexistencedyadicgrids}. Because the cubes we are using in this paper are those from \cite{HK12}, which are larger than the ones from \cite{C90}, the desired estimates do not immediately follow from \cite[Theorem 11]{C90}, and so, for the sake of completeness, we provide the proof for all the cubes of the systems.

\begin{proposition}\label{furtherpropertiesofdyadicsets} 

For the cubes of the collection $\mathcal{D},$ the following property holds. There exist constants $C, \eta>0$ depending only on $\theta$ and $C_\mu$ so that
\begin{equation}\label{estimatemeasureboundarydyadicsets}
\mu \left( \lbrace x\in \overline{Q} \: : \: \dist(x, X \setminus Q) \leq \lambda \theta^{g(Q)} \rbrace \right) \leq C  \lambda^{\eta} \mu(Q), \quad \text{for all} \quad  \lambda>0, \, Q\in \mathcal{D}.
\end{equation} 

\end{proposition}

In the proof of the above proposition, we follow the strategy from \cite[Theorem 11]{C90}, adapted to these collections of half-open cubes. Let us briefly describe how the cubes $\lbrace Q_\alpha^k \rbrace_{k,\alpha}$ of a fixed dyadic grid are constructed, as we will employ some of their specific properties. For fixed constants $0 < c_0 \leq C_0< \infty,$ and for every $k\in \Z,$ let $\lbrace z_\alpha^k \rbrace_{\alpha\in \mathcal{A}_k}$ be a set of points in $X$ for which
$$
d(z_\alpha^k,z_\beta^k)\geq c_0 \theta^k \quad \text{if} \quad \alpha \neq \beta, \quad \text{and} \quad \min_\alpha d(x,z_\alpha^k) < C_0 \theta^k \quad \text{for every} \quad x\in X.
$$ 
There exists a partial order $\preceq$ in the set of all labels $ \lbrace (k,\alpha) \, : \, k\in \Z, \, \alpha\in \mathcal{A}_k \rbrace$ with the following properties:
$$
d(z_\beta^{k+1}, z_\alpha^k) < \frac{c_0}{2} \theta^k \implies (k+1,\beta) \preceq (k,\alpha);
$$ 
$$
(k+1,\beta) \preceq (k,\alpha) \implies d(z_\beta^{k+1},z_\alpha^k)<c_0 \theta^k;
$$
$$
\text{for every } (k+1,\beta) \text{ there is a unique } (k,\alpha) \text{ so that } (k+1,\beta) \preceq (k,\alpha);
$$
$$
\text{for every } (k,\alpha) \text{ there are between } 1 \text{ and } M \text{ labels } (k+1,\beta) \text{ so that } (k+1,\beta) \preceq (k,\alpha);
$$
\begin{align*}
  (l,\beta) \preceq (k,\alpha)  \iff &   l \geq k \text{ and exist } \gamma_k=\alpha, \gamma_{k+1}, \ldots, \gamma_{l-1}, \gamma_l =\beta, \\
& \text{with } (j+1, \gamma_{j+1}) \preceq (j,\gamma_j)  \text{ for } j=k,\ldots,l-1.
\end{align*}
Here $M\in \N$ is a constant only depending on the geometric doubling constant of $X,$ and thereby only depending on our constant $C_\mu$ from \eqref{Cmudoublingmeasure}.

The closed and open cubes $\overline{Q}_\alpha^k$, $\widetilde{Q}_\alpha^k$ are defined by
$$
\overline{Q}_\alpha^k= \overline{\bigcup_{(l,\beta) \preceq (k,\alpha)} \lbrace z_\beta^l \rbrace }, \quad \widetilde{Q}_\alpha^k: =X \setminus \bigcup_{\beta \in \mathcal{A}_k,\, \beta \neq \alpha} \overline{Q}_\beta^k.
$$
The cubes $\overline{Q}_\alpha^k$ and $\widetilde{Q}_\alpha^k$ are respectively closed and open subsets of $X.$ The formula defining the \textit{half-open} cubes $Q_\alpha^k$ is more intricate, and we do not recall it here; see \cite[Lemma 2.18]{HK12}. For these three types of cubes, the following properties hold. 
$$
B(z_\alpha^k, a \theta^k) \subset \widetilde{Q}_\alpha^k \subset Q_\alpha^k \subset \overline{Q}_\alpha^k \subset B(z_\alpha^k, A \theta^k), \quad  \text{for all} \quad k\in \Z, \,  \alpha\in \mathcal{A}_k;
$$
$$
 \textrm{int}(\overline{Q}_\alpha^k) = \widetilde{Q}_\alpha^k \quad \text{and} \quad \overline{\widetilde{Q}_\alpha^k} = \overline{Q}_\alpha^k, \quad  \text{for all} \quad k\in \Z, \,  \alpha\in \mathcal{A}_k;
 $$
$$
\overline{Q}_\alpha^k= \bigcup_{\beta\in \mathcal{A}_l, \, (l,\beta) \preceq (k,\alpha)} \overline{Q}_\beta^l,\quad \text{for all} \quad l, k\in \Z, \, l \geq k, \, \alpha\in \mathcal{A}_k;
$$
$$
\widetilde{Q}_\beta^l \subset \widetilde{Q}_\alpha^k  \text{ if }  (l,\beta) \preceq (k,\alpha) \text{ and }    \widetilde{Q}_\beta^l \cap \overline{Q}_\alpha^k= \overline{Q}_\beta^l \cap \widetilde{Q}_\alpha^k = \emptyset \text{ otherwise}; \,   l, k\in \Z, \, l \geq k, \, \alpha\in \mathcal{A}_k;
$$
$$
X= \bigcup_{\alpha\in \mathcal{A}_k} \overline{Q}_{\alpha}^k = \bigcup_{\alpha\in \mathcal{A}_k} Q_{\alpha}^k, \text{ with the latter union disjoint,} \quad \text{for all } k \in \Z.
$$
See \cite{HK12} for a proof of all the above properties. The constants $0<a<A < \infty$ depend on the initial choice of $c_0$ and $C_0,$ but being $\theta$ so small ($\theta < 1/1000$), all these constants can be considered to be absolute.

Notice that, because $\textrm{int}(Q_\alpha^k)=\widetilde{Q}_\alpha^k,$ one can replace the term $\dist(\cdot, X \setminus Q)$ with $\dist(\cdot, X \setminus \widetilde{Q})$ in the estimate \eqref{estimatemeasureboundarydyadicsets}. To prove Proposition \ref{furtherpropertiesofdyadicsets}, we first need a qualitative version of \eqref{estimatemeasureboundarydyadicsets}. The following key proposition corresponds to \cite[Lemma 5.12]{HK12}.

\begin{proposition}\cite[Lemma 5.12]{HK12}\label{chainsjoiningpointsclosetoborder}
If parameters $N\in \N \cup \lbrace 0 \rbrace,$ and $\lambda >0$ are so that $12\lambda \leq c_0 \theta^N,$ and $(k,\alpha)$ is a label so that $x\in \overline{Q}_\alpha^k$ with $\dist(x, X \setminus \widetilde{Q}_\alpha^k)< \lambda \theta^k,$ then for any chain of labels
$$
(k+N,\sigma) =(k+N, \sigma_{k+N}) \preceq \cdots \preceq (k+1, \sigma_{k+1}) \preceq (k,\alpha)
$$
with $x\in \overline{Q}_{\sigma}^{k+N},$ one has $d(z_{\sigma_j}^j, z_{\sigma_i}^i) \geq \varepsilon_1\theta^j$ for all $k \leq j<i \leq k+N;$ where $\varepsilon_1= c_0/12.$ 
\end{proposition}

\begin{proposition}\label{qualitativeestimatemeasureboudaries}
For every $\varepsilon>0$ there exists $\lambda \in (0,1]$ so that 
$$
\mu\left(\lbrace x\in \overline{Q}_\alpha^k \, : \, \dist(x, X\setminus \widetilde{Q}_\alpha^k) \leq \lambda \theta^k \rbrace \right) < \varepsilon  \mu(Q_\alpha^k), \quad \text{for all} \quad \alpha\in \mathcal{A}_k, \, k \in \Z.
$$
\end{proposition}
\begin{proof}
Fix a cube $\overline{Q}_\alpha^k$ and denote $E_\lambda=\lbrace x\in \overline{Q}_\alpha^k \, : \, \dist(x, X\setminus \widetilde{Q}_\alpha^k) \leq \lambda \theta^k \rbrace$ for $\lambda>0.$ Given a natural $N\in \N,$ let $\lambda \in (0, \frac{c_0}{12}\theta^N)$, and define, for $x\in E_\lambda$, the set of indices
$$
\mathcal{I}_{k+N}(x) =\lbrace \sigma \in \mathcal{A}_{k+N} \, : \, (k+N, \sigma) \preceq (k,\alpha) \text{ and } x\in \overline{Q}_\sigma^{k+N} \rbrace.
$$ 
Notice that $\mathcal{I}_{k+N}(x)$ is non-empty, since $\overline{Q}_\alpha^k = \bigcup_{ \lbrace \sigma \in \mathcal{A}_{k+N}\, : \, (k+N, \sigma ) \preceq (k,\alpha) \rbrace } \overline{Q}_\sigma^{k+N}.$ Also, for $x\in E_\lambda,$ and $k \leq j \leq k+N,$ we define 
$$
\mathcal{I}_j^\lambda=\lbrace \gamma \in \mathcal{A}_j \, : \, \text{there are } x\in E_\lambda \text{ and } (k+N,\sigma) \preceq \cdots \preceq (j,\gamma) \preceq \cdots \preceq (k,\alpha) \text{ with } \sigma \in \mathcal{I}_{k+N}(x) \rbrace.
$$
We now make the following claim: there exists $0<\varepsilon_2 \leq \min\lbrace a, \varepsilon_1, c_0/2 \rbrace$ so that
\begin{equation}\label{claimdisjointedfamily}
\text{the family of balls } \: \lbrace B(z_\gamma^j, \varepsilon_2 \theta^j) \,  : \, \gamma \in \mathcal{I}_j^\lambda, \: k \leq j \leq k+N \rbrace \:\ \text{ is disjointed}.
\end{equation}
 
Indeed, for every two different indices $\gamma \in \mathcal{I}_j^\lambda,$ $\beta \in \mathcal{I}_i^\lambda$ with $k \leq j \leq i \leq k+N,$ let us show that  $B(z_\gamma^j, \varepsilon_2 \theta^j) $ and $B(z_\beta^i, \varepsilon_2 \theta^j)$ are disjoint, for a constant $\varepsilon_2>0$. In the case $i=j,$ because $\gamma \neq \beta,$ the points are $z_\gamma^j$ and $z_\beta^j$ are $c_0 \theta^j$-separated, and the claim follows if $\varepsilon_2 \leq c_0/2.$ Now, if $i>j,$ let us denote $\gamma_j= \gamma$ and $\beta_i=\beta.$ Let $x,y\in E_\lambda$ be so that $(j,\gamma_j)$ belongs to a chain $(k+N,\gamma_{k+N}) \preceq \cdots \preceq (k,\gamma_N)=(k,\alpha)$ with $\gamma_{k+N}\in \mathcal{I}_{k+N}(x)$, and so that $(i, \beta_i)$ belongs to a chain $(k+N,\beta_{k+N}) \preceq \cdots \preceq (k,\beta_N)=(k,\alpha)$ with $\beta_{k+N}\in \mathcal{I}_{k+N}(y).$ Because $x\in E_\lambda \cap \overline{Q}_{\gamma_{k+N}}^{k+N},$ $y\in E_\lambda \cap \overline{Q}_{\beta_{k+N}}^{k+N}$ and $\lambda< \frac{c_0 \theta^N}{12},$ Proposition \ref{chainsjoiningpointsclosetoborder} can be applied to obtain
$$
d(z_{\gamma_j}^j, z_{\gamma_i}^i) \geq \varepsilon_1 \theta^j \quad \text{ and } \quad d(z_{\beta_j}^j, z_{\beta_i}^i) \geq \varepsilon_1 \theta^j.
$$
If $\beta_j = \gamma_j$, then we immediately get $d(z_{\gamma_j}^j, z_{\beta_i}^i) =  d(z_{\beta_j}^j, z_{\beta_i}^i) \geq \varepsilon_1 \theta^j.$ In the other case, $\beta_j \neq \gamma_j$, and since $(i,\beta_i) \preceq (j,\beta_j),$ then $(i,\beta_i) \not\preceq (j,\gamma_j)$. The properties of the cubes $\widetilde{Q}, \overline{Q}$ imply $\widetilde{Q}_{\beta_i}^i \cap \overline{Q}_{\gamma_j}^j = \emptyset,$ and so the balls $B(z_{\beta_i}^i,a \theta^i)$ and $B(z_{\gamma_j}^j,a \theta^j)$ are disjoint too. In either case we see that there is $\varepsilon_2>0$ so that $B(z_\beta^i, \varepsilon_2 \theta^j) \cap B(z_\gamma^j, \varepsilon_2 \theta^j) = \emptyset,$ proving claim \eqref{claimdisjointedfamily}.

Before starting with the desired estimates for the measures, observe that $\mathcal{I}_{k+N}^\lambda$ can be written as
$$
\mathcal{I}_{k+N}^\lambda=\lbrace \sigma \in \mathcal{A}_{k+N} \, : \, (k+N, \sigma) \preceq (k,\alpha) \text{ and } x\in \overline{Q}_\sigma^{k+N} \text{ for some } x\in E_\lambda \rbrace,
$$
and thus
$$
E_\lambda \subset \bigcup_{\sigma \in \mathcal{I}_{k+N}^\lambda} \overline{Q}_\sigma^{k+N} \subset \bigcup_{\sigma \in \mathcal{I}_{k+N}^\lambda} B(z_\sigma^{k+N}, A\theta^{k+N}).
$$
In addition, $\mathcal{I}_{k+N}^\lambda \subset \bigcup_{\gamma \in \mathcal{I}_j^\lambda} \lbrace \sigma \in \mathcal{A}_{k+N} \, : \, (k+N,\sigma) \preceq (j,\gamma) \rbrace$ for every $k \leq j \leq k+N.$ We will also take into account that for every $k \leq l \leq k+N,$ and $\beta \in \mathcal{I}_l^\lambda,$ the ball $B(z_\beta^l, \varepsilon_2 \theta^l)$ is contained in $\widetilde{Q}_\beta^l$, and thereby contained in $\widetilde{Q}_\gamma^j$ whenever $(l,\beta) \preceq (j,\gamma).$ All these observations lead us to 
\begin{align*}
\mu(E_\lambda) & \leq   \sum_{\sigma \in \mathcal{I}_{k+N}^\lambda} \mu\left( B(z_\sigma^{k+N}, A \theta^{k+N}) \right) \leq C(C_\mu,A,\theta,\varepsilon_2) \sum_{\sigma \in \mathcal{I}_{k+N}^\lambda} \mu\left( B(z_\sigma^{k+N}, \varepsilon_2 \theta^{k+N}) \right) \\
& \leq C \sum_{\gamma\in \mathcal{I}_j^\lambda}\sum_{\substack{\sigma \in \mathcal{A}_{k+N} \\ (k+N,\sigma)\preceq (j,\gamma)}} \mu\left( B(z_\sigma^{k+N}, \varepsilon_2 \theta^{k+N}) \right) \leq C \sum_{\gamma\in \mathcal{I}_j^\lambda} \mu( \widetilde{Q}_\gamma^j ) \\
& \leq C \sum_{\gamma\in \mathcal{I}_j^\lambda} \mu( B(z_\gamma^j, A \theta^j) ) \leq C^2 \sum_{\gamma\in \mathcal{I}_j^\lambda} \mu( B(z_\gamma^j, \varepsilon_2 \theta^j) ).
\end{align*}
These estimates, along with \eqref{claimdisjointedfamily}, permit to write
$$
(N +1) \mu(E_\lambda) \leq C^2 \sum_{j=k}^{k+N} \sum_{\gamma\in \mathcal{I}_j^\lambda} \mu( B(z_\gamma^j, \varepsilon_2 \theta^j) ) = C^2 \mu \left(  \bigcup_{k \leq j \leq k+N} \bigcup_{\gamma\in \mathcal{I}_j^\lambda} B(z_\gamma^j, \varepsilon_2 \theta^j) \right) \leq C^2 \mu(Q_\alpha^k),
 $$
 which holds for every $\alpha,k$ whenever $\lambda \in (0, \frac{c_0}{12} \theta^N).$ In the last inequality we have used the fact that all the balls $B(z_\gamma^j, \varepsilon_2 \theta^j) $ are contained in $\widetilde{Q}_\gamma^j$, and so, in $\widetilde{Q}_\alpha^k$ as well because $(j,\gamma) \preceq (k,\alpha)$. 
\end{proof}

Using the technique from \cite{C90}, where a version of Proposition \ref{qualitativeestimatemeasureboudaries} for the open cubes is combined with an iterative argument, we next show Proposition \ref{furtherpropertiesofdyadicsets}.

\begin{proof}[Proof of Proposition \ref{furtherpropertiesofdyadicsets}] Denote $E_\lambda(\overline{Q}_\alpha^k)= \lbrace x\in \overline{Q}_\alpha^k \, : \, \dist(x, X\setminus \widetilde{Q}_\alpha^k) \leq \lambda \theta^k \rbrace.$ The desired estimate in the case $\lambda \geq \theta$ is true because in this case we can simply estimate by
$$
\mu(E_\lambda(\overline{Q}_\alpha^k)) \leq \mu(\overline{Q}_\alpha^k) \leq C(C_\mu) \mu(Q_\alpha^k) \leq C(C_\mu) \theta^{-\eta} \lambda^\eta\mu(Q_\alpha^k),
$$
using any $\eta>0.$ 
So let us assume $\lambda \in (0,\theta)$ in the rest of the proof. 
Define for every $j \in \N\cup \lbrace 0 \rbrace,$ the set of labels
$$
\mathcal{I}_j(\overline{Q}_\alpha^k) := \lbrace (k+j,\beta) \preceq (k,\alpha) \, : \, \dist( \overline{Q}_\beta^{k+j}, X \setminus \widetilde{Q}_\alpha^k) \leq   \theta^{k+j} \rbrace,
$$
together with the corresponding sets of points $S_j(\overline{Q}_\alpha^k) :=  \bigcup_{(k+j,\beta)\in \mathcal{I}_j(\overline{Q}_\alpha^k)} \overline{Q}_\beta^{k+j}.$ We claim the following.

\begin{claim}\label{claimepsilonboundariesforSsets}
There exist constants $C, \eta >0$ depending on $\theta$ and $C_\mu$ so that $\mu(S_j(\overline{Q}_\alpha^k)) \leq C \theta^{j \eta} \mu(Q_\alpha^k)$ for every $j\in \N,$ and every $\alpha, k.$ 
\end{claim}
Before proving this claim, let us explain why this would imply Proposition \ref{furtherpropertiesofdyadicsets}. Assume for the moment that Claim \ref{claimepsilonboundariesforSsets} is true. For any $\lambda \in (0,\theta)$, let $j \in \N \cup \lbrace 0 \rbrace$ be so that $ \theta^{j+1} <\lambda \leq  \theta^j.$ If $x\in \overline{Q}_\alpha^k$ is so that $\dist(x, X \setminus \widetilde{Q}_\alpha^k) \leq \lambda \theta^k,$ and $\beta$ so that $(k+j,\beta) \preceq (k,\alpha)$ and $x\in \overline{Q}_\beta^{k+j},$ then $(k+j,\beta) \in \mathcal{I}_j(\overline{Q}_\alpha^k),$ showing that $E_\lambda(\overline{Q}_\alpha^k) \subset S_j(\overline{Q}_\alpha^k).$ This yields 
$$
\mu(E_\lambda(\overline{Q}_\alpha^k)) \leq \mu( S_j(\overline{Q}_\alpha^k)) \leq C \theta^{j\eta}\mu(Q_\alpha^k) \leq \theta^{-\eta} C \lambda^\eta \mu(Q_\alpha^k).
$$
Let us now prove Claim \ref{claimepsilonboundariesforSsets}. The triangle inequality and the fact that $\overline{Q}_\gamma^l \subset B(z_\gamma^l, A \theta^l)$ for all $(l,\gamma)$ give
$$
\mu( S_j(\overline{Q}_\alpha^k)) \leq \mu \left( \lbrace x\in \overline{Q}_\alpha^k \, : \, \dist(x, X\setminus \widetilde{Q}_\alpha^k) \leq (2A+1) \theta^j \theta^{k} \rbrace \right).
$$
Hence, according to Proposition \ref{qualitativeestimatemeasureboudaries}, we can find $J\in \N$ with the property that
\begin{equation}\label{measureestimateforJfixednatural}
\mu( S_J(\overline{Q}_\alpha^k) ) \leq \frac{1}{2} \mu(Q_\alpha^k), \quad \text{for all }  \,  \alpha \in \mathcal{A}_k,\,  k \in \Z.
\end{equation}
We define new subfamilies of labels obtained from the previous $(\mathcal{I}_j)_j:$ 
$$
\widehat{\mathcal{I}}_1(\overline{Q}_\alpha^k):=\mathcal{I}_J(\overline{Q}_\alpha^k), \quad \widehat{\mathcal{I}}_{n+1}(\overline{Q}_\alpha^k):= \bigcup_{(k+nJ,\beta) \in \widehat{\mathcal{I}}_n(\overline{Q}_\alpha^k)} \mathcal{I}_J(\overline{Q}_\beta^{k+nJ}), \quad n \geq 1.
$$
The corresponding sets of points are defined as $ \widehat{S}_n(\overline{Q}_\alpha^k) :=  \bigcup_{(k+nJ,\beta)\in \widehat{\mathcal{I}}_n(\overline{Q}_\alpha^k)} \overline{Q}_\beta^{k+nJ}.$ 

Let us prove by induction on $n$ that $\mathcal{I}_{nJ}(\overline{Q}_\alpha^k) \subset \widehat{\mathcal{I}}_n(\overline{Q}_\alpha^k).$ If $n=1, $ the inclusion is actually an identity and it follows from the definitions. Now assume that the inclusion holds for some natural $n$ and let us show the corresponding inclusion for $n+1.$ Given $(k+(n+1)J, \beta) \in \mathcal{I}_{(n+1)J}(\overline{Q}_\alpha^k)$ there exists (a unique) $(k+nJ, \gamma)$ with $(k+(n+1)J, \beta) \preceq (k+nJ, \gamma) \preceq (k,\alpha).$ The index $(k+(n+1)J, \beta)$ belongs to $\mathcal{I}_J(\overline{Q}_\gamma^{k+nJ}) $, because $\widetilde{Q}_\gamma^{k+nJ} \subset \widetilde{Q}_\alpha^k$ and
$$
\dist(\overline{Q}_\beta^{k+(n+1)J}, X \setminus \widetilde{Q}_\gamma^{k+nJ}) \leq \dist(\overline{Q}_\beta^{k+(n+1)J}, X \setminus \widetilde{Q}_\alpha^{k}) \leq  \theta^{k+(n+1)J}=\theta^{(k+nJ)+J}.
$$
On the other hand, $\overline{Q}_\beta^{k+(n+1)J} \subset \overline{Q}_\gamma^{k+nJ}$ and then clearly
$$
\dist(\overline{Q}_\gamma^{k+nJ}, X \setminus \widetilde{Q}_\alpha^{k}) \leq \dist(\overline{Q}_\beta^{k+(n+1)J}, X \setminus \widetilde{Q}_\alpha^{k}) \leq  \theta^{k+(n+1)J} \leq  \theta^{k+nJ}.
$$
This tells us that $(k+nJ,\gamma)  \in \mathcal{I}_{nJ}(\overline{Q}_\alpha^k),$ and therefore $(k+nJ,\gamma) \in \widehat{\mathcal{I}}_n(\overline{Q}_\alpha^k)$ by the induction hypothesis. Together with the already proven fact that $(k+(n+1)J, \beta)\in \mathcal{I}_J(\overline{Q}_\gamma^{k+nJ}), $ we may conclude $(k+(n+1)J, \beta) \in \widehat{\mathcal{I}}_{n+1}(\overline{Q}_\alpha^k).$

We summarize the key inclusions here:
\begin{equation}\label{inclusionnewsetsofindices}
\mathcal{I}_{nJ}(\overline{Q}_\alpha^k) \subset \widehat{\mathcal{I}}_n(\overline{Q}_\alpha^k), \quad S_{nJ}(\overline{Q}_\alpha^k) \subset \widehat{S}_{n}(\overline{Q}_\alpha^k), \quad n \geq 2. 
\end{equation}
Observing that $\widehat{S}_{n}(\overline{Q}_\alpha^k) \subset \bigcup_{(k+(n-1)J,\beta) \in \widehat{\mathcal{I}}_{n-1}(\overline{Q}_\alpha^k)} S_J(\overline{Q}_\beta^{k+(n-1)J})$ for $n\geq 2,$ and employing \eqref{measureestimateforJfixednatural}, we can write
\begin{align*}
\mu\left( \widehat{S}_{n}(\overline{Q}_\alpha^k) \right) & \leq \sum_{(k+(n-1)J,\beta) \in \widehat{\mathcal{I}}_{n-1}(\overline{Q}_\alpha^k)} \mu(S_J(\overline{Q}_\beta^{k+(n-1)J})) \leq \frac{1}{2} \sum_{(k+(n-1)J,\beta) \in \widehat{\mathcal{I}}_{n-1}(\overline{Q}_\alpha^k)} \mu( Q_\beta^{k+(n-1)J}) \\
& = \frac{1}{2} \mu  \Big (   \bigcup_{(k+(n-1)J,\beta) \in \widehat{\mathcal{I}}_{n-1}(\overline{Q}_\alpha^k)}  Q_\beta^{k+(n-1)J} \Big ) \leq \frac{1}{2}\mu(\widehat{S}_{n-1}(\overline{Q}_\alpha^k))  \\
& \leq \cdots \leq \frac{1}{2^{n-1}} \mu(\widehat{S}_1(\overline{Q}_\alpha^k)) =\frac{1}{2^{n-1}} \mu(S_J( \overline{Q}_\alpha^k)) \leq \frac{1}{2^{n}} \mu( Q_\alpha^k), \quad \text{for every} \quad n \geq 1.
\end{align*}
Then \eqref{inclusionnewsetsofindices} leads us to $\mu\left( S_{nJ}(\overline{Q}_\alpha^k) \right) \leq \mu\left( \widehat{S}_{n}(\overline{Q}_\alpha^k) \right) \leq 2^{-n} \mu(Q_\alpha^k)$ for each $n\in \N.$ Taking $\eta>0$ so that $\theta^{J \eta}=1/2,$ we obtain Claim \ref{claimepsilonboundariesforSsets} when $j$ is a multiple of $J,$ and with the constant $C=1.$ When $j\in \N$ is arbitrary with $j \geq J,$ we find $m\in \N$ with $mJ \leq j < (m+1)J$ and observe that $S_j(\overline{Q}_\alpha^k) \subset S_{mJ}(\overline{Q}_\alpha^k).$ Then the result for the multiples of $J$ gives the estimate $\mu( S_j(\overline{Q}_\alpha^k)) \leq (\theta^{mJ})^\eta\mu(Q_\alpha^k),$ and this term is bounded above by $\theta^{-J\eta} (\theta^j)^\eta \mu(Q_\alpha^k).$ And when $j <J,$ we simply write
$$
\mu(S_j(\overline{Q}_\alpha^k)) \leq \mu(\overline{Q}_\alpha^k) \leq C(C_\mu) \mu(Q_\alpha^k) \leq C(C_\mu) \theta^{-J\eta} \theta^{j\eta} \mu(Q_\alpha^k).
$$
This completes the proof of Claim \ref{claimepsilonboundariesforSsets}, and hence Proposition \ref{furtherpropertiesofdyadicsets}.
\end{proof}

\section{Weakly porous sets and dyadic weakly porous sets}\label{sectiondefinitionswporosity}

We begin by defining the maximal $E$-free hole for a ball $B=B(x,R)\subset X$:
\begin{equation}\label{definitionmaxholeballs}
h_E(B ) := \sup \lbrace 0<r \leq 2R \: : \: \text{there exists } y\in X \text{ with } B(y,r) \subset B(x,R) \setminus E \rbrace.
\end{equation}
If for a ball $B=B(x,R)$ there is no ball $B(y,r) \subset B(x,R) \setminus E,$ we understand $h_E(B)=0.$ Notice that, even if two balls $B(x,R)$ and $B(x',R')$ define the same set of points, their maximal $E$-free holes $h_E(B(x,R))$ and $h_E(B(x',R'))$ may be distinct.

The upper bound $r \leq 2R$ for the radii in the supremum prevents $h_E(B)$ from being equal to $\infty,$ which may otherwise occur, e.g., when $X$ is a bounded space. This bound is also motivated by the fact that if two balls satisfy $B(y,r) \subset B(x,R)$ for $r $ arbitrary, then the ball $B(y,t)$, with $t=\min\lbrace r,2R \rbrace$, defines the same set of points as $B(y,r)$, and has radius $t$ bounded above by $2R.$

Let us now give the definition of a weakly porous set.

\begin{definition}[Weak porosity]\label{definitionweakporballs}
We say that a set $E \subset X$ is \textit{weakly porous} if there are two constants $c, \delta \in (0,1)$ such that for every open ball $B=B(x,R) \subset X$ there exists a finite collection of mutually disjoint open balls $\lbrace B_i=B(x_i,r_i) \rbrace_{i=1}^N $ contained in $B\setminus E$, with:
\begin{enumerate}[label=(\alph*)]
\item \label{weakporradupperbound} $r_i \leq 2R$ for every $i=1,\ldots,N$; 
\item \label{weakporradlowerbound} $r_i \geq \delta h_E(B)$ for every $i=1,\ldots,N$;
\item \label{weakporcomparablemeasures} $ \sum_{i=1}^N \mu(B_i) \geq c \mu(B). $
\end{enumerate}
\end{definition}

We immediately observe that if $E$ is weakly porous, then $h_E(B)>0$ for every ball $B \subset X.$ Moreover, as a consequence of the Lebesgue differentiation theorem, every weakly porous set $E$ satisfies $\mu(E)=0.$ We also notice that, since the balls we consider are open sets, a set $E$ is weakly porous if and only if its closure $\overline{E}$ is weakly porous, and with the exact same constants $c,\delta.$

It is worth comparing with the (much stronger) classical porous sets: a set $E$ is porous if there is a constant $c \in (0,1)$ so that every ball $B(x,R)$ contains another ball $B'$ with radius at least $cR$ and $B' \cap E = \emptyset.$ Then, by the doubling property of the measure \eqref{Cmudoublingmeasure}, this implies that every porous set is weakly porous. The converse is false; it suffices to equip $X=\R$ with the usual metric and the Lebesgue measure, and consider $E= \mathbb{Z}$ or $E= \lbrace \pm j^{-\gamma} \, : \, j\in \N \rbrace \cup \lbrace 0 \rbrace$, $\gamma>0.$ A detailed study of the $A_1$ properties of distances to these sets can be found in \cite[Theorem 7.1]{ALMV22}.

\begin{remark}
Definition \ref{definitionweakporballs} of weak porosity can be re-formulated in terms of admissible $E$-free balls that are not necessarily mutually disjoint: there are two constants $c, \delta \in (0,1)$ such that for every open ball $B=B(x,R) \subset X$ there exists a finite collection of open balls $\lbrace B_i=B(x_i,r_i) \rbrace_{i=1}^N $ contained in $B\setminus E$, with:
\begin{enumerate} 
\item[(a')] $r_i \leq 2R$ for every $i=1,\ldots,N$; 
\item[(b')] $r_i \geq \delta h_E(B)$ for every $i=1,\ldots,N$;
\item[(c')] $ \mu\left( \bigcup_{i=1}^N B_i \right) \geq c \mu(B). $
\end{enumerate}
Indeed, if $E \subset X,$ $c,\delta \in (0,1)$, $B$ and $\lbrace B_i \rbrace_i$ are as above, the radii of the balls $\lbrace B_i \rbrace_i$ are uniformly bounded by $2R$ and the $5$-covering lemma (see e.g., \cite[p. 5]{BBbook}) can be applied to obtain a subcollection $\lbrace B_j \rbrace_j$ of $\lbrace B_i \rbrace_i$ of mutually disjoint balls with $\bigcup_i B_i \subset \bigcup_j 5 B_j$. Using that $\mu(5B_j) \leq C_\mu^3 \, \mu(B_j),$ it is easy to see that the family $\lbrace B_j \rbrace_j$ satisfies the properties \ref{weakporradupperbound}, \ref{weakporradlowerbound}, \ref{weakporcomparablemeasures} of Definition \ref{definitionweakporballs} with constants $(C_\mu^{-3} c, \delta)$ instead of $(c,\delta).$ 
 
\end{remark}

For the reasons explained in the introduction, we also need a formulation of weak porosity in terms of the dyadic structures from Section \ref{sectionpreliminaries}. From now on, we fix a parameter $0<\theta<1/1000$ and consider the associated dyadic systems $\mathcal{D}=\lbrace \mathcal{D}_t \rbrace_{t=1}^T$ from Proposition \ref{theoremexistencedyadicgrids}.

Given a cube $Q_0$ in a dyadic grid $\mathcal{D}_t,$ the dyadic version of the maximal $E$-free hole for $Q_0$ is defined by the formula
\begin{equation}\label{definitiondyadicmaximalhole}
g_E(Q_0)=\min\lbrace g(Q) \: : \: Q \in \mathcal{D}(Q_0),\:    Q \cap E =\emptyset \rbrace.
\end{equation}
In the trivial case where every $Q\in \mathcal{D}(Q_0)$ intersects $E,$ we understand $g_E(Q_0)=+\infty,$ and $\theta^{g_E(Q_0)}=0.$ Note that $g_E(Q_0) \geq g(Q_0)$ for every $Q_0.$ Equipped with this definition, we introduce the dyadic weakly porous sets in a metric measure space $(X,d,\mu).$

\begin{definition}[Dyadic weak porosity]\label{definitiondyadicweakporosity}
We say that a set $E \subset X$ is \textit{dyadic weakly porous} if there exist constants $c  \in (0,1)$ and $M\in \N$ such that for every cube $Q_0 \in \mathcal{D}$ of $X$, there exists a finite collection of mutually disjoint cubes $Q_1, \ldots, Q_N \in \mathcal{D}(Q_0)$ that are contained in $Q_0\setminus E$, with $g(Q_i) \leq M + g_E(Q_0)$ for every $i=1,\ldots,N$, and $\sum_{i=1}^N \mu(Q_i) \geq c \mu(Q_0)$.
\end{definition}

We next define the doubling properties for the size of maximal $E$-free holes, both in terms of balls and dyadic sets.

\begin{definition}\label{definitionofdoubling}
Let $E \subset X$ be a subset. 
\begin{enumerate}
\item[$\bullet$] We say that \textit{$h_E$ is doubling} if there exists $C>0$ for which
$$
h_E(B(x,2R)) \leq C h_E(B(x,R)), \quad \text{for all} \quad x\in X, \, R>0.
$$
\item[$\bullet$] We say that \textit{$g_E$ is doubling} if there exists $m\in \N$ for which
$$
g_E(Q) \leq m + g_E(Q^*), \quad \text{whenever } \: Q \in \mathcal{D}(Q^*), \, g(Q)=1+g(Q^*), \, Q^* \in \mathcal{D}.
$$
\end{enumerate}

\end{definition}

First notice that $h_E=h_{\overline{E}}$ since the balls we consider are all open.

Also observe that if $g_E$ is doubling with constant $m\in \N$, then for every two cubes $Q, Q^*$ with $Q \in \mathcal{D}(Q^*),$ we have, as a consequence of Lemma \ref{observationsdyadicsets}(iii),
\begin{equation}\label{generalizedgEdoubling}
g_E(Q) \leq m\left( g(Q)-g(Q^*) \right) + g_E(Q^*).
\end{equation}

The rest of this section is devoted to proving that the two formulations of weak porosity and doubling conditions are equivalent. More precisely, a weakly porous set $E$ whose maximal hole function $h_E$ is doubling is dyadic weakly porous with $g_E$ doubling as well, and vice-versa. We begin with the following lemma, which we will invoke several times in the paper.

\begin{lemma}[Large cubes inside balls]\label{lemmafittingcubeinsideball}
Let $B=B(x,r)$ be a ball, and $Q$ a cube in some grid $\mathcal{D}_t$ with $B \subsetneq Q$. Then there exists a cube $\widehat{Q} \in \mathcal{D}(Q)$ with $x\in \widehat{Q} \subset B$ and $r\leq   \frac{2A}{\theta}  \theta^{g(\widehat{Q})}.$
\end{lemma}
\begin{proof}
Consider the set of integers
\begin{equation}\label{minimumgenerationfittinglemma}
\lbrace g(Q') \, : \, Q' \in \mathcal{D}(Q),\, x\in Q' \subset B\rbrace.
\end{equation}
Let us first explain why this set is non-empty. The cube $Q$ can be written as the union of cubes of $\mathcal{D}(Q)$, all of them with generation equal to $l,$ where $l \geq g(Q)$ is so that $2A \theta^l<r.$ Because $B(x,r)$ is contained in $Q,$ one of these cubes, say $Q'$, must contain the point $x.$ Denoting by $z'$ the reference point of $Q',$ the ball $B(z',A \theta^{l})$ contains $Q'$ and then every point $y \in Q'$ satisfies $d(x,y)< 2A \theta^{l}<r.$ This implies that $x\in Q' \subset B(x,r),$ showing that the set in \eqref{minimumgenerationfittinglemma} is nonempty. 
Also, this set is bounded below by $g(Q)\in \Z,$ by the definition of $\mathcal{D}(Q)$. Therefore, the minimum of the set \eqref{minimumgenerationfittinglemma} is well-defined and there exists a cube $\widehat{Q} \in \mathcal{D}(Q)$ with $x\in \widehat{Q} \subset B$ with $g(\widehat{Q})$ equal to the minimum of \eqref{minimumgenerationfittinglemma}. Since $B \subsetneq Q,$ we must have $g(\widehat{Q}) > g(Q),$ as otherwise the equality $g(\widehat{Q})=g(Q)$ would lead us to $Q= \widehat{Q}$ by virtue of Lemma \ref{observationsdyadicsets}, and thus implying $B=Q,$ a contradiction. Consequently, $g(\widehat{Q})$ is strictly larger than $g(Q)$ and hence there exists another cube $Q^* \in \mathcal{D}(Q)$ with $\widehat{Q} \subset Q^*$ and $g(Q^*)=g(\widehat{Q})-1.$ By the minimality of $g(\widehat{Q}),$ the cube $Q^*$ cannot be contained in $B,$ and then $d(p,x) \geq r$ for some $p\in Q^*.$ Denoting by $z^*$ the reference point of $Q^*,$ the ball $B(z^*,A \theta^{g(Q^*)})$ contains $Q^*$ and, because $x\in Q^*$, we have
$$
r \leq d(p,x) \leq \diam(Q^*) \leq \diam \left( B(z^*,A \theta^{g(Q^*)}) \right) \leq 2A \theta^{g(Q^*)} = \frac{2A}{\theta} \theta^{g(\widehat{Q})}.
$$
This proves the assertion.
\end{proof}

\begin{remark}\label{remarkdyadicwpclosure}
Among other properties, Lemma \ref{lemmafittingcubeinsideball} permits to show that a set $E$ is dyadic weakly porous if and only if its closure $\overline{E}$ is, though possibly with different parameters $M,c.$

 Indeed, assume first that $E$ is dyadic weakly porous with constants $M\in \N$ and $c\in(0,1).$ If $Q_0 \in \mathcal{D}$ is a cube and $Q \in \mathcal{D}(Q_0)$ with $Q \subset Q_0 \setminus E,$ then $Q$ contains the open ball $B:=B(z, a \theta^{g(Q)}),$ where we denote by $z$ the reference point of $Q.$ If $Q=B$ as sets of points, then $Q$ is an open set and thus we have $Q \subset Q_0 \setminus \overline{E}.$ If $B \subsetneq Q,$ Lemma \ref{lemmafittingcubeinsideball} gives some $\widehat{Q} \in \mathcal{D}(Q) \subset \mathcal{D}(Q_0) $ so that $\widehat{Q} \subset B$ and $a\theta^{g(Q)} \leq \frac{2A}{\theta} \theta^{g(\widehat{Q})}$. We deduce $g(\widehat{Q}) \leq n + g(Q)$ for some natural $n$ depending only on $\theta,$ and also that $\widehat{Q} \subset Q_0 \setminus \overline{E}$ because $B$ is open. Also, denoting by $\widehat{z}$ the reference point of $\widehat{Q},$ the inclusions
 $$
 B(\widehat{z}, a \theta^{g(\widehat{Q})}) \subset \widehat{Q} \subset Q \subset B(z, A \theta^{g(Q)})
 $$
together with \eqref{Cmudoublingmeasure}, imply $\mu(\widehat{Q}) \geq c(\theta,C_\mu) \mu(Q).$ We have shown that given cubes $Q_1, \ldots, Q_N \in \mathcal{D}(Q_0)$ that are admissible for the definition of dyadic weak porosity with constants $M$,$c$, one can find pairwise disjoint cubes $\widehat{Q}_1, \ldots, \widehat{Q}_N \in \mathcal{D}(Q_0)$ so that $\widehat{Q}_i \subset Q_0 \setminus \overline{E}$, 
$$
g(\widehat{Q}_i) \leq n+ g(Q_i) \leq n+M +g_E(Q_0) \leq n+M  +g_{\overline{E}}(Q_0),
$$
and $\sum_{i=1}^N \mu(\widehat{Q}_i) \geq c(\theta,C_\mu) \sum_{i=1}^N\mu(Q_i) \geq c(\theta,C_\mu) \, c \mu(Q_0).$

Conversely, if $\overline{E}$ is dyadic weakly porous with parameters $M\in \N,$ $c\in(0,1),$ one can use the same argument as above involving Lemma \ref{lemmafittingcubeinsideball} to show that $g_{\overline{E}}(Q_0) \leq n + g_E(Q_0)$ for $n=n(\theta) \in \N$. Thus, every collection of admissible cubes $Q_1, \ldots , Q_N \in \mathcal{D}(Q_0)$ for the definition of dyadic weak porosity of $\overline{E}$ on $Q_0$ is also admissible for dyadic weak porosity of $E$, with parameters $M+n\in \N$ and $c\in(0,1).$ 
\end{remark}

\begin{lemma}\label{assumingdoublingdyadicwpimplieswp}
Assume that $g_E$ is doubling. Then the following holds.
\begin{enumerate}
[label=\normalfont{(\roman*)}]
\item $h_E$ is doubling.  
\item If, in addition, $E$ is dyadic weakly porous, then $E$ is weakly porous.
\end{enumerate}
\end{lemma}
 \begin{proof} \textbf{(i)} Fix the balls $B_0=B(x_0,R_0),$ $2B_0=B(x_0, 2 R_0) .$ If $B(x_0,R_0) \cap E = \emptyset$, then 
 $$
 h_E(B(x_0,2R_0)) \leq 4R_0 \leq 4 h_E(B(x_0,R_0)),
 $$ 
and the doubling property for $h_E$ holds.

Now suppose that $B(x_0,R_0) \cap E \neq \emptyset.$ We may also assume that $B_0 \subsetneq 2B_0$ because in the case where $B_0=2B_0$ (as sets of points), for every ball $B(y,r) \subset 2B_0 \setminus E$ with $0<r\leq 4R_0$, the concentric ball $B(y,r/2)$ is contained in $B_0 \setminus E$ and $0<r/2\leq 2R_0,$ from which we would obtain $h_E(2B_0) \leq 2 h_E(B_0).$

By Proposition \ref{theoremexistencedyadicgrids}, there exist $t$ and $Q_0^* \in \mathcal{D}_t$ with $2B_0 \subset Q_0^*,$ and $\theta^{g(Q_0^*)+2} < 2 R_0 \leq \theta^{g(Q_0^*)+1}.$ 
If $B=B(y,r) \subset 2B_0 \setminus E$, then $B \subsetneq 2B_0^*$ since $B_0 \cap E \neq \emptyset,$ and thereby $B \subsetneq Q_0^*.$ Lemma \ref{lemmafittingcubeinsideball} provides us with a cube $Q \in \mathcal{D}(Q_0^*)$ with $Q \subset B$ and $r \leq \frac{2A}{\theta} \theta^{g(Q)}.$ Because $Q \cap E= \emptyset$, this implies
$$
h_E(2 B_0) \leq \frac{2A}{\theta} \theta^{g_E(Q_0^*)}.
$$
Since $B_0 \subsetneq 2B_0 \subset Q_0^*$ we can apply Lemma \ref{lemmafittingcubeinsideball} to obtain a cube $Q_0 \in \mathcal{D}(Q_0^*)$ with $Q_0 \subset B_0$ and $R_0\leq \frac{2A}{\theta} \theta^{g(Q_0)}$. Then the inequality $\theta^{g(Q_0^*)+2} < 2R_0$ tells us that $g(Q_0)$ and $g(Q_0^*)$ are comparable, meaning that $g(Q_0)\leq n + g(Q_0^*)$ for some $n=n(\theta) \in \N.$ Using the doubling property of $g_E$ (more precisely, the estimates \eqref{generalizedgEdoubling}) for the cubes $Q_0$ and $Q_0^*$, we get
$$
\theta^{g_E(Q_0^*)}   \leq C(\theta,m) \theta^{g_E(Q_0)};
$$
where $m \in \N$ is as in Definition \ref{definitionofdoubling}. Given a cube $Q \in \mathcal{D}(Q_0)$ with reference point $z$ and $Q \cap E= \emptyset,$ the $B(z, a \theta^{g(Q)})$ is contained in $B_0 \setminus E,$ with 
$$
  \theta^{g(Q)}  \leq   \theta^{g(Q_0)} \leq    \theta^{g(Q_0^*)} <  \frac{2R_0}{\theta^2} .
$$
Thus, the $\theta^2$-contraction of $B(z, a \theta^{g(Q)})$ is a ball inside $B_0 \setminus E$ that is admissible for the definition of $h_E(B_0),$ and so 
$$
a \theta^2 \theta^{g(Q)} \leq h_E(B_0).
$$
Therefore $\theta^{g_E(Q_0)} \leq \frac{1}{a \theta^2} h_E(B_0),$ and gathering all the estimates, we conclude $h_E(2 B_0) \leq C(m,\theta) h_E(B_0).$ This shows that $h_E$ is doubling.

\medskip

\textbf{(ii)} Let $B_0=B(x_0,R_0)$ be a fixed ball with $B_0 \cap E \neq \emptyset$. By Proposition \ref{theoremexistencedyadicgrids}, we can find $Q_0\in \mathcal{D}$ with $B_0 \subset Q_0$ and $\theta^{g(Q_0)+2} < R_0 \leq \theta^{g(Q_0)+1}.$

Let us fit a cube $\widehat{Q}_0 \in \mathcal{D}(Q_0)$ in $B_0$ of \textit{length} $\theta^{g(\widehat{Q}_0)}$ comparable to $R_0.$ If $B_0=Q_0$ (as sets of points), then we simply take $\widehat{Q}_0$ as $Q_0.$ If $B_0 \subsetneq Q_0,$ by Lemma \ref{lemmafittingcubeinsideball} we get a cube $\widehat{Q}_0 \in \mathcal{D}(Q_0)$ with $Q_0 \subset B_0$ and $R_0 \leq \frac{2A}{\theta} \theta^{g(\widehat{Q}_0)}.$ In any case, we have the relation
\begin{equation}\label{proportionradiidyadicsize}
  \theta^{g(Q_0)+2}< R_0 \leq    \frac{2A}{\theta} \theta^{g(\widehat{Q}_0)}.
\end{equation}
Also, denoting the reference points of $Q_0$, $\widehat{Q}_0$ respectively by $z_0,$ $\widehat{z}_0$ we have the inclusions,
$$
B(\widehat{z}_0,a \theta^{g({\widehat{Q}_0})}) \subset \widehat{Q}_0 \subset B_0 \subset Q_0 \subset B(z_0,A \theta^{g(Q_0)}) \subset B(\widehat{z}_0,3A \theta^{g(Q_0)})\subset B(\widehat{z}_0,6A^2 \theta^{-3} \theta^{g({\widehat{Q}_0})});
$$
from which the doubling condition \eqref{Cmudoublingmeasure} for $\mu$ leads us to
\begin{equation}\label{proportiondyadicmeasure}
\mu(\widehat{Q}_0)\geq   C_\mu^{-1-\left\lfloor \log_2(6A^2a^{-1} \theta^{-3}) \right\rfloor} \mu(B_0).
\end{equation}

Because $E$ is dyadic weakly porous, there exist $E$-free and pairwise disjoint cubes $Q_1, \ldots, Q_N$ of $\mathcal{D}(\widehat{Q}_0)$ with 
\begin{equation}\label{displayedpropertiesauxiliarcubes}
g(Q_i) \leq M + g_E(\widehat{Q}_0), \:\: i=1,\ldots,N, \quad \text{and} \quad \sum_{i=1}^N \mu(Q_i) \geq c \mu(\widehat{Q}_0).
\end{equation}

If $z_i$ denotes the reference point of $Q_i$ define the ball $B_i=B ( z_i, r_i ),$ with $r_i:=a \theta^2 \theta^{g(Q_i)}.$ By Proposition \ref{theoremexistencedyadicgrids}, each $B_i$ is contained in $Q_i,$ implying that the collection $\lbrace B_i \rbrace_i$ is disjoint, and that $B_i \subset \widehat{Q}_0 \setminus E \subset B_0 \setminus E.$ Let us verify the properties \ref{weakporradupperbound}, \ref{weakporradlowerbound} and \ref{weakporcomparablemeasures} from Definition \ref{definitionweakporballs} for the collection $\lbrace B_i \rbrace_i$.

Notice that $r_i \leq a  \theta^{g(Q_0)+2} < R_0,$ and so property \ref{weakporradupperbound} holds.

To prove \ref{weakporradlowerbound}, observe that for each ball $B=B(p,r) \subset B_0 \setminus E$ with $0<r \leq 2 R_0,$ we have $B \subsetneq B_0 \subset Q_0$, and Lemma \ref{lemmafittingcubeinsideball} can be applied to obtain $Q \in \mathcal{D}(Q_0) $ with $Q \subset B$ and $r\leq  \frac{2A}{\theta}\theta^{g(Q)}.$ 
Since $Q \cap E = \emptyset,$ this proves $h_E(B_0) \leq \frac{2A}{\theta} \theta^{g_E(Q_0)}.$ The radius $r_i$ of the ball $B_i$ satisfies $r_i \geq a \theta^{M+2} \theta^{g_E(\widehat{Q}_0)},$ and using first that $g_E$ is doubling and then \eqref{proportionradiidyadicsize}, we get
$$
r_i \geq a \theta^{M+2} \theta^{g_E(\widehat{Q}_0)} \geq a \theta^{M+2} \theta^{m (g(\widehat{Q}_0)-g(Q_0))} \theta^{g_E( Q_0)}  \geq a \theta^{M+2}  \left( \frac{\theta^3}{2A} \right)^m \theta^{g_E( Q_0)} \geq \delta h_E(B_0);
$$
where $\delta= a\theta^{3}    \frac{\theta^{M+3m}}{(2A)^{m+1}}.$

As concerns property \ref{weakporcomparablemeasures}, we use that $B_i=B(z_i, a \theta^2 \theta^{g(Q_i)}) \subset  Q_i \subset B ( z_i, A \theta^{g(Q_i)}),$ and the doubling condition \eqref{Cmudoublingmeasure} on the measure $\mu$ to obtain
$$
\mu(Q_i) \leq C_\mu^{1+\left\lfloor \log_2( a^{-1} \theta^{-2 }A) \right\rfloor} \mu(B_i).
$$
Combining this estimate with \eqref{displayedpropertiesauxiliarcubes} and \eqref{proportiondyadicmeasure}, we obtain
$$
\sum_{i=1}^N \mu(B_i) \geq  c_0(\theta,C_\mu) \, c \,  \mu(B_0).
$$
for some constant $c_0(\theta,C_\mu)>0$ depending only on $\theta$ and $C_\mu.$ 
\end{proof}

Now we show that the formulation with balls implies the dyadic formulation.

\begin{lemma} \label{assumingdoublingwpimpliesdyadicwp}
Assume that $h_E$ is doubling with constant $C.$ Then the following holds. 
\begin{enumerate}[label=\normalfont{(\roman*)}]
\item $g_E$ is doubling with constant $m=m(\theta,C) \in \N.$
\item If, in addition, $E$ is weakly porous with constants $(c,\delta)$, then $E$ is dyadic weakly porous with constants $c'=c'(\theta,C_\mu,c)$ and $M=M(\theta,\delta,C).$
\end{enumerate}
\end{lemma}
\begin{proof} \textbf{(i)} Let $Q_0,$ $Q_0^*$ be two cubes in a common grid, with $Q _0 \in \mathcal{D}( Q_0^*)$ and $g(Q_0^*)=g(Q_0)-1.$ If $z_0$, $z_0^*$ are respectively the reference points of $Q_0$ and $Q_0^*,$ define $B_0=B(z_0,a\theta^{g(Q_0)})$ and $B_0^*=B(z_0^*, A \theta^{g(Q_0^*)}).$ By Proposition \ref{theoremexistencedyadicgrids}, $B_0 \subset Q_0$ and $Q_0^* \subset B_0^*.$ Also observe that $B_0^* \subset B(z_0,\frac{2A}{\theta} \theta^{g(Q_0)}),$ and this implies, by virtue of the doubling condition for $h_E$,
\begin{equation}\label{lemmadoublingimpliesdyadicestimateballs}
h_E(B_0^*) \leq C^{1+\lfloor \log_2\left(\frac{2A}{a \theta}\right) \rfloor } h_E(B_0);
\end{equation}
where $C>0$ is that of Definition \ref{definitionofdoubling}.

Let us now verify that $ \theta^{g_E(Q_0^*)} \lesssim h_E(B_0^*)$ and $h_E(B_0) \lesssim \theta^{g_E(Q_0)}.$

If $Q \in \mathcal{D}(Q_0^*)$ is such that $Q \cap E =\emptyset$ and has reference point $z,$ then the ball $B(z, a \theta^{g(Q)})$ is contained in $B_0^* \setminus E$ and its radius $a \theta^{g(Q)}$ is smaller than $A \theta^{g(Q_0^*)}$. By the definition of $h_E(B_0^*)$ and $g_E(Q_0^*)$ \eqref{definitionmaxholeballs}--\eqref{definitiondyadicmaximalhole}, we get $a\theta^{g_E(Q_0^*)} \leq h_E(B_0^*).$ In the case where $Q \cap E \neq \emptyset$ for all $Q\in \mathcal{D}(Q_0^*),$ we have $g_E(Q_0^*)=+\infty$ and $\theta^{g_E(Q_0^*)}=0,$ and the estimate $\theta^{g_E(Q_0^*)} \lesssim h_E(B_0^*)$ is true as well.

For the other estimate $h_E(B_0) \lesssim \theta^{g_E(Q_0)}$, suppose first that $Q_0 \cap E = \emptyset.$ Then $g_E(Q_0)=g(Q_0)$ and $h_E(B_0) \leq 2a \theta^{g(Q_0)}=2a \theta^{g_E(Q_0)}$. Now, if $Q_0 \cap E \neq \emptyset,$ let $B=B(x,r)$ be a ball contained in $B_0\setminus E$ (If such ball does not exist, then $h_E(B_0)=0$ and the desired estimate follows immediately). Because of the strict inclusion $B \subsetneq Q_0$ (as $B$ does not intersect $E$), we can apply Lemma \ref{lemmafittingcubeinsideball} in order to obtain a cube $Q \in \mathcal{D}(Q_0)$ with $Q \subset B$ and $r \leq \frac{2A}{\theta} \theta^{g(Q)}.$ Because $Q \cap E = \emptyset,$ we may conclude $h_E(B_0) \leq \frac{2A}{\theta} \theta^{g_E(Q_0)}.$

These estimates together with \eqref{lemmadoublingimpliesdyadicestimateballs} give $\theta^{g_E(Q_0^*)} \leq \kappa(\theta,C) \theta^{g_E(Q_0)},$ showing that $g_E$ is doubling.

\medskip

\textbf{(ii)} Let $Q_0$ a cube of some dyadic grid $\mathcal{D}_t,$ with reference point $z_0.$ We may assume that $Q_0 \cap E \neq \emptyset,$ as otherwise the dyadic weak porosity property holds for $Q_0.$ Defining $B_0:=B(z_0, a\theta^{g(Q_0)})$ and $B_0^*:= B(z_0, A\theta^{g(Q_0)}),$ we have the inclusions $B_0 \subset Q_0 \subset B_0^*,$ and so $\mu(B_0) \geq \kappa(C_\mu) \mu(Q_0)$ by virtue of the doubling condition for $\mu.$

We first claim that $\theta^{g_E(Q_0)} \leq \kappa(\theta,C) h_E(B_0)$ where $C$ is the constant of Definition \ref{definitionofdoubling} for $h_E.$ Indeed, if $Q\in \mathcal{D}(Q_0) $ with $Q \cap E =\emptyset,$ and $Q$ has reference point $z,$ then the ball $B(z,a \theta^{g(Q)})$ is contained in $ B_0^* \setminus E,$ and obviously $ a \theta^{g(Q)} \leq A \theta^{g(Q_0)}.$ This argument implies that $a \theta^{g_E(Q_0)} \leq h_E(  B_0^*).$ Because $h_E$ is doubling, we have
\begin{equation}\label{referenceballofacubehavecomparableholes}
a \theta^{g_E(Q_0)} \leq h_E(  B_0^*) \leq \kappa(C) h_E(B_0),
\end{equation}
which proves our claim.

Now, we apply the weak porosity condition for $E$ to the ball $B_0$, obtaining disjoint balls $\lbrace B_i=B(x_i,r_i) \rbrace_{i=1}^N$ contained in $B_0 \setminus E$ as in Definition \ref{definitionweakporballs}, with constants $c, \delta \in (0,1).$ Since $B_i \subsetneq Q_0,$ Lemma \ref{lemmafittingcubeinsideball} gives a cube $Q_i \in \mathcal{D}(Q_0)$ with $Q_i \subset B_i$ and $r_i \leq \frac{2A}{\theta} \theta^{g(Q_i)}.$ In particular, $Q_i \cap E = \emptyset$ for each $i,$ and the cubes $\lbrace Q_i \rbrace_i$ are mutually disjoint.

On the other hand, $2 A \theta^{-1} \theta^{g(Q_i)} \geq r_i \geq \delta h_E(B_0)$, and by \eqref{referenceballofacubehavecomparableholes} this yields $\theta^{g(Q_i)} \geq \kappa(\theta,\delta,C) \theta^{g_E(Q_0)} .$ Thus, there exists $ M(\delta,C, \theta) \in \N$ so that
$$
g(Q_i) \leq M(\delta,C, \theta)+  g_E(Q_0), \quad \text{for each} \quad i=1,\ldots,N.
$$

Also, if $z_i$ denotes the reference point of $Q_i,$ the following inclusions hold
$$
B(z_i, a \theta^{g(Q_i)}) \subset Q_i \subset B_i \subset B(z_i, 4A \theta^{-1} \theta^{g(Q_i)}).
$$
Then, the doubling property of $\mu$ implies that $  \mu(B_i) \leq C_\mu^{1+ \left\lfloor \log_2(4A \theta^{-1} a^{-1}) \right\rfloor} \mu(Q_i).$ So, the cubes $\lbrace Q_i \rbrace_i$ are disjoint $E$-free cubes in $\mathcal{D}(Q_0)$ with
$$
\kappa(\theta,C_\mu)\sum_{i=1}^N \mu(Q_i) \geq \sum_{i=1}^N \mu(B_i) \geq c \, \mu(B_0) \geq  \kappa(C_\mu) \, c \, \mu(Q_0),
$$
and then we get $\sum_{i=1}^N \mu(Q_i) \geq \kappa(\theta,C_\mu) \, c \, \mu(Q_0),$ for certain constant $\kappa(\theta,C_\mu)>0.$ 
\end{proof}

\section{Implications of the $A_1$-condition}\label{sectionecessitymaintheorem}

The purpose of this section is to prove the implication \ref{maintheoremconditionwp}$\implies$\ref{maintheoremconditionA1} of Theorem \ref{maintheoremallspaces}. In order to do so, we start by proving that if $\dist( \cdot, E)^{-\alpha} \in A_1(X)$ for some $\alpha>0,$ then the set $E$ is dyadic weakly porous, and $g_E$ is doubling; see Definitions \ref{definitiondyadicweakporosity} and \ref{definitionofdoubling}. We will use the fact that a weight $w\in A_1(X)$ (see \eqref{d.A1}) if and only if the same $A_1$-estimates hold when replacing all balls with all cubes $Q\in \mathcal{D},$ i.e, there exists a constant $C>\infty$ so that
$$
\intav_Q w \ud\mu \leq C \essinf_{Q} w, \quad \text{for all} \quad Q\in \mathcal{D}.
$$
This can be easily proved using properties (3) and (4) of Proposition \ref{theoremexistencedyadicgrids}, together with the doubling condition \eqref{Cmudoublingmeasure} for the measure $\mu.$ 

\begin{lemma}\label{dyadicimplicationsA1}
If $E \subset X$ is a subset, and $w= \dist( \cdot, E)^{-\alpha} \in A_1(X)$ with $\alpha>0,$ the following holds.
\begin{enumerate}[label=\normalfont{(\roman*)}]
\item For every $c\in (0,1),$ there exists $M=M(c,\alpha,[w]_1,\theta,C_\mu) \in \N$ such that, for every dyadic cube $Q_0 \subset X,$ there are mutually disjoint cubes $Q_1,\ldots,Q_N \in \mathcal{D}(Q_0)$ containing no point of $E,$ such that $g(Q_i) \leq M + g_E(Q_0)$ for every $i,$ and $\sum_{i=1}^N \mu(Q_i) \geq c \mu(Q_0).$ In  particular, $E$ is dyadic weakly porous.
\item There exists $m=m(\alpha,[w]_1,\theta,C_\mu) \in \N$ such that, for every two cubes $Q_0, Q_0^*$ in a same grid with $Q_0 \in \mathcal{D}(Q_0^*)$ and $g(Q_0)=g(Q_0^*)+1,$ we have
$$
g_E(Q_0) \leq m  + g_E(Q_0^*).
$$ 
In other words, $g_E$ is doubling with constant $m$.
\end{enumerate}

\end{lemma}

\begin{proof} First of all, notice that $\mu(E)=0$ since $w:=\dist(\cdot,E)^{-\alpha} $ is locally integrable in $X.$

\smallskip

\item[$\textbf{(i)}$] Thanks to Remark \ref{remarkdyadicwpclosure}, we may assume that $E$ is closed. Fix a cube $Q_0$ in some system $\mathcal{D}_t$ and assume that $Q_0 \cap E \neq \emptyset.$ Let $M\in \N$ be a number to be chosen later, and consider
$$
\mathcal{A}_M= \bigcup \lbrace Q \in \mathcal{D}(Q_0) \, : \, Q \cap E = \emptyset \, \text{ and }  \, g(Q) \leq M+ g_E(Q_0) \rbrace.
$$
Now, for $x\in \left( Q_0 \setminus E \right) \setminus \mathcal{A}_M,$ define
$$
g_x=\min \lbrace g(Q) \, : \, Q \in \mathcal{D}(Q_0), \,  x\in Q,  \text{ and }  \, Q \cap E =\emptyset \rbrace.
$$
Because $E$ is closed, we have that $\dist(x,E)>0,$ and then an argument similar to the one in the beginning of the proof of Lemma \ref{lemmafittingcubeinsideball} shows that $g_x$ is well-defined. 
 Let $Q_x \in \mathcal{D}(Q_0)$ be so that $x\in Q_x$, $Q_x \cap E= \emptyset,$ and $g(Q_x)=g_x.$ Observe that $g_x >M + g_E(Q_0)$ since $x\notin \mathcal{A}_M.$ Also, if $Q_x^*$ is a cube in $\mathcal{D}(Q_0)$ with $Q_x \in \mathcal{D}( Q_x^*)$ and $g(Q_x^*)= g(Q_x)-1$ (such cube exists because $Q_0 \cap E \neq \emptyset$), we must have $Q_x^* \cap E \neq \emptyset.$ Denoting by $z^*$ the reference point of $Q_x^*$, the inclusion $Q_x^* \subset B \left( z^*, A \theta^{g(Q_x^*)} \right)$ holds and thus
$$
\dist(x,E) \leq \diam\left( B \left( z^*, A \theta^{g(Q_x^*)} \right) \right) \leq   \frac{2A}{\theta}   \theta^{g_x} \leq   \frac{2A}{\theta}   \theta^{M} \theta^{g_E(Q_0)}.
$$
Employing this inequality, we have
\begin{align}\label{estimateessinfmaximaldyadichole1}
\mu \left(  Q_0  \setminus \mathcal{A}_M \right) & \leq \left(   \frac{2A}{\theta}   \theta^{M} \theta^{g_E(Q_0)} \right)^{\alpha} \int_{\left( Q_0 \setminus E \right) \setminus \mathcal{A}_M } \dist(x,E)^{-\alpha} \ud\mu(x) \nonumber \\
& \leq \left(   \frac{2A}{\theta}   \theta^{M} \theta^{g_E(Q_0)} \right)^{\alpha} \mu (Q_0) [w]_{A_1} \essinf_{Q_0} \dist(\cdot,E)^{-\alpha}.
\end{align}
Now, let $Q_{E} \in \mathcal{D}(Q_0)$ with $Q_E \cap E = \emptyset$ and $g(Q_E)=g_E(Q_0)$, and let $z$ be its reference point. Because $B \left( z, a \theta^{g_E(Q_0)} \right) \subset Q_E,$ we have $d(z,E) > a \theta^{g_E(Q_0)},$ and therefore, continuing from \eqref{estimateessinfmaximaldyadichole1} we derive
\begin{align}\label{estimateessinfmaximaldyadichole2}
\mu \left( Q_0  \setminus \mathcal{A}_M \right) & \leq \left(   \frac{2A}{\theta}   \theta^{M} \theta^{g_E(Q_0)} \right)^{\alpha} \mu (Q_0) [w]_{A_1} \dist(z,E)^{-\alpha} \\
& \leq \left(   \frac{2A}{\theta}   \theta^{M} \theta^{g_E(Q_0)} \right)^{\alpha} \mu (Q_0) [w]_{A_1} \left(   a     \theta^{g_E(Q_0)} \right)^{-\alpha} = \left(   \frac{2A}{a \theta}   \theta^{M}  \right)^{\alpha}  [w]_{A_1}  \mu (Q_0). \nonumber 
\end{align}
For every $c\in (0,1),$ we can find $M=M(\theta, [w]_{A_1}, \alpha,c) \in \N$ large enough so that $\left(   \frac{2A}{a \theta}   \theta^{M}  \right)^{\alpha} [w]_{A_1}<1-c,$ obtaining
$$
\mu(\mathcal{A}_M) \geq c \mu(Q_0).
$$
The family of cubes $\widehat{\mathcal{F}}_M(Q_0)=\lbrace Q \in \mathcal{D}(Q_0) \, : \, Q \cap E = \emptyset \, \text{ and }  \, g(Q) \leq M+ g_E(Q_0) \rbrace$ is not necessarily disjoint, but we can make it disjoint by taking a \textit{maximal} subcollection of $\widehat{\mathcal{F}}_M(Q_0).$\footnote{This type of collections of \textit{maximal} cubes will be defined and used in Section \ref{sectiondiscreteinequalityA1} below.} Indeed, given $Q\in \widehat{\mathcal{F}}_M(Q_0),$ let $Q^\sharp \in \widehat{\mathcal{F}}_M(Q_0)$ be a cube with $Q \subset Q^\sharp$ such that
$$
g(Q^\sharp)=\min\lbrace g(Q') \, : \, Q \subset Q' \text{ and } Q'\in \widehat{\mathcal{F}}_M(Q_0) \rbrace.
$$
Notice that $Q \in \mathcal{D}(Q^\sharp)$ for every $Q\in \widehat{\mathcal{F}}_M(Q_0).$ If we define $\mathcal{F}_M(Q_0)=\lbrace Q^\sharp \, : \, Q \in \widehat{\mathcal{F}}_M(Q_0) \rbrace,$ for any two cubes $P, P' \in \mathcal{F}_M(Q_0)$ with non-empty intersection, one has that $P$ and $P'$ are the same cube. Indeed, if $P \subsetneq P',$ then $g(P')<g(P)$ by virtue of Lemma \ref{observationsdyadicsets}. But $P=Q^\sharp$ for some $Q \in \widehat{\mathcal{F}}_M$ and since $P'$ also contains $Q$, the minimality of $g(P)$ would lead us to $g(P') \geq g(P),$ a contradiction. Similarly, one can prove that $P' \subsetneq P$ is impossible, and the properties of the dyadic cubes tell us that $P=P'$ as sets. Moreover, by the minimality of their generations, one has $g(P)=g(P').$ According to Lemma \ref{observationsdyadicsets}, $P$ and $P'$ define the exact same cube. Therefore, the cubes of the family $\mathcal{F}_M(Q_0)$ are mutually disjoint. 

Clearly $\mathcal{A}_M= \bigcup \widehat{\mathcal{F}}_M(Q_0)=\bigcup  \mathcal{F}_M(Q_0).$ Also, every cube $Q \in \mathcal{F}_M $ contains the ball $B_Q:=B\left( z_Q, r   \right),$ with $r=a \theta^{M+g_E(Q_0)}$ and $z_Q$ being the reference point of $Q.$ The balls of the family $\lbrace B_Q \rbrace_{Q \in   \mathcal{F}_M(Q_0)}$ are disjoint and contained in $Q_0$. 
Because $Q_0$ is in turn contained in a ball $B_0$ of radius $R_0= A \theta^{g(Q_0)},$ then $B_0 \subset B(z_Q, 2R_0)$ for each $Q \in \mathcal{F}_M(Q_0)$. The doubling condition \eqref{Cmudoublingmeasure} for $\mu$ gives the estimates $\mu(B_0) \leq C(C_\mu, R_0/r) \mu (B_Q)$ for all $Q \in \mathcal{F}_M(Q_0).$ Thus $  \mathcal{F}_M(Q_0)$ must be a finite collection. This concludes the proof of the dyadic weak porosity of $E,$ since
$$
 \sum_{Q\in  \mathcal{F}_M(Q_0)} \mu(Q) =\mu\left( \bigcup\mathcal{F}_M(Q_0)  \right)= \mu(\mathcal{A}_M)   \geq c \mu(Q_0).
$$

\medskip

\item[$\textbf{(ii)}$] Let $Q_0, Q_0^*\in \mathcal{D}$ be cubes in the same grid with $Q_0 \in \mathcal{D}(Q_0^*)$ with $g(Q_0^*)=g(Q_0)-1.$ In the case where $Q_0 \cap E = \emptyset$, then $g_E(Q_0)=g(Q_0)$ and, because we always have the estimate $g(Q_0^*) \leq g_E(Q_0^*),$ we can write
$$
g_E(Q_0)  = g(Q_0) \leq g(Q_0)-g(Q_0^*) + g_E(Q_0^*) = 1 + g_E(Q_0^*)  .
$$

Suppose now that $Q_0 \cap E \neq \emptyset$. If $x\in Q_0$, there exists $Q \in \mathcal{D}(Q_0)$ with $x\in Q $ and $g(Q)=g_E(Q_0).$ Also, by the assumption, there is $Q^*  \in\mathcal{D}(Q_0)$ so that $Q \in \mathcal{D}(Q^*)$ and $g(Q^*)=g(Q)-1.$ Due to the minimality of $g_E(Q_0),$ the cube $Q^*$ must intersects $E,$ and so 
$$
d(x,E) \leq \diam(Q^*) \leq \diam \left( B(z^*,A \theta^{g(Q^*)}) \right) \leq 2A \theta^{g(Q)-1} = \frac{2A}{\theta} \theta^{g_E(Q_0)};
$$ 
where $z^*$ denotes the reference point of $Q^*.$ This argument permits to write
\begin{align}\label{estimateessinfdyadicparentcondition}
\mu(Q_0) \left( \frac{2A}{\theta} \theta^{g_E(Q_0)} \right)^{-\alpha} & \leq \int_{Q_0} \dist(x,E)^{-\alpha}\ud\mu(x)\leq \int_{Q_0^*} \dist(x,E)^{-\alpha}\ud\mu(x) \nonumber \\
& \leq \mu(Q_0^*) [w]_{A_1} \essinf_{Q_0^*} \dist(\cdot, E)^{-\alpha} \leq \mu (Q_0^*) [w]_{A_1} \left(   a     \theta^{g_E(Q_0^*)} \right)^{-\alpha};
\end{align}
where we argued as in \eqref{estimateessinfmaximaldyadichole1}-\eqref{estimateessinfmaximaldyadichole2} to obtain the last inequality. Also, denoting by $z_0$ and $z_0^*$ the reference points of $Q_0$ and $Q_0^*$, we have the inclusions
$$
B(z_0, a \theta^{g(Q_0)} ) \subset Q_0 \subset Q_0^* \subset B(z_0^*, A \theta^{g(Q_0^*)} )\subset B(z_0, 3 A \theta^{g(Q_0^*)} ).
$$
The doubling condition on the measure gives $\mu(Q_0^*) \leq C(\theta,C_\mu) \mu(\widehat{Q}_0).$ Using this estimate in \eqref{estimateessinfdyadicparentcondition} and reorganizing the terms we conclude
\begin{equation}\label{estimateparentconditiondyadicholes}
g_E(\widehat{Q}_0) \leq m+  g_E(Q_0) ,
\end{equation}
for some $m=m( \alpha,[w]_{A_1},\theta,C_\mu) \in \N.$ 
\end{proof}

As a consequence of the previous lemma, we obtain the implication \ref{maintheoremconditionA1}$\implies$\ref{maintheoremconditionwp} of Theorem \ref{maintheoremallspaces}.
\begin{lemma}\label{resultnecessityweakporosity}
If $E \subset X$ is a subset and $w= \dist( \cdot, E)^{-\alpha} \in A_1(X)$ with $\alpha>0,$ then $E$ is weakly porous and $h_E$ is doubling.
\end{lemma}
\begin{proof}
We apply Lemma \ref{dyadicimplicationsA1} to $E,$ obtaining that $E$ is dyadic weakly porous and that $g_E$ is doubling. By Lemma \ref{assumingdoublingdyadicwpimplieswp}, we conclude that $E$ is weakly porous, with $h_E$ doubling. 
\end{proof}

\section{$E$-free dyadic decompositions and a key discrete inequality}\label{sectiondiscreteinequalityA1}

In this section, we begin with the proof of the implication \ref{maintheoremconditionwp}$\implies$\ref{maintheoremconditionA1} in Theorem \ref{maintheoremallspaces}. Suppose that $E$ is weakly porous and that $h_E$ is doubling. By the subsequent comments to Definitions \ref{definitionweakporballs} and \ref{definitionofdoubling}, the same properties hold for $\overline{E}$ as well. Since our goal is to prove that $\dist(\cdot, E)^{-\alpha}$ for some $\alpha>0,$ and $\dist(\cdot, E) = \dist(\cdot, \overline{E})$, we may and do assume that $E$ is closed throughout the Sections \ref{sectiondiscreteinequalityA1} and \ref{sectionsufficiencymaintheorem}.

By Lemma \ref{assumingdoublingwpimpliesdyadicwp}, $E$ is dyadic weakly porous, and $g_E$ is doubling (Definitions \ref{definitiondyadicweakporosity} and \ref{definitionofdoubling}).

Let $M, m \in \N$, $c\in (0,1)$ be the constants from Definitions \ref{definitiondyadicweakporosity} and \ref{definitionofdoubling} for the set $E.$ Because $E$ is (dyadic) weakly porous, for every $P \in \mathcal{D}$ there exists a cube $Q\in \mathcal{D}(P)$ with $Q \cap E= \emptyset.$ Thus, the number $g_E(P)$ is well-defined; see formula \eqref{definitiondyadicmaximalhole}. 

Now, given $P \in \mathcal{D},$ we define the families of cubes
$$
\widehat{\mathcal{F}}_M(P)=\lbrace Q \in \mathcal{D}(P) \: : \: Q \cap E = \emptyset, \: g(Q) \leq M +g_E(P) \rbrace;
$$
$$
\widehat{\mathcal{G}}_M(P)=\lbrace Q \in \mathcal{D}(P) \: : \: Q \subset P \setminus \bigcup_{R \in \widehat{\mathcal{F}}_M(P)} R \rbrace.
$$
Let us construct \textit{maximal} subfamilies of $\widehat{\mathcal{F}}_M(P)$ and $\widehat{\mathcal{G}}_M(P).$ Given $Q\in \widehat{\mathcal{F}}_M(P),$ let $Q^\sharp \in \widehat{\mathcal{F}}_M(P)$ and $Q^\flat \in \widehat{\mathcal{G}}_M(P)$ be cubes with $Q \subset Q^\sharp, Q^\flat$ such that
$$
g(Q^\sharp)=\min\lbrace g(Q') \, : \, Q \subset Q' \text{ and } Q'\in \widehat{\mathcal{F}}_M(P) \rbrace,
$$
$$
g(Q^\flat)=\min\lbrace g(Q') \, : \, Q \subset Q' \text{ and } Q'\in \widehat{\mathcal{G}}_M(P) \rbrace.
$$
We define
$$
\mathcal{F}_M(P):= \lbrace Q^\sharp \, : \, Q \in \widehat{\mathcal{F}}_M(P) \rbrace, \quad \mathcal{G}_M(P)=\lbrace Q^\flat \, : \, Q \in \widehat{\mathcal{G}}_M(P) \rbrace.
$$

It is important to notice that, by the minimality of their generations, any two cubes $R, R'$ of $\mathcal{F}_M(P)$ (or $\mathcal{G}_M(P)$) defining the same set of points must have the same generation. We remind that this implies that $R$ and $R'$ have exactly the same label within the system $\mathcal{D}_t$ to which $P$ belongs, according to Lemma \ref{observationsdyadicsets}. This guarantees that the families $\mathcal{F}_M(P)$ and $\mathcal{G}_M(P)$ consist of non-duplicate cubes.

Let us make several remarks and clarifications concerning the families $\mathcal{F}_M$ and $\mathcal{G}_M.$ 
\begin{remark}\label{propertiesmaximalsubfamiliesFG}
If $P \in \mathcal{D},$ the families $\mathcal{F}_M(P)$ and $\mathcal{G}_M(P)$ satisfy the following.
\begin{enumerate}
\item[(a)] $\mathcal{F}_M(P) \subset  \widehat{ \mathcal{F}}_M(P)$, $\mathcal{G}_M(P) \subset \widehat{ \mathcal{G}}_M(P)$, and $\bigcup \mathcal{F}_M(P) = \bigcup \widehat{\mathcal{F}}_M(P)$, $\bigcup \mathcal{G}_M(P) = \bigcup \widehat{\mathcal{G}}_M(P)$. Moreover, $P$ is the disjoint union of $\bigcup \mathcal{F}_M(P) $ and $\bigcup \mathcal{G}_M(P) $.
\item[(b)] $\mathcal{F}_M(P) \neq \emptyset$ always, and $\mathcal{F}_M(P)=\lbrace P \rbrace$ (or equivalently $P\notin \mathcal{F}_M(P)$) if and only if $P \cap E = \emptyset.$ Also, $\mathcal{G}_M(P) \neq \emptyset  $ if and only if $P \cap E \neq \emptyset.$
\item[(c)] If $R, R' \in \mathcal{F}_M(P)$ (or $R,R'\in \mathcal{G}_M(P)$), then either $R \cap R' = \emptyset$; or else $R$ and $R'$ are the same cube, meaning that they define the same set of points and $g(R)=g(R').$ 
\item[(d)] If $R \in \mathcal{G}_M(P),$ and $R^* \in \mathcal{D}(P)$ is so that $R \subset R^*$ and $g(R^*)=g(R)-1,$ then $R^*$ intersects some $Q \in \mathcal{F}_M(P).$ 
\item[(e)] If $R \in \mathcal{F}_M(P),$ then $R \cap E = \emptyset.$ And if $R \in \mathcal{G}_M(P),$ then $g(R) \leq M  + g_E(P)$ and $R \cap E \neq \emptyset.$
\item[(f)] $\mu\left( \bigcup_{Q \in \mathcal{F}_M(P)} Q  \right)= \sum_{Q \in \mathcal{F}_M(P)} \mu(Q) \geq c \mu(P).$
\end{enumerate}

\end{remark}
\begin{proof}
The first part of (a) is immediate from the definitions of $Q^\sharp$ and $Q^\flat.$ For the second part, observe that for any two cubes $R\in \mathcal{F}_M(P)$ and $R' \in \mathcal{G}_M(P)$. Then also $R \in \widehat{\mathcal{F}}_M(P)$ and $R' \in \widehat{\mathcal{G}}_M(P)$ and the definition of $\widehat{\mathcal{G}}_M(P)$ implies $R' \subset P \setminus \bigcup \widehat{\mathcal{F}}_M(P) \subset P \setminus R.$

The statement (b): $\mathcal{F}_M(P) \neq \emptyset$ for every $P \in \mathcal{D}$ thanks to the (dyadic) weak porosity of $E.$ Now, if $P \cap E \neq \emptyset,$ then $P \notin \widehat{\mathcal{F}}_M(P),$ and therefore $\mathcal{F}_M(P) \neq \lbrace P \rbrace.$ Also, $\bigcup_{R\in \widehat{\mathcal{F}}(P)} R \subsetneq P$ because $P \cap E \neq \emptyset,$ and since all the cubes of $\widehat{\mathcal{F}}_M(P)$ have generation at least $M+g_E(P),$ there must exist $Q \in \mathcal{D}(P)$ with $Q \subset P \setminus E.$ Thus $Q \in\bigcup_{R\in \widehat{\mathcal{F}}(P)} R ,$ and therefore $Q \in \widehat{\mathcal{G}}_M(P),$ and $ \mathcal{G}_M(P) \neq \emptyset .$ And if $P \cap E =\emptyset,$ then, by the minimality of the generations of the cubes in $\mathcal{F}_M(P),$ we must have $\mathcal{F}_M(P) =\lbrace P \rbrace$ and so $\mathcal{G}_M(P) = \emptyset.$

For property (c), suppose that $R,R'\in \mathcal{F}_M(P)$ with $R \cap R' \neq \emptyset.$ Because $R$ and $R'$ are in the same grid $\mathcal{D}_t,$ we have $R \subset R'$ or $R' \subset R.$ Assume for example $R\subset R'.$ Now write $R=Q^\sharp$ and $R'=(Q')^\sharp$, with $Q, Q' \in \widehat{\mathcal{F}}_M(P)$. Because $Q \subset R \subset R',$ by the definition of $Q^\sharp,$ we must have $g(R) \leq g(R').$ This implies $R=R'$ as otherwise we would have $R \subsetneq R'$ and so, by Lemma \ref{observationsdyadicsets}, $g(R) >g(R') ,$ a contradiction. So, $R=R',$ and because $Q' \subset R' =R,$ the definition of $(Q')^\sharp$ gives the reverse inequality $g(R') \leq g(R)$. Lemma \ref{observationsdyadicsets} implies that $R$ and $R'$ are the same cube. The proof for cubes $R,R'\in \mathcal{G}_M(P)$ is identical.

To prove (d), observe that if $R^* \subset P \setminus \bigcup_{Q\in \mathcal{F}_M(P)} Q,$ then $R^* \in \widehat{\mathcal{G}}_M(P)$, contradicting the minimality of $g(R)$.

Let us prove (e). If $R \in \mathcal{F}_M(P),$ then $R \in \widehat{\mathcal{F}}_M(P)$ as well, and so $R \cap E = \emptyset.$ Now assume that $R \in \mathcal{G}_M(P).$ By (b) this implies $P \cap E \neq \emptyset,$ and so $R \subset P \setminus \bigcup \mathcal{F}_M(P) \subsetneq P,$ which in turn yields $g(R) >g(P)$ thanks to Lemma \ref{observationsdyadicsets}. Thus, we can find $R^*\in \mathcal{D}(P)$ with $R \subset R^*$ and $g(R^*)=g(R)-1.$ By property (d), $R^*$ intersects some cube $Q\in \mathcal{F}_M(P),$ and hence either $R^* \subseteq Q$ or $Q \subsetneq R^*.$ Since $R \cap Q = \emptyset$ and $R \subset R^*,$ the first case is impossible. Therefore, we must have $Q \subsetneq R^*,$ yielding
$$
M+g_E(P)\geq  g(Q) \geq g(R^*)+1 = g(R).
$$
Thus $g(R) \leq M+ g_E(P).$ Now, because $R \subset P \setminus \bigcup \mathcal{F}_M(P) ,$ the cube $R \in \mathcal{D}(P)$ is not in $\widehat{\mathcal{F}}_M(P),$ and because $g(R)$ satisfies the estimate $g(R) \leq M+ g_E(P),$ we must necessarily have $R \cap E \neq \emptyset.$

Finally, let us show (f). By the dyadic weak porosity of $E$, there are pairwise disjoint cubes $Q_1, \ldots, Q_N \in \mathcal{D}(P)$ belonging to $\widehat{\mathcal{F}}_M(P)$ and satisfying $\sum_{j=1}^N \mu(Q_j) \geq c \mu(P).$ Each of these cubes $Q_j$ is contained in some cube of the disjointed family $\mathcal{F}_M(P),$ and we can write
$$
\sum_{Q \in \mathcal{F}_M(P)} \mu(Q) =\mu ( \bigcup_{Q\in \mathcal{F}_M(P)} Q  )  \geq \mu \left( \bigcup_{j=1}^N Q_j  \right) = \sum_{j=1}^N \mu(Q_j) \geq c \mu(P).
$$
\end{proof}

\medskip

We continue defining new families of cubes. Fix $P \in \mathcal{D},$ and define
$\mathcal{G}_M^0(P)=\{P \}$, $\mathcal{F}_M^1(P)=\mathcal{F}_M(P)$, $\mathcal{G}_M^1(P)=\mathcal{G}_M(P)$, and in general, for $k=2,3,4,\ldots$, we define
\begin{equation}\label{definitionsFMkGMk}
\mathcal{F}_M^k(P)=\bigcup_{R\in\mathcal{G}_M^{k-1}(P)} \mathcal{F}_M(R),\,\qquad \mathcal{G}_M^k(P)=\bigcup_{R\in\mathcal{G}_M^{k-1}(P)}\mathcal{G}_M(R).
\end{equation}

We now prove that the cubes of the families $\lbrace \mathcal{F}_M^k(P) \rbrace_k$ cover the complement of $E$ in $P.$

\begin{lemma}\label{lemmacoveringwithFfamilies}
For every $P \in \mathcal{D},$ we have $P\setminus E= \bigcup_{k=1}^\infty \bigcup_{Q \in \mathcal{F}_M^k(P)} Q.$
\end{lemma}
\begin{proof}
For every $Q \in \mathcal{F}_M^k(P),$ with $k \in \N,$ we have that $Q \in \mathcal{F}_M(R),$ for some $R\in \mathcal{G}_M^{k-1}(P).$ In particular, $Q \cap E =\emptyset$ (see Remark \ref{propertiesmaximalsubfamiliesFG}(e)), showing the inclusion $\bigcup_{k=1}^\infty \bigcup \mathcal{F}_M^k(P) \subset P \setminus E.$

 To prove the reverse inclusion, we may assume that $P \cap E \neq \emptyset,$ as otherwise $P \setminus E=P,$ $\mathcal{F}_M^1(P) =\lbrace P \rbrace$ and $\mathcal{F}_M^k(P) = \emptyset$ for every $k \geq 2.$ Now, let $x\in P \setminus E,$ and let $Q \in \mathcal{D}(P)$ be a cube with $x\in Q$ and $Q \cap E= \emptyset.$ Notice that such cube $Q$ exists because $E$ is closed, and for every $l\in \Z$ with $l \geq g(P),$ $P$ can be written as union of dyadic subcubes of $P$ with diameter at most $2A \theta^l.$

We will next show that $Q \subset \bigcup_{k=1}^\infty \bigcup \mathcal{F}_M^k(P).$ Suppose, for the sake of contradiction, that $Q \not \subset   \bigcup_{k=1}^\infty \bigcup \mathcal{F}_M^k(P).$ In particular, $Q \not \subset \bigcup \mathcal{F}_M^1(P).$ By Remark \ref{propertiesmaximalsubfamiliesFG}(a), there must exist some cube $R_1\in \mathcal{G}_M^1(P)$ intersecting $Q$. And this actually implies that $Q \subsetneq R_1,$ because if $R_1 \subseteq  Q$ the facts that $R_1 \cap E \neq \emptyset$ (see Remark \ref{propertiesmaximalsubfamiliesFG}(e)) and $Q \cap E = \emptyset$ lead us to contradiction. Therefore $Q$ is entirely and strictly contained in $R_1.$ But now $Q \not \subset \bigcup \mathcal{F}_M^2(P) = \bigcup_{R \in \mathcal{G}_M^1(P)}\bigcup \mathcal{F}_M(R)$, and so $Q \not \subset \bigcup \mathcal{F}_M(R_1)$. Using the same argument as above, the cube $Q$ must be strictly contained in some $R_2 \in \mathcal{G}_M(R_1) \subset \mathcal{G}_M^2(P) .$ And since $Q \not \subset \bigcup \mathcal{F}_M^3= \bigcup_{R \in \mathcal{G}_M^2(P)} \bigcup \mathcal{F}_M(R)$, we have $Q \not \subset \bigcup \mathcal{F}_M(R_2).$ Repeating this procedure up to any natural number $k,$ we obtain cubes 
 $$
 R_1 \supset R_2 \supset \cdots \supset R_k \supsetneq Q,
 $$
with $Q \not \subset \bigcup \mathcal{F}_M(R_k)$ and $R_j \in \mathcal{G}_M(R_{j-1})$ for every $j=2,\ldots,k.$ Because $Q$ is strictly contained in $R_k,$ then $g(Q) > g(R_k),$ by Lemma \ref{observationsdyadicsets}, and we can say that $Q \in \mathcal{D}(R_k).$ Moreover, $Q$ is not contained in any $R\in \mathcal{F}_M(R_k)$, and $Q\in \mathcal{D}(R_k)$, $Q \cap E = \emptyset.$ According to the definition of $\mathcal{F}_M(R_k),$ this tells us that $g(Q) > M + g_E(R_k).$ On the other hand, each $R_j$ is strictly contained in $R_{j-1}$ because $R_j \subset R_{j-1} \setminus \bigcup \mathcal{F}_M(R_{j-1}),$ and $\mathcal{F}_M(R_{j-1}) \neq \emptyset$. Thus $g(R_j) \geq g(R_{j-1}) +1$ for every $j,$ and all these observations lead us to
$$
g(Q) > M + g_E(R_k) \geq M +g(R_k) \geq M + g(R_1) + k-1.
$$
Since $k$ is arbitrary, we get a contradiction, showing that $Q \subset \bigcup_{k=1}^\infty \bigcup \mathcal{F}_M^k(P).$ 

\end{proof}

\begin{lemma}\label{keylemmasufficiency}
There exists a constant $\beta=\beta(\theta,c,m,M)\in (0,1)$ depending on $\theta, c,m,M$ such that 
\[
\sum_{k=1}^\infty \sum_{Q\in\mathcal{F}_M^k(Q_0) } \theta^{-  g(Q) \gamma} \mu( Q)
\leq \frac{2}{c} \, \theta^{- g_E(Q_0) \gamma}  \mu( Q_0).
\]
for every cube $Q_0 \in \mathcal{D} $ and every $\gamma \in (0,\beta].$  
\end{lemma}
\begin{proof}
Let $\beta \in (0,1)$ be a constant whose value will be specified later. Fix a cube $Q_0 \in \mathcal{D}.$ We will need the following claim.
\begin{claim}\label{claiminduction}
For every natural $k \in \N,$ we have
$$
\sum_{R \in \mathcal{G}_M^{k-1}(Q_0)} \theta^{-g_E(R)\beta} \mu(R)\leq  \left( (\theta^{m+M} )^{-\beta} (1-c) \right)^{k-1} \theta^{-g_E(Q_0)\beta} \mu(Q_0) .
$$
\end{claim}
\begin{proof}[Proof of Claim \ref{claiminduction}]
For $k=1,$ we have $\mathcal{G}_M^{k-1}(Q_0)=\lbrace Q_0 \rbrace,$ and the estimate trivially holds. 
Now, suppose that the inequality holds for $k,$ and let us estimate from above the term $\sum_{R \in \mathcal{G}_M^{k}(Q_0)} \theta^{-g_E(R)\beta} \mu(R)$. Each $R\in \mathcal{G}_M^k(Q_0)$ belongs to $\mathcal{G}_M(P)$ for some $P\in \mathcal{G}_M^{k-1}(Q_0).$ In particular $R\in \mathcal{D}(P),$ and if $R^* \in \mathcal{D}(P)$ with $R \subset R^*$ and $g(R^*)=g(R)-1,$ then $R^* \cap Q \neq \emptyset,$ for some $Q\in \mathcal{F}_M(P)$ by Remark \ref{propertiesmaximalsubfamiliesFG}(d). 
If $R^* \subseteq Q,$ then we have $R \subset R^* \subset \bigcup_{Q'   \in \mathcal{F}_M(P)} Q',$ which contradicts the fact that $R$ belongs to $\mathcal{G}_M(P).$ Thus, we necessarily have $Q \subsetneq R^*.$ We remind that this implies $g(Q) > g(R^*)$ by Lemma \ref{observationsdyadicsets}, and so $Q \in \mathcal{D}(R^*)$. In particular $g_E(R^*) \leq g(Q)$ because $Q \cap E =\emptyset.$ Using that $g_E$ is doubling (see Definition \ref{definitionofdoubling}) and the fact $Q\in \mathcal{F}_M(P)$, we obtain the estimates
\begin{equation}\label{usingthedoublingcondition}
\theta^{g_E(R)} \geq \theta^m \theta^{g_E(R^*)} \geq \theta^m \theta^{g(Q)} \geq \theta^{m+M} \theta^{g_E(P)}.
\end{equation}
On the other hand, the weak porosity condition together with Remark \ref{propertiesmaximalsubfamiliesFG}(a) give
\begin{equation}\label{estimateupperboundmeasure}
\sum_{R \in \mathcal{G}_M(P)} \mu(R) = \mu (P) - \sum_{R \in \mathcal{F}_M(P)} \mu(R) \leq (1-c) \mu(P).
\end{equation}
Using first \eqref{usingthedoublingcondition}, then \eqref{estimateupperboundmeasure}, and finally the induction hypothesis, we derive
\begin{align*}
\sum_{R \in \mathcal{G}_M^{k}(Q_0)} \theta^{-g_E(R)\beta} \mu(R) & = \sum_{P \in \mathcal{G}_M^{k-1}(Q_0)} \sum_{R\in \mathcal{G}_M(P)}  \theta^{-g_E(R)\beta} \mu(R) \\
& \leq (\theta^{m+M} )^{-\beta} \sum_{P \in \mathcal{G}_M^{k-1}(Q_0)} \theta^{-g_E(P)\beta} \sum_{R\in \mathcal{G}_M(P)}   \mu(R) \\
& \leq (\theta^{m+M} )^{-\beta} (1-c) \sum_{P \in \mathcal{G}_M^{k-1}(Q_0)} \theta^{-g_E(P)\beta} \mu(P) \\
& \leq (\theta^{m+M} )^{-\beta} (1-c) \left( (\theta^{m+M} )^{-\beta} (1-c)  \right)^{k-1} \theta^{-g_E(Q_0)\beta} \mu(Q_0) \\
& = \left( (\theta^{m+M} )^{-\beta} (1-c)  \right)^{k}\theta^{-g_E(Q_0)\beta} \mu(Q_0).
\end{align*} 
This shows that the desired estimate holds for $k+1$ as well, and the proof of Claim \ref{claiminduction} is complete. \end{proof}

Now we continue with the proof of Lemma \ref{keylemmasufficiency}. For every $k\in \N,$ we can write
\begin{align*}
\sum_{Q\in\mathcal{F}_M^k(Q_0) } \theta^{-g(Q) \beta} \mu( Q) & = \sum_{R\in\mathcal{G}_M^{k-1} (Q_0) } \sum_{Q\in \mathcal{F}_M(R)} \theta^{-  g(Q) \beta} \mu( Q) \\
& \leq \sum_{R\in\mathcal{G}_M^{k-1}(Q_0) }\theta^{-M\beta} \theta^{-  g_E(R) \beta} \sum_{Q\in \mathcal{F}_M(R)}  \mu( Q) \\
& \leq \theta^{-M\beta}  \sum_{R\in\mathcal{G}_M^{k-1}(Q_0) }\theta^{-  g_E(R) \beta}    \mu( R),
\end{align*}
where in the last inequality we used the fact that the cubes of $\mathcal{F}_M(R)$ are disjoint. Applying Claim \ref{claiminduction} to the last term, we obtain
\begin{align*}
\sum_{k=1}^\infty\sum_{Q\in\mathcal{F}_M^k (Q_0) } \theta^{-g(Q) \beta} \mu( Q) & \leq \theta^{-M\beta} \sum_{k=1}^\infty \sum_{R\in\mathcal{G}_M^{k-1}(Q_0) }  \theta^{-  g_E(R) \beta}    \mu( R) \\
&  \leq   \theta^{-M\beta}   \theta^{-g_E(Q_0)\beta} \mu(Q_0) \sum_{k=1}^\infty \left( (\theta^{m+M} )^{-\beta} (1-c) \right)^{k-1}.
\end{align*}
If the parameter $\beta>0$ is small enough, then $(\theta^{m+M} )^{-\beta} (1-c) <1$ and the constant multiplying $\theta^{-g_E(Q_0)\beta} \mu(Q_0)$ becomes $\theta^{-M\beta} \left( 1- (\theta^{m+M} )^{-\beta} (1-c) \right)^{-1}.$ For sufficiently small $\beta$ this term is bounded by $2/c.$ 
\end{proof}

\section{Proof of Theorem \ref{maintheoremallspaces}}\label{sectionsufficiencymaintheorem}

In this section we prove the implication \ref{maintheoremconditionwp}$\implies$\ref{maintheoremconditionA1} of Theorem \ref{maintheoremallspaces}. Let $E$ be a weakly porous set such that $h_E$ is doubling. As pointed out in Section \ref{sectiondiscreteinequalityA1}, Lemma \ref{assumingdoublingwpimpliesdyadicwp} then implies that $E$ is dyadic weakly porous and $g_E$ is doubling, which allows us to use Lemma \ref{keylemmasufficiency}. Also, $E$ can be assumed to be closed, as we explained in the beginning of Section \ref{sectiondiscreteinequalityA1}.

In the next three subsections, we are going to prove the $A_1$-estimate \eqref{e.a_1} for the weight $\dist(\cdot,E)^{-\gamma}$ for suitable $\gamma>0,$ for all the cubes of the systems $\mathcal{D}$ instead of balls. Then, a simple argument will allow us to deduce \eqref{e.a_1} for every ball of $X.$

\subsection{$E$-free cubes that are far from $E$} \label{sufficiencyestimatesfarcubes}
Let $Q_0 \in \mathcal{D}$ such that $Q_0 \cap E= \emptyset,$ and $\dist(Q_0,E) \geq 2 \diam(Q_0).$ Observe that, for every $x,y\in Q_0,$ we have $\dist(y,E) \leq 2 \dist(x,E)$. Also, if $z_0$ denotes the reference point of $Q_0,$ because $Q_0 \cap E = \emptyset,$ the ball $B(z_0, a \theta^{g(Q_0)})$ does not contain any point of $E,$ and so $\dist(z_0, E) \geq a\theta^{g(Q_0)}.$  These observations lead us to the estimates 
\begin{equation}\label{A1estimateforcubesfar}
\intav_{Q_0} \dist(x, E)^{-\gamma}\ud\mu \leq 2^\gamma \essinf_{y\in Q_0} \dist(y, E)^{-\gamma} \leq 2^\gamma \dist(z_0,E)^{-\gamma} \leq (2/a)^\gamma \theta^{-g(Q_0) \gamma} ,
\end{equation}
for every $\gamma >0.$
\subsection{$E$-free cubes that are close to $E$}\label{sufficiencyestimatesclosecubes}
Let $Q_0 \in \mathcal{D}$ such that $Q_0 \cap E= \emptyset,$ $\dist(Q_0,E) < 2 \diam(Q_0).$ Notice that, since $E$ is closed, $\dist(x,E) >0$ for every $x\in Q_0.$ Also, we have that $E \subset X \setminus Q_0,$ and, if $x\in Q_0$, 
$$
\dist(x, X \setminus Q_0) \leq \dist(x,  E) < 3 \diam(Q_0) \leq 6 A \theta^{g(Q_0)}.
$$

Now we define the sets $A_j=\lbrace x\in Q_0 \, : \, 6 A\theta^{j+1}< \dist(x, E) \leq 6 A\theta^j \rbrace$, for every $j \in \Z$ with $j \geq g(Q_0)$. By the preceding remarks, $Q_0 = \bigcup_{j \geq g(Q_0)} A_j$, and the union is of course disjoint. Also, employing the estimate \eqref{estimatemeasureboundarydyadicsets} from Proposition \ref{furtherpropertiesofdyadicsets} with $t=6 A \theta^{j-g(Q_0)}$, we get positive constants $C$ and $\eta$ depending only on $C_\mu$ and $\theta$ so that
$$
\mu(A_j) \leq \mu \left( \lbrace x\in Q_0 \, : \, \dist(x, X \setminus Q_0) \leq 6 A\theta^j \rbrace \right) \leq C(6A)^\eta \theta^{(j-g(Q_0)) \eta} \mu(Q_0).
$$
We momentarily use $C$ to denote this constant, which should not be confused with any of the previously used $C$'s in the doubling condition for $h_E$. All these properties lead us to
\begin{align*}
\int_{Q_0} \dist(x, E)^{-\gamma} \ud\mu &   =\sum_{j=g(Q_0)}^\infty \int_{A_j} \dist(x, E)^{-\gamma}  \ud\mu \leq \sum_{j=g(Q_0)}^\infty (6A \theta^{j+1})^{-\gamma} \mu(A_j)\\
& \leq   C (6A)^\eta (6A\theta)^{-\gamma} \theta^{-g(Q_0) \eta} \mu(Q_0)   \sum_{j=g(Q_0)}^\infty \theta^{j(\eta-\gamma)}.
\end{align*}
Assuming $\gamma< \eta$, the series converges and we obtain the estimate
\begin{equation}\label{partialA1estimatesufficiency}
\int_{Q_0} \dist(x, E)^{-\gamma} \ud\mu \leq \kappa(\theta, C_\mu,\gamma) \theta^{-g(Q_0)\gamma} \mu(Q_0).
\end{equation}
Since $\dist(Q_0,E) < 2 \diam(Q_0),$ we have that $\dist(x,E) \leq 3 \diam(Q_0)$ for every $x \in Q_0,$ and so
$$
\essinf_{Q_0} \dist(\cdot,E)^{-\gamma} \geq (3 \diam(Q_0))^{-\gamma} \geq (6A \theta^{g(Q_0)})^{-\gamma}.
$$
Combining this estimate with \eqref{partialA1estimatesufficiency}, we conclude 
\begin{equation}\label{conclusionsubsectionclosecubes}
\vint_{Q_0} \dist(x, E)^{-\gamma} \ud\mu \leq \kappa(\theta, C_\mu,\gamma) \essinf_{Q_0} \dist(\cdot,E)^{-\gamma} .
\end{equation}

\subsection{Non-$E$-free cubes}\label{sectionestimatesnonfree}

Fix $Q_0 \in \mathcal{D}$ with $Q_0 \cap E \neq \emptyset$. Given $x\in Q_0 \setminus E,$ we have that $\dist(x,E)>0$ because $E$ is closed. For $l\in \Z$, $l\geq g(Q_0)$ so that $2 A \theta^l < \dist(x,E),$ we can write $Q_0$ as union of subcubes $Q \in \mathcal{D}(Q_0)$ with $g(Q) = l.$ The point $x$ must be contained in some of these cubes, say $Q,$ and because $\diam(Q) \leq 2A \theta^l,$ we have $Q \cap E = \emptyset.$

This allows us to find $Q_x\in \mathcal{D}(Q_0)$ with $x\in Q_x,$ $ Q_x \cap E= \emptyset$, and 
$$
g(Q_x)=\min\lbrace g(Q) \, : \, x\in Q, \, Q\in \mathcal{D}(Q_0),  \text{ and } Q \cap E= \emptyset \rbrace.
$$
Notice that $g(Q_x) \geq g_E(Q_0).$ Also, because $Q_x \subsetneq Q_0$ (since $Q_0 \cap E \neq \emptyset= Q_x \cap E$), we can find $Q_x^*\in \mathcal{D}(Q_0)$ with $Q_x \subset Q_x^*$ and $g(Q_x^*)= g(Q_x)-1.$ By the minimality of $g(Q_x),$ the cube $Q_x^*$ must intersect $E.$ Then, if $z^*$ is the reference point of $Q_x^*$, the containing ball $B \left( z^*, A \theta^{g(Q_x^*)} \right)$ of $Q_x^*$ also intersects $E,$ and we derive
$$
\dist(x,E) \leq \diam\left( B \left( z^*, A \theta^{g(Q_x^*)} \right) \right) \leq   \frac{2A}{\theta}   \theta^{g(Q_x)} \leq   \frac{2A}{\theta}   \theta^{g_E(Q_0)}.
$$
Consequently, the following estimate holds:
\begin{equation}\label{estimateforessinfsuficieny}
\essinf_{x\in Q_0} \dist(x,E)^{-\gamma} \geq  \left( \frac{\theta}{2A} \right)^\gamma \theta^{-g_E(Q_0) \gamma}.
\end{equation}
On the other hand, using Lemma \ref{lemmacoveringwithFfamilies} and the fact that all the cubes in the family $\lbrace Q \, : \, Q \in \mathcal{F}_M^k(Q_0), \, k\in \N \rbrace$ are mutually disjoint, we can write
$$
\int_{Q_0 \setminus E} \dist(x,E)^{-\gamma} \ud\mu  =  \sum_{k=1}^\infty \sum_{Q \in \mathcal{F}_M^k(Q_0)} \int_{Q}\dist(x,E)^{-\gamma} \ud\mu .
$$
Now, each cube $Q \in \mathcal{F}_M^k(Q_0)$ does not intersect $E;$ see the definition of this families in \eqref{definitionsFMkGMk} together with Remark \ref{propertiesmaximalsubfamiliesFG}(e). Thus, we can estimate from above each of the integrals in the previous series, by applying either \eqref{A1estimateforcubesfar} or \eqref{partialA1estimatesufficiency} depending on whether the cube $Q \in \mathcal{F}_M^k(Q_0)$ satisfies the assumption of Subsections \ref{sufficiencyestimatesfarcubes} or \ref{sufficiencyestimatesclosecubes}. Notice that \eqref{A1estimateforcubesfar} holds for any positive $\gamma,$ whereas \eqref{partialA1estimatesufficiency} holds if $\gamma < \eta.$ Hence, taking $\gamma < \eta,$ the preceding series can be estimated from above by
$$
 \sum_{k=1}^\infty \sum_{Q \in \mathcal{F}_M^k(Q_0)} \kappa(\theta,C_\mu,\gamma) \theta^{-g(Q)\gamma} \mu(Q).
$$
Now, if $\beta=\beta(\theta,c,m,M) >0$ is the number from Lemma \ref{keylemmasufficiency}, we take further $\gamma < \beta,$ and the same lemma gives the estimate
$$
 \sum_{k=1}^\infty \sum_{Q \in \mathcal{F}_M^k(Q_0)} \kappa(\theta,C_\mu,\gamma) \theta^{-g(Q)\gamma} \mu(Q) \leq \kappa(\theta,C_\mu, \gamma,c)\theta^{-g_E(Q_0) \gamma} \mu( Q_0).
 $$ 
Recall that $M,c$ are the constants of dyadic weak porosity of $E,$ and that $m$ is the doubling constant for $g_E.$ Together with \eqref{estimateforessinfsuficieny}, these estimates yield
\begin{equation}\label{conclusionsubsectionnonfreecubes}
\intav_{Q_0} \dist(x,E)^{-\gamma} \ud\mu \leq \kappa(\theta,C_\mu,\gamma,c) \essinf_{x\in Q_0} \dist(x,E)^{-\gamma}.
\end{equation}

Notice that we have used the fact that $\mu(E)=0,$ which follows from the weak porosity condition for $E;$ see Definition \ref{definitionweakporballs} and the subsequent comment.

\subsection{The $A_1$-estimate for all balls of $X$}

Collecting the estimates from the three previous subsections, we obtain the following theorem, which is the implication \ref{maintheoremconditionwp}$\implies$\ref{maintheoremconditionA1} in Theorem \ref{maintheoremallspaces}.

\begin{theorem}\label{thm:rightimplicationmaintheorem}
Let $E$ be a weakly porous set with constants $c,\delta$ and so that $h_E$ is doubling with constant $C.$ Then there exists $\gamma=\gamma(\theta,C_\mu,c,\delta,C) >0$ so that $\dist(\cdot, E)^{-\gamma} \in A_1(X).$ The $A_1$-constant of the weight $\dist(\cdot, E)^{-\gamma}$ depends on $\theta,C_\mu,\gamma,c.$
\end{theorem}
\begin{proof}
By Lemma \ref{assumingdoublingwpimpliesdyadicwp}, $E$ is dyadic weakly porous with some new constants $c'=c'(\theta,C_\mu,c)$ and $M=M(\theta,\delta,C)$, and $g_E$ is doubling with constant $m=m(\theta,C) \in \N.$

We combine the conclusions of Subsections \ref{sufficiencyestimatesfarcubes}, \ref{sufficiencyestimatesclosecubes}, and \ref{sectionestimatesnonfree}, i.e., \eqref{A1estimateforcubesfar}, \eqref{conclusionsubsectionclosecubes} and \eqref{conclusionsubsectionnonfreecubes}. We deduce that, for $\gamma=\gamma(\theta,C_\mu,c',M,m) >0$ small enough, and for every cube $Q\in \mathcal{D},$ the estimate
$$
\intav_{Q} \dist( \cdot, E)^{-\gamma} \ud\mu \leq \kappa(\theta,C_\mu,\gamma,c') \essinf_Q \dist(\cdot, E)^{-\gamma}
$$
holds true. Now, given any open ball $B=B(x,R) \subset X,$ Proposition \ref{theoremexistencedyadicgrids} provides a cube $Q\in \mathcal{D}$ with $B \subset Q$ and $\theta^{g(Q)+2} < R \leq \theta^{g(Q)+1}.$ If $z$ denotes the reference point of $Q,$ we have the inclusion $Q \subset B(z, A \theta^{g(Q)}) $. Because $x$ is contained in this ball and $\theta^{g(Q)+2} < R,$ the triangle inequality gives $B(z,A  \theta^{g(Q)}) \subset B(x, 2 A \theta^{-2} R).$ Hence, the doubling condition for $\mu$ yields $\mu(Q) \leq C(\theta, C_\mu) \mu(B).$ These observations permit to write
\begin{align*}
 \intav_B  \dist( \cdot, E)^{-\gamma} \ud\mu  & \leq \kappa(\theta, C_\mu) \intav_Q  \dist( \cdot, E)^{-\gamma} \ud\mu \\
 & \leq \kappa(\theta,C_\mu,\gamma,c') \essinf_Q \dist(\cdot, E)^{-\gamma} \leq \kappa(\theta,C_\mu,\gamma,c) \essinf_B \dist(\cdot, E)^{-\gamma} .
\end{align*}

\end{proof}

\begin{proof}[Proof of Theorem \ref{maintheoremallspaces}]
It now suffices to combine Lemma \ref{resultnecessityweakporosity} and Theorem \ref{thm:rightimplicationmaintheorem}
\end{proof}

\section{Proof of Theorem \ref{maintheoremannular} and a bonus characterization}\label{sectionproofmainthmannularabsolute}

We now prove that in spaces with a slightly weaker variant of the standard annular decay property, the weak porosity condition implies the doubling property of the maximal holes. This permits to complete the proof of Theorem \ref{maintheoremannular}. Then, without the assumption of the variant of the annular decay property, we give a sufficient condition for a dyadic weakly porous set $E$ to imply the doubling property for its dyadic maximal free holes $g_E,$ and this will lead us to another characterization of the property $\dist(\cdot,E)^{-\alpha}\in A_1(X)$ for some $\alpha>0.$  

\subsection{Admissible balls with small radii} The following lemma of rescaling for the parameter $\delta$ in the definition of weak porosity is valid in any metric space with a doubling measure. It essentially tells us that one can replace the parameters $(c,\delta)$ in the definition of weak porosity with $(c \, C_\mu^{-2}, \delta')$ for $\delta'$ as small as desired, and so that the admissible balls have arbitrarily small radii.  

\begin{lemma}\label{rescalinglemma}
Suppose $E$ is weakly porous with constants $c, \delta \in (0,1).$ Then, for every $\tau \in (0,1/2),$ there exists $\delta_\tau \in (0,1)$ depending on $\delta$ and $\tau,$ for which the following holds.

Given a ball $B_0=B(x_0,R_0) \subset X,$ there exists a finite collection of disjoint balls $\lbrace B(z_n,s_n) \rbrace_n$ contained in $B_0 \setminus E,$ with 
$$
\tau R_0 \geq s_n \geq \delta_\tau h_E(B_0) \quad \text{and} \quad \sum_n \mu( B(z_n,s_n) ) \geq c \, C_\mu^{-2}  \mu(B_0).
$$
\end{lemma}
\begin{proof}
Let $\lbrace B(x_j,r_j) \rbrace_j$ be a collection of disjoint balls as in Definition \ref{definitionweakporballs} for the ball $B_0.$ For every $j,$ let $\lbrace z_j^i \rbrace_i$ be a finite $\tau r_j$-maximal net of the relatively compact set $B(x_j,(1-\tau) r_j).$ This means that $z_j^i \in B(x_j,(1-\tau) r_j)$ for every $i,j$; that $B(x_j,(1-\tau) r_j) \subset \bigcup_i B(z_j^i, \tau r_j)$ for every $j$; and the balls $\lbrace B(z_j^i, \frac{\tau r_j}{2}) \rbrace_i$ are pairwise disjoint. In addition, the triangle inequality gives $\bigcup_i B(z_j^i, \tau r_j) \subset B(x_j,r_j)$ for every $j.$ Define $s_j=\frac{\tau r_j}{2}$ and $\delta_\tau=\frac{\tau \delta}{2}.$ Since the balls of the collection $\lbrace B(x_j,r_j) \rbrace_j$ are pairwise disjoint and are contained in $B_0 \setminus E,$ the new family $\lbrace B(z_j^i, s_j) \rbrace_{i,j}$ satisfies these two properties as well. Also, their radii $s_j$ satisfy 
$$
 \tau R_0 \geq  s_j \geq \frac{\tau \delta}{2} h_E(B_0) = \delta_\tau h_E(B_0),
$$
and the first claim is proven. As concerns the estimate for the measures, we write
\begin{align*}
c \mu(B_0) &  \leq \sum_j \mu(B(x_j,r_j)) \leq C_\mu \sum_j \mu(B(x_j,(1-\tau)r_j)) \\
& \leq  C_\mu \sum_j \sum_i \mu(B(z_j^i, \tau r_j)) \leq C_\mu^2  \sum_{j,i} \mu(B(z_j^i, s_j)).
\end{align*}
Dividing by $C_\mu^2$ in both sides, we get the desired inequality. 

\end{proof}

\subsection{Measures with annular decay properties}\label{sectionannulardecay}
Assume that the space $(X, d, \mu)$ satisfies the following \textit{annular decay property}: for every $\eta>0,$ there exists $\varepsilon \in (0,1)$ such that
\begin{equation}\label{measureannulardecayproperty}
\sup_{x\in X,\, 0<t\leq s, \, 1- \frac{t}{s}\leq \varepsilon} \frac{\mu\left(  B(x,s) \setminus B(x,t)   \right)}{\mu(B(x,s))} < \eta.
\end{equation}

The standard annular decay property with power-type rate of decay (see e.g., \cite[Definition 1.1]{BBL17}) is slightly stronger than \eqref{measureannulardecayproperty}. In this setting, if $E$ is weakly porous, then $h_E$ is doubling.

\begin{lemma}\label{lemmaannulardecayweaklyporousdoubling}
Let $E \subset X$ be a weakly porous set, with constants $c, \delta \in (0,1).$

\begin{enumerate}[label=\textup{(\roman*)}]
\item There exist constants $\lambda^* \in (1,2)$ and $\delta^* \in (0,1)$ depending on $c,\delta,C_\mu$ and the parameters from \eqref{measureannulardecayproperty}, for which $h_E(B(x_0,R_0)) \geq \delta^* h_E(B(x_0, \lambda^* R_0))$ for all $x_0 \in X,$ $R_0>0.$

\item There are positive constants $b=b(\delta^*)$ and $d=d(\delta^*, \lambda^*)$ such that, for every $x_0\in X$, $R_0>0,$ and $L\geq 1,$ we have
$$
h_E(B(x_0, L R_0)) \leq b L^d h_E(B(x_0,R_0)).
$$
In particular, $h_E$ is doubling. 
\end{enumerate}
\end{lemma}
\begin{proof}
Let us prove (i). For $ \eta= c \, C_\mu^{-2},$ let $\varepsilon \in (0,1)$ be as in \eqref{measureannulardecayproperty}. Now, let parameters $\lambda \in (1,2)$, $\tau \in (0,1/4)$ be so that $\lambda \in (1,2)$ with $ \lambda^{-1}=1- \frac{\varepsilon}{2}$ and $\tau=\varepsilon/4.$ The desired constants $\lambda^*$ and $\delta^*$ are respectively defined as $\lambda$ and $\delta_\tau,$ where $\delta_\tau$ is the one associated with $\tau$ from Lemma \ref{rescalinglemma}.

Now, for a fixed ball $B_0=B(x_0,R_0),$ we will denote $K B_0=B(x_0,K R_0)$ for any $K >0.$ If $s=\lambda^* R_0$ and $t= (1-2 \tau \lambda^*) R_0,$ we have that $1- \frac{t}{s} = \varepsilon,$ and so \eqref{measureannulardecayproperty} yields
 \begin{equation}\label{optimalannulusmeasure}
\mu\left( \lambda^* B_0 \setminus (1-2\tau \lambda^*)B_0 \right) < c \, C_\mu^{-2}\mu(\lambda^* B_0).
\end{equation}

Suppose for the sake of contradiction, that $h_E(B_0) < \delta^* h_E(\lambda^* B_0).$ 

Let $B=B(y,r) \subset \lambda^* B_0 \setminus E$ be a ball with $ \tau \lambda^* R_0 \geq r \geq \delta^* h_E(\lambda^* B_0).$ We claim that the center $y$ of the ball $B$ does not belong to $(1- \tau \lambda^*)B_0.$ Indeed, otherwise $d(x_0,y)<(1- \tau\lambda^*) R_0,$ and the ball $B(y,r)$ is contained in $B_0 \setminus E,$ with $r \leq \tau \lambda^* R_0 < 2R_0$ and $r \geq \delta^* h_E(\lambda^* B_0) > h_E(B_0),$ contradicting the definition of $h_E(B_0).$ Thus, the point $y$ does not belong to $(1-\tau \lambda^*)B_0,$ and because $r\leq \tau \lambda^*  R_0,$ the ball $B(y,r)$ does not contain any point of $(1-2\tau \lambda^*)B_0,$ by virtue of the triangle inequality. Therefore $B \subset \lambda^* B_0 \setminus    (1-2\tau \lambda^*)B_0.$ 

Now, let $B_1,\ldots, B_N$ be pairwise disjoint balls contained in $\lambda^* B_0 \setminus E,$ satisfying the conclusion of Lemma \ref{rescalinglemma} for the ball $\lambda^* B_0$ and $\tau$ as above. By the previous argument, all these balls are contained in $\lambda^* B_0 \setminus (1-2\tau \lambda^*)B_0,$ and so we have the estimates
$$
\mu\left( \lambda^* B_0 \setminus (1-2\tau \lambda^*)B_0 \right) \geq \sum_{i=1}^N \mu(B_i) \geq c \, C_\mu^{-2} \mu(\lambda^* B_0)  .
$$
But this contradicts \eqref{optimalannulusmeasure}. We conclude $h_E(B_0) \geq \delta^*  h_E(\lambda^* B_0)$.

\medskip

Next, we show (ii). Given $L >1,$ let $n\in \N$ be so that $n-1 <\frac{\log L}{\log(\lambda^*)} \leq n$, where $\lambda^*$ is the constant from (i). Then $ L \leq (\lambda^*)^n$ and we can apply repeatedly (i) to obtain
\begin{align*}
h_E( L B_0) & \leq h_E((\lambda^*)^{n} B_0) \leq (\delta^*)^{-1} h_E((\lambda^*)^{n-1} B_0) \leq \cdots \\
&  \leq (\delta^*)^{-n} h_E( B_0) \leq (\delta^*)^{-1} L^{\frac{\log(1/\delta^*)}{\log(\lambda^*)}} h_E(B_0).
\end{align*}
\end{proof}

We conclude from Lemma \ref{lemmaannulardecayweaklyporousdoubling} that in a metric measure space $(X,d,\mu)$ with the property \eqref{measureannulardecayproperty}, the statements \ref{maintheoremannularwpcondition} of Theorem \ref{maintheoremannular} and \ref{maintheoremconditionwp} of Theorem \ref{maintheoremallspaces} are equivalent. Theorem \ref{maintheoremannular} is then a consequence of Theorem \ref{maintheoremallspaces} in this class of spaces.

\subsection{Absolute dyadic weak porosity} 
In this subsection, we only assume that $(X,d,\mu)$ is a metric space with doubling measure. The following lemma shows that if we can cover a sufficiently large portion of the measure of all the dyadic cubes $Q_0$ using the $E$-free dyadic subcubes of $Q_0$ that are admissible for the definition of dyadic weak porosity, then the maximal dyadic hole is doubling.
\begin{lemma}\label{lemmaabsolutedyadicwp}
Let $E$ be dyadic weakly porous set with constants $M\in \N,$ and $c\in (0,1).$ If $c \geq 1- C_\mu^{-1-\left\lfloor \log_2 \left( \frac{2A}{a \theta} \right) \right\rfloor},$ then $g_E$ is doubling with number $M.$ 
\end{lemma}
\begin{proof}
Let $Q_0, Q_0^* \in \mathcal{D}$ be two cubes of the same grid with $Q_0 \in \mathcal{D}(Q_0^*),$ $g(Q_0^*)=g(Q_0)-1,$ and $Q_0 \cap E \neq \emptyset.$ Suppose for the sake of contradiction, that $g_E(Q_0) > M+g_E(Q_0^*).$ Let $P_1,\ldots,P_N \in \mathcal{D}(Q_0^*)$ be $E$-free and mutually disjoint cubes with $g(P_i) \leq M + g_E(Q_0^*),$ and $\sum_{i=1}^N \mu(P_i) \geq c \mu(Q_0^*).$ Suppose that one of these cubes $P_i$ intersects $Q_0.$ Then either $Q_0 \subseteq P_i$ or $P_i \subsetneq Q_0.$ The first case is impossible, because $Q_0 \cap E \neq \emptyset = P_i \cap E.$ The second one would imply, by Lemma \ref{observationsdyadicsets}, that $P_i \in \mathcal{D}(Q_0),$ with $P_i \cap E =\emptyset,$ and so $g_E(Q_0) \leq g(P_i) \leq M + g_E(Q_0^*),$ contradicting the assumption. Consequently, we must have $Q_0 \cap P_i = \emptyset$ for every $i=1,\ldots, N,$ and thus we obtain
$$
\mu(Q_0^*\setminus Q_0) \geq \sum_{i=1}^N \mu(P_i) \geq c \mu(Q_0^*).
$$
Then $\mu(Q_0) \leq (1-c) \mu(Q^*_0) \leq C_\mu^{-1-\left\lfloor \log_2 \left( \frac{2A}{a \theta} \right) \right\rfloor} \mu(Q_0^*).$ But denoting $z_0$ and $z_0^*$ respectively the reference points of $Q_0$ and $Q_0^*,$ the chain of inclusions
$$
B(z_0,a \theta^{g(Q_0)}) \subset Q_0 \subset Q_0^* \subset B(z_0^*,A\theta^{g(Q_0^*)}) \subset B(z_0, 2A \theta^{-1} \theta^{g(Q_0)}),
$$
together with the doubling property for $\mu,$ lead us to a contradiction. We conclude $g_E(Q_0) \leq M+g_E(Q_0^*).$ On the other hand, in the case where $Q_0 \cap E= \emptyset,$ one has $g_E(Q_0)=g(Q_0)=g(Q_0^*)+1 \leq g_E(Q_0^*)+1,$ and the result follows as well.
\end{proof}

Define $b_\mu:=1- C_\mu^{-1-\left\lfloor \log_2 \left( \frac{2A}{a \theta} \right) \right\rfloor},$ the lower bound from Lemma \ref{lemmaabsolutedyadicwp}. We say that $E$ has \textit{absolute dyadic weak porosity} when $E$ is dyadic weakly porous with some constants $M\in \N$ and $c\in [ b_\mu,1).$

Lemma \ref{lemmaabsolutedyadicwp} has the following consequence: for a set $E$ with absolute dyadic weak porosity, $g_E$ is doubling. Therefore, we obtain that $\dist(\cdot , E)^{-\alpha} \in A_1(X)$ for some $\alpha>0,$ e.g., using the implication \ref{maintheoremconditionwp}$\implies$\ref{maintheoremconditionA1} in Theorem \ref{maintheoremallspaces} together with Lemma \ref{assumingdoublingdyadicwpimplieswp}. On the other hand, in Lemma \ref{dyadicimplicationsA1}(i) we showed that the $A_1$ condition for $ \dist(\cdot , E)^{-\alpha} $ implies that $E$ satisfies the condition of dyadic weak porosity with constant $c$ as close to $1$ as desired. This observation leads us to the following corollary.

\begin{corollary}\label{corollaryabsolutedyadicwp}
Let $(X,d,\mu)$ be a metric space with doubling measure, and $E \subset X$ a subset. The following statements are equivalent:
\begin{enumerate} 
\item[\normalfont{(1)}] $E$ has absolute dyadic weak porosity.
\item[\normalfont{(2)}] There exists $\alpha>0$ for which $\dist(\cdot, E)^{-\alpha} \in A_1(X)$.
\end{enumerate}
\end{corollary}

\section{Weak porosity and doubling maximal free holes are unrelated}\label{sectiondoublingholesexamples}

The purpose of this section is to show that the conditions of weak porosity and of doubling maximal $E$-free holes $h_E$ for a set $E$ are, in general, independent of each other. To do this, we present two examples in the following subsections, the latter of which is Theorem \ref{thm:exampleDoublingNotAinfty}.

\subsection{Weak porosity does not imply doubling maximal holes} 

The following example tells us that not for every weakly porous set $E$ the maximal hole $h_E$ is doubling. Combining Lemmas \ref{dyadicimplicationsA1} and \ref{assumingdoublingdyadicwpimplieswp}, this implies that not every weakly porous set $E$ has the property that $\dist(\cdot, E)^{-\alpha} \in A_1(X)$ for some $\alpha >0.$ Furthermore, we can actually prove that the weights of the form $\dist(\cdot, E)^{-\alpha}$, with $\alpha>0$, are not doubling. 

\begin{example}\label{examplenotdoublinghole}
Consider the space $X= \lbrace (x_1,x_2) \in \R^2 \, : \, x_1 x_2=0, \: x_1 \geq 0,\, x_2 \geq 0 \rbrace,$ equipped with the distance $d:=\ell_\infty \! \! \! \upharpoonright_X$, and with the measure $\mu:=\mathcal{H}^1$. Then $(X, d,\mu)$ is a connected metric space with doubling measure. In addition, it can be easily checked\footnote{This property tells us that the truth or falsity of the implication ``$E$ weakly porous $\implies h_E$ doubling'' is not necessarily related to the existence of unique radii and centers of the balls in the underlying space.} that every open ball $B \subset X$ has unique center and radius. The set $E= (\N \cup \lbrace 0 \rbrace) \times \lbrace 0 \rbrace \subset X$ satisfies the following properties.
\begin{enumerate}[label=\textbf{(\arabic*)}]
\item \label{exmetric:weakp}   $E$ is weakly porous with constants $c=1/3$ and $\delta=1/2.$ 
\item \label{exmetric:notdobhol} $h_E$ is not doubling.  
\item \label{exmetric:distnotdob} For every $\alpha \in (0,1),$ the weight $w=\dist(\cdot, E)^{-\alpha}$ is not doubling. Consequently, $\dist(\cdot, E)^{-\alpha} \notin A_p(X)$ for all $\alpha>0$ and $p\geq 1;$ see the comment subsequent to Definition \ref{d.A1}. 
\item \label{exmetric:sharpweakpconstant} If $c >1/3,$ then for every $\delta \in (0,1)$ there exists an open ball $B_0 \subset X,$ such that for any finite collection of open balls $\lbrace B_i=B(p_i,r_i) \rbrace_i$ contained in $B_0 \setminus E$ with $r_i \geq \delta h_E(B_0),$ we have $\mu \left( \bigcup_i B_i \right) < c \mu(B_0).$ Therefore, if $c>1/3,$ for no $\delta \in (0,1)$ the set $E$ is weakly porous with constants $(c,\delta). $
\end{enumerate} 
\end{example}
\begin{proof}
\item[\ref{exmetric:weakp}] There are two types of open balls $B_0=B(p,R_0) \subset X,$ $p=(p_1,p_2)\in X,$ $R_0>0$ to be checked. 
\begin{enumerate}
\item[(I)] $(0,0) \notin B_0.$ This is equivalent to saying that $R_0 \leq \max \lbrace p_1,p_2 \rbrace.$ If $p_1=0,$ then $B_0 \subset \lbrace x_1=0 \rbrace,$ $B_0 \cap E =\emptyset,$ and the condition of weak porosity trivially holds. If $p_2=0,$ then $B_0 \subset \lbrace x_2=0 \rbrace$, and the condition of weak porosity with constants $c=1/3$ and $\delta =1/2$ for $B_0$ follows by easy computations.  

\item[(II)] $(0,0) \in B_0.$ Then $R_0 > \max\lbrace p_1,p_2 \rbrace,$ and $B_0$ can be written as
$$ 
B_0=\left\lbrace 
	\begin{array}{lccl}
	  \left( (0, p_1+R_0) \times \lbrace 0 \rbrace \right) \cup  \left(\lbrace 0 \rbrace \times (0,R_0) \right) \cup \lbrace (0,0) \rbrace & \mbox{ if } &\:   p_2=0 & \\
	\left( (0, R_0) \times \lbrace 0 \rbrace \right) \cup  \left(\lbrace 0 \rbrace \times (0,p_2+R_0) \right) \cup \lbrace (0,0) \rbrace & \mbox{ if } &\:   p_1 = 0, &  
 	\end{array}
	\right.
$$
where the unions are disjoint. 

In the case where $p_2=0,$ the ball $B^*=\lbrace 0 \rbrace \times (0,R_0) \subset B_0 \setminus E $ has radius equal to $R_0/2$ and $\mu(B^*) =R_0 \geq \mu(B_0)/3.$ The value of $h_E(B_0)$ is $h_E(B_0)=\frac{1}{2}\max\lbrace R_0, \min\lbrace 1, p_1+R_0 \rbrace \rbrace$, and because $p_1<R_0,$ for any value of $h_E(B_0)$ we have that the radius $R_0/2$ of $B^*$ satisfies $R_0/2 \geq h_E(B_0)/2.$ Thus, for the balls $B_0$ in this subcase where $p_2=0,$ the weak porosity for $B_0$ is true with $c=1/3$ and $\delta=1/2.$ 

Now, when $p_1=0,$ the ball $B^*=\lbrace 0 \rbrace \times (0,p_2+R_0) \subset B_0 \setminus E$ has radius equal to $(p_2+R_0)/2$ and $\mu(B^*)=p_2+R_0 \geq \mu(B_0)/2.$ In this subcase, we have $h_E(B_0)=(p_2+R_0)/2,$ so the condition of weak porosity for $B_0$ is true with constants $c=\delta=1/2$   

\end{enumerate}

\item[\ref{exmetric:notdobhol}] For every $n\in \N,$ consider the open balls $B_n=B((n,0),n)$ and $2B_n=B((n,0),2n)$ of $X$ centered at $(n,0)$ and with radius $n$ and $2n$ respectively. Then $B_n = (0,2n) \times \lbrace 0 \rbrace$ and $2B_n= \left( (0,3n) \times \lbrace 0 \rbrace \right) \cup \left( \lbrace 0 \rbrace \times (0, 2n) \right) \cup \lbrace (0,0) \rbrace$. Clearly, $h_E(B_n)= 1/2,$ but the open ball $\widehat{B}_n=\lbrace 0 \rbrace \times (0,2n)$ is contained in $2B_n \setminus E$ and has radius equal to $n.$ Therefore $h_E(2B_n) \geq n,$ and so
$$
\lim_n \frac{h_E(2B_n)}{h_E(B_n)} = \infty.
$$
According to Definition \ref{definitionofdoubling}, $h_E$ is not doubling. 

\item[\ref{exmetric:distnotdob}] Fix $\alpha \in (0,1),$ and set $w=\dist(\cdot,E)^{-\alpha}.$ For the balls $B_n=B((0,n),n)= \lbrace 0 \rbrace \times (0, 2n)$ and $2B_n=B((0,n),2n)= \left( \lbrace 0 \rbrace \times (0, 3n) \right) \cup \left( (0,2n) \times \lbrace 0 \rbrace \right) \cup \lbrace (0,0) \rbrace,$ we have
$$
w(B_n)=\int_0^{2n} t^{-\alpha} \ud t = \frac{(2n)^{1-\alpha}}{1-\alpha},
$$
whereas
$$
 w(2B_n) \geq \int_{  (0,2n) \times \lbrace 0 \rbrace } w\ud \mu = \sum_{k=1}^{2n} \int_{k-1}^{k} \dist \left( t, \lbrace k-1,k \rbrace \right)^{-\alpha} \ud t= c(\alpha) 2n.
$$
Here $c(\alpha)>0$ is a constant depending only on $\alpha.$ Because $\alpha>0,$ $\lim_n \frac{w(2B_n)}{w(B_n)} = \infty,$ and hence $w$ is not a doubling weight. On the other hand, if $\alpha \geq 1,$ $w=\dist(\cdot, E)^{-\alpha}$ is not locally integrable (e.g., over balls centered at $(0,0)$) and thus $w$ does not even define a weight function.

\item[\ref{exmetric:sharpweakpconstant}] Fix the parameters $c > 1/3$ and $\delta\in (0,1).$ Let $\lambda >1$ be close enough to $1$ so that $\frac{\lambda}{1+2\lambda} <c,$ and $n\in \N$ be so that $ \lambda \delta n >2.$ For the ball $B_0=B( (n,0), \lambda n),$ we have $h_E(B_0)=\frac{\lambda n}{2}$, and every ball $B(p,r)$ contained in $B_0 \setminus E $ with the property that $r \geq \delta h_E(B_0)$ must be contained in $B^*_0=\lbrace 0 \rbrace \times (0, \lambda n)$. Therefore, for any finite collection of balls $B_1, \ldots,  B_N $ contained in $B_0 \setminus E$ and with radius at least $\delta h_E(B_0),$ one has
$$
 \mu \Big ( \bigcup_i B_i  \Big ) \leq \mu(B_0^*)  = \lambda n < c (1+2\lambda) n = c \mu(B_0). 
$$ 
\end{proof}

Notice that property \ref{exmetric:sharpweakpconstant} in Example \ref{examplenotdoublinghole} shows that the condition of weak porosity of $E$ is, in general, not self-improving in the sense of enlarging the parameter $c.$

\begin{remark}
It is not difficult to adapt Example \ref{examplenotdoublinghole} to a bounded space $X.$ Indeed, it suffices to consider $X = [0,1] \times [0,1],$ with the same metric and measure as above, and the set $E = \lbrace (0,0) \rbrace \cup \left( \lbrace 1/n \, : \, n \in \N \rbrace \times \lbrace 0 \rbrace \right).$ Then $h_E$ does not satisfy the doubling property for balls $B_n$ centered at $(1/(2n), 0)$ and with radius $1/(2n),$ because $h_E(2B_n) = 1/n$ and $h_E(B_n) \approx 1/n^2.$ And with similar arguments to those in Example \ref{examplenotdoublinghole}\ref{exmetric:weakp}, one can show that $E$ is weakly porous. Indeed, the proof differs only in the case where $(0,0) \notin B_0=B(p,R_0)$ where $p=(p_1,0),$ so that $p_1\geq R_0$ and $B_0 = (0, \min\lbrace 1, p_1 + R_0 \rbrace) \times \lbrace 0 \rbrace.$ The set $E$ satisfies the weak porosity condition for $B_0$ for the same reason that the set $\lbrace 1/n   \rbrace_n$ satisfies the weak porosity for any interval $(0,a),$ $a>0.$ The easy proof of this fact is left to the reader.

\end{remark}

\subsection{A non-weakly porous set with doubling maximal holes}\label{sectionbigcounterexample} We will now prove Theorem \ref{thm:exampleDoublingNotAinfty}, where we obtain a non-weakly porous set $E \subset \R$ whose maximal free holes are doubling. As a matter of fact, although the functions $\dist(\cdot,E)^{-\alpha}$, with $\alpha \in (-\infty,1)\setminus \lbrace 0 \rbrace,$ define doubling weights, they do not belong to any class $A_p(\R)$, $p \geq 1.$ The following example is a more precise formulation of the theorem. 

\begin{example}\label{exampledoublignholesdoesnotimplywp}
We endow $\R$ with the usual distance and the Lebesgue measure. There exists a set $E \subset \R$ such that:
 \begin{enumerate}[label=\textbf{(\arabic*)}]
\item\label{exbig:weight}  $w_{\alpha}:=\dist(\cdot,E)^{-\alpha} \in L_{\textrm{loc}}^1(\R)$ for $\alpha < 1, $ i.e., $w_{\alpha}$ defines a weight function.
\item \label{exbig:notAp} $E$ is not weakly porous. In fact, $w_{\alpha}\notin \bigcup_{p \geq 1} A_p(\R)$ for every $\alpha \neq 0.$
\item \label{exbig:dobholes}  $h_E$ is doubling with constant $16.$
\item \label{exbig:dobweight} $w_{\alpha}$ defines a doubling weight in $\R$ for all $\alpha <1$; meaning that $w_\alpha(K) \leq C w_\alpha(J)$ whenever $J \subset K \subset \R$ are intervals with $|K| \leq 2 |J|.$  
\end{enumerate} 
\end{example}
\begin{proof}
We construct $E$ as in \cite[Section 8]{ALMV22}, but here all the dilations will be equal to $1/2.$ Define $E_0:=\lbrace 0,1 \rbrace$ and $E_n:=E_{n-1} \cup E_{n-1}^1 \cup E_{n-1}^2,$ for every $n\in \N,$ where:
\begin{enumerate}
\item[$\bullet$] $E_{n-1}^1$ is a translation of $E_{n-1}$ dilated by $1/2$ and whose first point is the last point of $E_{n-1},$
\item[$\bullet$] $E_{n-1}^2$ is a translation of $E_{n-1}$ whose first point is the last point of $E_{n-1}^1.$ 
\end{enumerate}
Finally, we define $E^+= \bigcup_{n=0}^\infty E_n$ and $E= E^+ \cup \left( - E^+ \right).$ Here $-E^+$ is the reflection of $E^+$ with respect to the origin. We let $\mathcal{I}_n,$ $\mathcal{I}_n^1,$ and $\mathcal{I}_n^2$ denote the smallest closed intervals containing $E_n$, $E_n^1,$ and $E_n^2$ respectively, for every $n\in \N \cup \lbrace 0 \rbrace.$ A closed interval $I$ whose endpoints are two consecutive points of $E$ is called an \emph{edge} of $E$. In particular, $\textrm{int}(I) \cap E = \emptyset$ if $I$ is an edge of $E.$ The main properties of $E$ are:
\begin{enumerate}
\item[$\bullet$] For every $n\in \N \cup \lbrace 0 \rbrace$ and $k\in \lbrace 0,\ldots, n\rbrace,$ the interval $\mathcal{I}_n$ contains exactly $\binom{n}{k} 2^{n-k}$ edges of $E$ of length equal to $2^{-k}.$ In particular, each of the intervals $\mathcal{I}_n,$ $\mathcal{I}_n^1,$ and $\mathcal{I}_n^2$ has exactly $3^n$ edges of $E.$  
\item[$\bullet$] Each of the intervals $\mathcal{I}_n$ and $\mathcal{I}_n^2$ contains translated copies of the intervals $\mathcal{I}_0, \ldots, \mathcal{I}_{n-1}$ distributed in a \emph{palindromic} manner: both $\mathcal{I}_n$ and $\mathcal{I}_n^2$ contain from left to right as well as from right to left intervals $\mathcal{I}_0^* \subset \mathcal{I}_1^* \subset \cdots \subset \mathcal{I}_{n-1}^*$ that are translated copies of $\mathcal{I}_0 \subset \mathcal{I}_1 \subset \cdots \subset \mathcal{I}_{n-1}$, respectively. 
\item[$\bullet$] Each interval $\mathcal{I}_n^1$ contains from left to right as well as from right to left intervals $(1/2) \mathcal{I}_0^* \subset (1/2) \mathcal{I}_1^* \subset \cdots \subset (1/2) \mathcal{I}_{n-1}^*$ that are translated copies of $\mathcal{I}_0 \subset \mathcal{I}_1 \subset \cdots \subset \mathcal{I}_{n-1}$ dilated by $1/2.$
\item[$\bullet$] $\dist(\cdot,E)= \dist(\cdot,E_n)$ on $\mathcal{I}_n$ for every $n\in \N \cup \lbrace 0 \rbrace.$ 
\item[$\bullet$] $\lvert \mathcal{I}_n \rvert = (5/2)^n$ for every $n\in \N \cup \lbrace 0 \rbrace.$ 
\end{enumerate}

Let us prove properties \ref{exbig:weight}--\ref{exbig:dobweight}.

\item[\ref{exbig:weight}] For every $n\in \N,$ the construction of $\mathcal{I}_n$ yields 
\begin{align*}
\int_{\mathcal{I}_n} \!\!\dist(\cdot,E)^{-\alpha}    & =\int_{\mathcal{I}_{n-1}} \!\!\!\!\dist(\cdot,E_{n-1})^{-\alpha}   + \int_{\mathcal{I}_{n-1}^1}\!\! \!\!\dist(\cdot,E_{n-1}^1)^{-\alpha}   + \int_{\mathcal{I}_{n-1}^2} \!\!\!\!\dist(\cdot,E_{n-1}^2)^{-\alpha}   \\
& = (2+2^{\alpha-1}) \int_{\mathcal{I}_{n-1}}\!\!\!\! \dist(\cdot,E_{n-1})^{-\alpha}   = (2+2^{\alpha-1})^n \int_{\mathcal{I}_0}\!\! \dist(\cdot,\lbrace 0,1 \rbrace)^{-\alpha}   .
\end{align*}
Because $\mathcal{I}_0=[0,1]$ and $\alpha < 1 ,$ the last integral is finite, and this shows that $\dist(\cdot,E)^{-\alpha}$ is integrable at every interval contained in $[0, +\infty).$ By the symmetry of $E$ around the origin, the same property holds for intervals contained in $(-\infty, 0].$

\smallskip

\item[\ref{exbig:notAp}] For the sake of contradiction, suppose $E$ is weakly porous with constants $\delta,c\in (0,1).$ Fix $n_0=n_0(\delta)\in \N$ so that $\delta \geq 2^{-n_0}.$ For $n \in \N$ with $n>n_0,$ one has $h_E(\mathcal{I}_n)= 1/2$ and so any edge $J$ of $E$ contained in $\mathcal{I}_n$ with radius at least $\delta h_E(\mathcal{I}_n)$ must satisfy $|J| \geq 2^{-n_0}.$ The interval $\mathcal{I}_n$ contains exactly $\binom{n}{k} 2^{n-k}$ edges of $E$ with length $2^{-k},$ for $k=0,\ldots,n.$ Denote by $F_n$ the union of all the edges of $E$ that are contained in $\mathcal{I}_n$ and have radius at least $\delta h_E(\mathcal{I}_n).$ By the assumption of weak porosity for $E$, we must have $|F_n| \geq c |\mathcal{I}_n|$ for every $n>n_0.$ However, the previous observations permit to write the estimates
\begin{align*}
|F_n | & \leq \sum_{k=0}^{n_0} \binom{n}{k} 2^{n-k} 2^{-k} = 2^{n-n_0} \sum_{k=0}^{n_0} \binom{n}{k} 2^{n_0-k} 2^{-k} \\
&  \leq 2^{n-n_0} \sum_{k=0}^{n_0} \binom{n}{n_0} \binom{n_0}{k} 2^{n_0-k} 2^{-k}  \leq \frac{2^{n-n_0} n^{n_0}}{n_0!} \sum_{k=0}^{n_0}  \binom{n_0}{k} 2^{n_0-k} 2^{-k} = C(n_0) 2^n n^{n_0}.
\end{align*}
Together with $|\mathcal{I}_n|=(5/2)^n,$ this implies
$$
\frac{|F_n|}{|\mathcal{I}_n|}\leq C(n_0) (4/5)^n n^{n_0}, \quad \text{for every} \quad n>n_0,
$$
contradicting that $|F_n| \geq c |\mathcal{I}_n|$ for every $n\in \N.$

Let us mention that the first part of \ref{exbig:notAp} can also be proven via Theorem \ref{maintheoremannular}. Indeed, for this set $E$, a computation similar to the one in \ref{exbig:weight} permits to prove that $\dist(\cdot, E)^{-\alpha} \notin A_1(\R^n)$ for $\alpha \in (0,1],$ and so Theorem \ref{maintheoremannular} implies that $E$ is not weakly porous. In fact, using similar computations as in \ref{exbig:weight}, we next show the second part of \ref{exbig:notAp}. Let us fix $p>1$ and $\alpha \in (0,1).$ For our purposes, we can assume further $p>2$, thanks to the inclusions $A_r(\R) \subset A_s(\R)$ for $r \leq s.$ For any $\beta > -1,$ those computations lead to
$$
\int_{\mathcal{I}_n} \!\!\dist(\cdot,E)^{\beta} = (2+2^{-1-\beta})^n \int_{\mathcal{I}_0}\!\! \dist(\cdot,\lbrace 0,1 \rbrace)^{\beta}, \quad n\in \N.
$$
Applying this equality with $\beta= -\alpha, \, \frac{\alpha}{p-1}, $ and using that $|\mathcal{I}_n|= (5/2)^n,$ we obtain
\begin{equation}\label{showingnotApforeveryp}
\vint_{\mathcal{I}_n} \!\!\dist(\cdot,E)^{-\alpha}  \left( \vint_{\mathcal{I}_n} \!\!\dist(\cdot,E)^{\frac{\alpha}{p-1}}  \right)^{p-1} = C(\alpha,p) \left( \frac{(2+2^{\alpha-1})(2+2^{-\frac{\alpha}{p-1}-1})^{p-1}}{(5/2)^p} \right)^n,
\end{equation}
for every $n\in \N.$ Defining the function $h(t)=(2+2^{t-1})(2+2^{-\frac{t}{p-1}-1})^{p-1},$ $t \in [0,1)$, and using that $p>2,$ a computation shows that $h'>0$ on $(0,1)$, and thus $h(\alpha)>h(0)=(5/2)^p$ for all $\alpha \in (0,1).$ Therefore, for every $\alpha \in (0,1)$ and $p>2,$ the right-hand side term of \eqref{showingnotApforeveryp} goes to $\infty$ as $n \to \infty.$ We conclude that $\dist(\cdot, E)^{-\alpha} \notin A_p(\R)$ for every $\alpha \in (0,1)$ and $p>1.$ For negative values of $\alpha,$ it suffices to note that if we had $w_{\alpha} \in A_p(\R)$, for some $p>1,$ then also $w_{\frac{\alpha}{1-p}} \in A_{p'}(\R),$ where $\alpha/(1-p)>0$, and this is impossible by what we have just proved. 

\smallskip

\item[\ref{exbig:dobholes}] The verification of this property requires more work. We first show the doubling estimates for $h_E$ over intervals that are contained in $[0,+\infty).$

\begin{lemma}\label{doublingconditionforpositivecubes}
We have $h_E(K) \leq 16 h_E(J)$ for any two intervals $J \subset K \subset [0,+ \infty)$ with $\lvert J \rvert \geq  \frac{1}{2}\lvert K \rvert.$ 
\end{lemma}
\begin{proof}[Proof of Lemma \ref{doublingconditionforpositivecubes}]
In the case where $J$ is contained in the union of three consecutive edges $J_1, J_2, J_3$ of $E,$ the result easily follows by observing that
$$
 2 h_E(J)   = \max_{1 \leq j \leq 3} \lvert J \cap J_j \rvert \geq \frac{1}{3} \sum_{j=1}^3  \lvert J \cap J_j \rvert = \frac{1}{3}\lvert J \rvert \geq   \frac{1}{6}\lvert K \rvert \geq \frac{1}{3} h_E(K) .
$$
Now, we prove by induction on $N$ the desired estimate for every interval $K \subset \mathcal{I}_N.$ The cases $N=0,1$ are covered by the preceding argument, so let us assume that the estimate holds true for $n=1, \ldots, N-1,$ and let us verify the property for intervals $K$ contained in $\mathcal{I}_N.$

\smallskip

We may and do assume that $K$ is not contained in $\mathcal{I}_{N-1},$ as otherwise we can simply apply the induction hypothesis. Moreover, let us suppose first that $K$ is entirely contained in one of the intervals $\mathcal{I}_{N-1}^1$, $\mathcal{I}_{N-1}^2.$ In the first case, it is clear that there are translations $J^*$ and $K^*$ dilated by $2$ of the intervals $J$ and $K$ respectively, and such that $J^*\subset K^* \subset \mathcal{I}_{N-1},$ $\lvert J^* \rvert \geq \frac{1}{2}\lvert K^*\rvert,$ $ h_E(J^*)=2h_E(J)$ and $h_E(K^*)=2h_E(K).$ The induction hypothesis tells us that $h_E(K^*)\leq 16 h_E(J^*),$ and so $h_E(K)\leq 16 h_E(J)$ as well. In the latter case, $K \subset \mathcal{I}_{N-1}^2,$ we consider translations $J^*\subset K^* \subset \mathcal{I}_{N-1}$ of $J \subset K$ and apply the induction hypothesis.

\smallskip

Thus, from now on we will assume that $K$ intersects both intervals of one of the couples $\lbrace \mathcal{I}_{N-1}, \mathcal{I}_{N-1}^1 \rbrace,$  $\lbrace \mathcal{I}_{N-1}^1, \mathcal{I}_{N-1}^2 \rbrace.$ We will only study the situation where $K$ intersects both of $\lbrace \mathcal{I}_{N-1}, \mathcal{I}_{N-1}^1 \rbrace,$ as the proof in the other situation is identical, thanks to the properties of symmetry of $\mathcal{I}_N.$ 

\smallskip

$\textbf{Case 1.}$ $J$ is contained in $\mathcal{I}_{N-1}.$ Denote $K^0=K \cap \mathcal{I}_{N-1},$ $K^1=K \cap \mathcal{I}_{N-1}^1$ and $K^2=K \cap \mathcal{I}_{N-1}^2.$ The fact that $K$ intersects $\mathcal{I}_{N-1}^1$ yields the identities $2h_E(K^0)=\min\lbrace 1, \lvert K^0 \rvert \rbrace,$ $2h_E(K^1)=\min\lbrace 1/2, \lvert K^1 \rvert \rbrace,$ and $2h_E(K^2)=\min\lbrace 1, \lvert K^2 \rvert \rbrace$. And, because $K^0$ contains $J,$ and $\lvert J \rvert \geq \frac{1}{2} \lvert K \rvert,$ it is clear that $\lvert K^1 \rvert, \lvert K^2 \rvert \leq \lvert K^0 \rvert.$ This implies
$$
2 h_E(K)  =2\max\lbrace  h_E(K^0),  h_E(K^1), h_E(K^2)  \rbrace \leq \min\lbrace 1, \lvert K^0 \rvert  \rbrace = 2h_E(K^0).
$$
But $K^0$ is a cube in $\mathcal{I}_{N-1}$ containing $J$ with $\lvert J \rvert \geq \frac{1}{2} \lvert K^0 \rvert$ and so, the induction hypothesis gives $h_E(K)\leq h_E(K^0) \leq 16 h_E( J).$ 

\smallskip

$\textbf{Case 2.}$ $J$ is contained in $\mathcal{I}_{N-1}^1.$ Consider the cube $K^1=K \cap \mathcal{I}_{N-1}^1.$ If $K^1$ intersects at most three edges of $E$ within $\mathcal{I}_{N-1}^1,$ then the same property holds for $J,$ and the argument from the beginning of the proof shows that $h_E(K) \leq 6 h_E(J).$

Now, suppose $K^1$ intersects more than three edges of $E$ within $\mathcal{I}_{N-1}^1.$ By the properties of $\mathcal{I}_{N-1}^1,$ and because $K$ intersects $\mathcal{I}_{N-1},$ there must exist $ 1\leq m \leq N-2$ such that $(1/2) \mathcal{I}_m^* \subset K^1 \subset (1/2) \mathcal{I}_{m+1}^*,$ where $(1/2) \mathcal{I}_m^*$ and $(1/2) \mathcal{I}_{m+1}^*$ are translations of $\mathcal{I}_m$ and $\mathcal{I}_{m+1}$ dilated by $1/2.$ Notice that $(1/2) \mathcal{I}_{m+1}^*$ can be regarded as the union of three intervals $(1/2) \mathcal{I}_{m}^*$, $(1/2) \left( \mathcal{I}_{m}^1 \right)^*$, $(1/2) \left( \mathcal{I}_{m}^2\right)^*,$ where $\left( \mathcal{I}_{m}^1 \right)^*$ and $\left( \mathcal{I}_{m}^2 \right)^*$ are translations of $\mathcal{I}_m^1$ and $\mathcal{I}_m^2.$ Furthermore, each of the intervals $\mathcal{I}_m^*, \left( \mathcal{I}_{m}^1 \right)^*,  \left( \mathcal{I}_{m}^2 \right)^*$ can be in turn subdivided into three intervals in the following manner. Denoting by $\mathcal{I}_{m-1}^*$, $\left( \mathcal{I}_{m-1}^1 \right)^*,$ $\left( \mathcal{I}_{m-1}^2 \right)^*$ various translated copies of $\mathcal{I}_{m-1}$, $\mathcal{I}_{m-1}^1$, $\mathcal{I}_{m-1}^2$, both $\mathcal{I}_m^*$ and $\left( \mathcal{I}_m^2 \right)^*$ are the union of the intervals $\mathcal{I}_{m-1}^*$, $\left( \mathcal{I}_{m-1}^1 \right)^*,$ $\left( \mathcal{I}_{m-1}^2 \right)^*$; and $\left( \mathcal{I}_{m}^1 \right)^*$ is the union of $(1/2)\mathcal{I}_{m-1}^*$, $(1/2) \left( \mathcal{I}_{m-1}^1 \right)^*,$ $(1/2)\left( \mathcal{I}_{m-1}^2 \right)^*.$ Therefore $(1/2)\mathcal{I}_{m+1}^*$ is the union of $9$ intervals, the largest of which has length equal to $(1/2) \lvert \mathcal{I}_{m-1} \rvert$. Because $K^1 \supset (1/2) \mathcal{I}_m^*,$ we have that
$$
\lvert J \rvert \geq \frac{1}{2} \lvert K^1 \rvert \geq \frac{1}{4} \lvert \mathcal{I}_m^* \rvert = \frac{1}{4} \lvert \mathcal{I}_m \rvert = \frac{5}{8} \lvert \mathcal{I}_{m-1} \rvert > \frac{1}{2} \lvert \mathcal{I}_{m-1} \rvert.
$$
This shows that $J$ must contain either the first half or the second half of one of those $9$ intervals. By the construction of the $\mathcal{I}_j$'s, the first and the last edges of $E$ within each of those $9$ intervals have length equal to $1/2,$ $1/4$ or $1/8.$ Therefore, 
$$
2h_E(J) \geq   \frac{1}{16} \geq \frac{1}{8} h_E(K).
$$

\smallskip

$\textbf{Case 3.}$ $J$ is contained in $\mathcal{I}_{N-1}^2.$ This case is identical to $\textbf{Case 1}$.

\smallskip

$\textbf{Case 4.}$ $J$ intersects both $\mathcal{I}_{N-1}$ and $\mathcal{I}_{N-1}^1.$ If $J$ contains the last edge of $E$ within $\mathcal{I}_{N-1}$ or the first edge of $E$ within $\mathcal{I}_{N-1}^1,$ then $2h_E(J) \geq 1/2 \geq h_E(K).$ Otherwise, we have 
$$
2h_E(J)=\max \lbrace \lvert J \cap \mathcal{I}_{N-1} \rvert, \lvert J \cap \mathcal{I}_{N-1}^1 \rvert \rbrace \geq \frac{1}{2}\lvert J \rvert \geq \frac{1}{4}\lvert K \rvert \geq \frac{1}{2} h_E(K). 
$$

\smallskip

$\textbf{Case 5.}$ $J$ intersects both $\mathcal{I}_{N-1}^1$ and $\mathcal{I}_{N-1}^2. $ This case is identical to $\textbf{Case 4}$.

\end{proof}

Now, using Lemma \ref{doublingconditionforpositivecubes}, we can prove property \ref{exbig:dobholes} of Example \ref{exampledoublignholesdoesnotimplywp}. Let $J \subset K \subset \R$ with $\lvert J \rvert \geq \frac{1}{2}\lvert K \rvert.$ By Lemma \ref{doublingconditionforpositivecubes} and the symmetry of $E$ about the origin, we may assume that $K$ intersects both $(-\infty, 0]$ and $[0, +\infty).$ Let us distinguish three cases.

\smallskip

$\textbf{Case 1.}$ $J$ is contained in $[0,+\infty).$ Denote $K^+=K \cap [0,+\infty)$ and $K^-=K \cap (-\infty, 0].$ Because $J \subset K^+,$ we may apply Lemma \ref{doublingconditionforpositivecubes} to obtain $h_E(K^+) \leq 16 h_E(J).$ Also, notice that $\lvert K^- \vert \leq \lvert K^+ \rvert$ as $K^+$ contains no less than half of the measure of $K.$ This fact and the symmetry of $E$ about the origin leads us to
$$
2h_E(K^-)=\min \lbrace 1, \lvert K^- \rvert \rbrace \leq \min \lbrace 1, \lvert K^+ \rvert \rbrace = 2h_E(K^+)\leq 32 h_E(J).
$$
Consequently, $h_E(K)=\max \lbrace h_E(K^-), h_E(K^+) \rbrace \leq 16 h_E(J).$ 

\smallskip

$\textbf{Case 2.}$ $J$ is contained in $(-\infty,0].$ This is identical to $\textbf{Case 1}$, due to the symmetry of $E$ about the origin.

\smallskip

$\textbf{Case 3.}$ $J$ intersects both $(-\infty, 0]$ and $[0, +\infty).$ If $J$ contains a subinterval of $E$ of length equal to $1,$ then $h_E(J)=1$ and the result follows immediately. Otherwise, denoting $J^+=J \cap [0,+\infty)$ and $J^-=J \cap (-\infty, 0],$ we have
$$
2h_E(J)= \max \lbrace \lvert J^- \rvert, \lvert J^+ \rvert \rbrace \geq \frac{1}{2}\lvert J \rvert \geq  \frac{1}{4}\lvert K \rvert \geq \frac{1}{2} h_E(K).
$$ 

\medskip

\item[\ref{exbig:dobweight}] Fix $\alpha<1$ and denote $w=w_\alpha.$ It is clear that if we find constants $\varepsilon= \varepsilon(\alpha) \in (0,1)$ and $C=C(\varepsilon,\alpha)>0$ so that, for any two intervals $J \subset K \subset \R$, $|J| \geq (1-\varepsilon) |K|$ implies $w(K) \leq C w(J),$ then also $|J| \geq \frac{1}{2}|K|$ implies $w(K) \leq  C w(J)$ for a new constant $C,$ and all intervals $J \subset K.$ Moreover, we will only prove the estimates when $J \subset K \subset [0, + \infty),$ as the general case can be deduced using the symmetry of $E$ about the origin, and arguments similar to those at the end of the proof of \ref{exbig:dobholes}.

To prove the existence of such $\varepsilon$ and $C,$ we first observe that, by the construction of $E,$ for any two consecutive edges of $E,$ say $R_1, R_2,$ one has
\begin{equation}\label{eq:neighboursedges}
\frac{1}{2}|R_2| \leq |R_1| \leq 2 |R_2|.
\end{equation}
This takes care of the first (and non-trivial) case.

\textbf{Case 0.} Suppose that $J$ contains at most $1$ point of $E.$ Then $J \subset I_1 \cup I_2$, where $I_1$ and $I_2$ are two consecutive edges of $E.$ We claim that if $J \subset K$ and $|J| \geq (1-\varepsilon) |K|,$ then $K \subset I_0 \cup I_1 \cup I_2 \cup I_3,$ where these $I_j$ are consecutive edges of $E$ ordered from left to right. Indeed, suppose that, for instance, there is another edge of $E$ lying in the left of $I_0$ and intersecting $K.$ Then, $\frac{1}{2}|J| \leq \max \lbrace |I_1|, |I_2| \rbrace,$ and the relation \eqref{eq:neighboursedges} gives 
$$
|K| \geq |I_0|+ |J| \geq  \max\left\lbrace \frac{1}{2} |I_1|, \frac{1}{4} |I_2| \right\rbrace + |J| \geq \frac{9}{8} |J|. 
$$ 
If $\varepsilon \in (0,1)$ is small enough, the above contradicts $|J| \geq (1-\varepsilon) |K|.$ Therefore $K$ is contained in the union of (at most) $4$ consecutive edges of $E,$ and the estimate $ w(K) \leq C_0 w(J)$ holds for some constant $C_0$ depending only on $\varepsilon$ and $\alpha.$

\smallskip

Thus, from now on, we assume that $J$ contains at least two points of $E.$

\textbf{Case 1.} $K \subset \mathcal{I}_1.$ Clearly, $|J| \geq (1-\varepsilon) |K|$ implies $w(K) \leq C w(J)$ for some constant $C=C(\varepsilon,\alpha).$

\smallskip

Now suppose that for some $N \geq 2$, and all $J \subset K $ with $K \subset \mathcal{I}_{N-1}$, $|J| \geq (1-\varepsilon) |K|$ implies $w(K) \leq C w(J)$ for some constant $C=C(\varepsilon,\alpha).$ We next show that the same holds for $\mathcal{I}_N,$ of course, without increasing the constant $C$ in the iteration. We will see how large the constant $C$ should be during the proof, but it will be independent of $N.$ The induction hypothesis will be used only in \textbf{Case 2} and \textbf{Case 3} below.

From now on, we also assume that $K \subset \mathcal{I}_N,$ and $K \not \subset \mathcal{I}_{N-1}$.

\textbf{Case 2.} $K \subset \mathcal{I}_{N-1}^2.$ Up to a translation, one can assume that both $J$ and $K$ are contained in $\mathcal{I}_{N-1},$ and then apply the induction hypothesis.

\smallskip

\textbf{Case 3.} $K \subset \mathcal{I}_{N-1}^1.$ Then there are intervals $K^*$ and $J^*$ contained in $\mathcal{I}_{N-1}$ so that $K$ (resp. $J$) is a translation of $K^*$ (resp. $J^*$) dilated by the factor $1/2.$ Since $|J^*| \geq (1-\varepsilon) |K^*|,$ by the induction hypothesis, one has 
$$
w(K)=2^{\alpha-1} w(K^*)\leq 2^{\alpha-1} C w(J^*) = C w(J).
$$

\textbf{Case 4.} $\mathcal{I}_{N-1}^1 \subset J.$ In this case, we simply estimate by
$$
w(K) \leq w(\mathcal{I}_N) = \left( 1+2^{2-\alpha} \right) w(\mathcal{I}_{N-1}^1) \leq \left( 1+2^{2-\alpha} \right) w(J).
$$

\textbf{Case 5.} $J \cap \mathcal{I}_{N-1}  \neq \emptyset \neq J \cap \mathcal{I}_{N-1}^1,$ $ J \cap \mathcal{I}_{N-1}^2 = \emptyset.$ By the construction of $E,$ there are integers $n_0, n_1, m_0, m_1, m_2 \geq -1$ and translated copies $\mathcal{I}_{j}^*$ of $\mathcal{I}_j$, for $j\in \lbrace n_0, n_1, m_0, m_1, m_2 \rbrace$, so that
$$
\mathcal{I}_{n_0}^* \subset J \cap \mathcal{I}_{N-1} \subset \mathcal{I}_{n_0+1}^*, \qquad  \frac{1}{2}\mathcal{I}_{n_1}^* \subset J \cap \mathcal{I}_{N-1}^1 \subset \frac{1}{2}\mathcal{I}_{n_1+1}^*,
$$ 
$$
\mathcal{I}_{m_0}^* \subset K \cap \mathcal{I}_{N-1} \subset \mathcal{I}_{m_0+1}^*, \quad  \frac{1}{2}\mathcal{I}_{m_1}^* \subset K \cap \mathcal{I}_{N-1}^1 \subset \frac{1}{2}\mathcal{I}_{m_1+1}^*, \quad \mathcal{I}_{m_2}^* \subset K \cap \mathcal{I}_{N-1}^2 \subset \mathcal{I}_{m_2+1}^*.
$$
We understand $\mathcal{I}_j= \emptyset$ whenever $j=-1.$ Define $n= \max\lbrace n_0, n_1 \rbrace$ and $m=\max\lbrace m_0,m_1,m_2 \rbrace.$ Since we are assuming that $J$ contains at least two points of $E,$ we must have $m \geq n \geq 0.$ Let us first see that $m$ and $n$ are comparable. We clearly have
$$
|\mathcal{I}_{n}| \geq c |J \cap \mathcal{I}_{N-1}|, \quad |\mathcal{I}_{n}| \geq  c |J \cap \mathcal{I}_{N-1}^1|,
$$
an so
$$
|\mathcal{I}_n| \geq c |J| \geq  c (1-\varepsilon) |K| \geq c(1- \varepsilon) |\mathcal{I}_m|;
$$
where $c$ denotes various constants depending only on $\alpha.$ This shows that $m \leq n + C,$ for a constant $C=C(\alpha,\varepsilon)>0.$ We can now estimate $w(K)$ by
\begin{align*}
w(K)  & = w(K \cap \mathcal{I}_{N-1})+ w(K \cap \mathcal{I}_{N-1}^1) + w(K \cap \mathcal{I}_{N-1}^2) \leq C(\alpha) w(\mathcal{I}_m) \\
& \leq C(\alpha,\varepsilon) w(\mathcal{I}_n) \leq C(\alpha,\varepsilon) \max\lbrace w(J \cap \mathcal{I}_{N-1}) , 2^{1-\alpha} w(J \cap \mathcal{I}_{N-1}^1) \rbrace \leq C(\alpha,\varepsilon) w(J).
\end{align*}

\textbf{Case 6.} $\mathcal{I}_{N-1}^1 \cap J \neq \emptyset \neq J \cap \mathcal{I}_{N-1}^2,$ $ J \cap \mathcal{I}_{N-1} = \emptyset.$ This case is identical to \textbf{Case 5}, due to the palindromic distribution of the edges of $E.$

\smallskip

\textbf{Case 7.} $J \subset \mathcal{I}_{N-1},$ and $K \not \subset \mathcal{I}_{N-1}.$ Let us first see that $K \cap \mathcal{I}_{N-1}^2 = \emptyset.$ Indeed, otherwise, $K$ contains both $\mathcal{I}_{N-1}^1$ and $J$, and so
$$
|K| \geq |J| + |\mathcal{I}_{N-1}^1| = |J| + \frac{1}{2} |\mathcal{I}_{N-1}| \geq |J| + \frac{1}{2}|J| = \frac{3}{2}|J|.
$$
If $\varepsilon$ is chosen small enough from the beginning, the above contradicts that $|J| \geq (1-\varepsilon) |K|.$ Thus $K \subset \mathcal{I}_{N-1} \cup \mathcal{I}_{N-1}^1$. Again, let $m_0, m_1 \geq -1$ be integers so that 
$$
\mathcal{I}_{m_0}^* \subset K \cap \mathcal{I}_{N-1} \subset \mathcal{I}_{m_0+1}^* \quad \text{and} \quad  \frac{1}{2}\mathcal{I}_{m_1}^* \subset K \cap \mathcal{I}_{N-1}^1 \subset \frac{1}{2}\mathcal{I}_{m_1+1}^*;
$$
where $\mathcal{I}_{m_0}^*$ and $\mathcal{I}_{m_1}^*$ are translated copies of $\mathcal{I}_{m_0}$ and $\mathcal{I}_{m_1},$ and understanding that $\mathcal{I}_{-1}= \emptyset.$ Either $m_0$ or $m_1$ must be nonnegative, since $K$ contains at least two points of $E.$ Also,  
$$
| \mathcal{I}_{m_1} | \leq 2 |K \cap \mathcal{I}_{N-1}^1| \leq 2 | K \setminus J| \leq   \frac{2\varepsilon}{1-\varepsilon} |J| \leq \frac{2\varepsilon}{1-\varepsilon} |\mathcal{I}_{m_0+1}^*| = \frac{2\varepsilon}{1-\varepsilon} |\mathcal{I}_{m_0+1}|. 
$$
Thus, for $\varepsilon$ small enough (but depending only on $\alpha$), we necessarily have $m_0 > m_1.$ In particular $m_0 \geq 0,$ and
$$
w(K) = w( K \cap \mathcal{I}_{N-1}) + w(K \cap \mathcal{I}_{N-1}^1) \leq C(\alpha) w(\mathcal{I}_{m_0}).
$$
In the case where $m_0=0,$ then $|J| \geq (1-\varepsilon) |\mathcal{I}_{0}^*|$ and $J$ is contained in $\mathcal{I}_{1}^*$. So, after translating $J$, $I_0^*$ and $I_1^*,$ it is clear, e.g., by \textbf{Case 1}, that $w(\mathcal{I}_{m_0}) \leq C(\alpha,\varepsilon) w(J).$ And if $m_0 \geq 1,$ then $J$ contains a translated copy of $\mathcal{I}_{m_0-1}$ or $\frac{1}{2} \mathcal{I}_{m_0-1}$ due to the fact that $|J| \geq (1-\varepsilon) |\mathcal{I}_{m_0}|$ (and choosing $\varepsilon$ small enough). This yields
$$
w(\mathcal{I}_{m_0}) \leq C(\alpha) w(\mathcal{I}_{m_0-1}) \leq   C(\alpha) w(J).
$$

\smallskip

\textbf{Case 8.} $J \subset \mathcal{I}_{N-1}^1 \not \supset K,$ or $J \subset \mathcal{I}_{N-1}^2 \not \supset K.$ Due to the relationship between $\mathcal{I}_{N-1} ,$ $ \mathcal{I}_{N-1}^1$ and $\mathcal{I}_{N-1}^2,$ we can repeat the arguments of \textbf{Case 7} above.

\end{proof}

\begin{remark}\label{rem:generalizationR^nExample}
A set $F \subset \R^n$ with properties similar to those of $E$ from Example \ref{exampledoublignholesdoesnotimplywp} can be constructed, for example, by setting $F: = E \times \R^{n-1} \subset \R^n.$ Indeed, it suffices to note that
\begin{equation}\label{eq:distanceprojectioncoordinate}
\dist((x,x_2,\ldots,x_n), F)^\beta = \dist(x, E)^\beta, \quad \text{for all} \quad (x,x_2,\ldots,x_n) \in \R^n,\: \beta \in \R,
\end{equation}
and that $\widetilde{h_F}([a,b)^n) = h_E([a,b))$ for all $a,b\in \R,$ $a<b,$ where
$$
\widetilde{h_F}([a,b)^n) : = \sup \lbrace r>0 \: : \: \text{there exists } t\in \R \text{ with } [t-r,t+r]^n \subset [a,b]^n \setminus F \rbrace.
$$
Comparing $\widetilde{h_F}$ with $h_F$ from \eqref{definitionmaxholeballs}, we see that if $Q, B \subset \R^n$ are respectively a cube and a Euclidean ball with $Q \subset B$ (resp. $B \subset Q$), then $ \widetilde{h_F} (Q) \lesssim h_F(B)$ (resp. $h_F(B) \lesssim \widetilde{h_F}(Q)$). These remarks show that $F$ satisfies the following properties:
 \begin{itemize}
\item  $w_{\alpha}:=\dist(\cdot,F)^{-\alpha} \in L_{\textrm{loc}}^1(\R^n)$ for $\alpha < 1, $ i.e., $w_{\alpha}$ defines a weight function.
\item  $F$ is not weakly porous. In fact, $w_{\alpha}\notin \bigcup_{p \geq 1} A_p(\R^n)$ for every $\alpha \neq 0.$
\item $w_{\alpha}$ defines a doubling weight in $\R^n$ for all $\alpha <1$.
\item  $h_F$ is doubling (recall the Definition \ref{definitionofdoubling}).
\end{itemize} 
The first three properties follow from \eqref{eq:distanceprojectioncoordinate}, Fubini's Theorem, and Example \ref{exampledoublignholesdoesnotimplywp} \ref{exbig:weight},\ref{exbig:notAp},\ref{exbig:dobweight}. The fourth one is a consequence of the previous remarks on $\widetilde{h_F}$ and Example \ref{exampledoublignholesdoesnotimplywp} \ref{exbig:dobholes}. 
\end{remark}

\section{Muckenhoupt exponents and $A_1$ conditions} \label{sectionMuexponent}

The purpose of this section is to give a precise quantification for the range of those $\alpha \in \R$ such that $\dist(\cdot, E)^{-\alpha} $ belongs to the $A_p(X)$ class for a given $1\leq p<\infty.$ Previous results on this direction were obtained in \cite{ALMV22,DILTV19} in the Euclidean setting. Moreover, the concept of Muckenhoupt exponent introduced in \cite{ALMV22} for subsets $E $ of $\R^n $ allowed us to determine the mentioned range of exponents, provided $E$ is a weakly porous set in $\R^n$. Here we extend these results to metric spaces with doubling measure $(X,d,\mu).$

\begin{definition}[Muckenhoupt exponent of a set]
Let $E\subset X$. If $h_E(B(x,R))>0$ for every $x\in E$ and $R>0$,
then the
\emph{Muckenhoupt exponent} $\Mu(E)$
is the supremum of the numbers $\alpha\in \R$ 
for which there exists a constant $C$ such that 
\begin{equation}\label{e.muckdim}
\frac{\mu (E_r\cap B(x,R))}{\mu( B(x,R) )} \le C\biggl(\frac {h_E(B(x,R))}{r}\biggr)^{-\alpha}
\end{equation}
for every $x\in E$ and $0<r  <  h_E(B(x,R)) \le 2R$. 
If $h_E(B(x,R))=0$ for some $x\in E$ and $R>0$,
then we set $\Mu(E)=0$. 
 \end{definition}
 
Here and below, we denote $E_r:=\lbrace y\in X \, : \, \dist(y,E) <r \rbrace, $ for every $r>0.$ The first goal of this section is to prove, for any nonempty subset $E\subset X$, that for $\alpha>0,$
$$
\Mu(E)> \alpha \: \text{ if and only if } \: \dist(\cdot,E)^{-\alpha} \in A_1(X).
$$
This is precisely Theorem \ref{t.aikawa_assouad} below, a consequence of the next two lemmas. We will only give a sketch of the proofs, because they are almost identical to the proofs of the corresponding results in $\R^n$ from \cite{ALMV22}, if one uses the doubling condition \eqref{Cmudoublingmeasure} for the measure $\mu$ at the appropriate points.

\begin{lemma} \label{A1impliesMu}
Let $E\subset X$ be a nonempty set and let $\alpha \geq 0$  be
such that $\dist(\cdot,E)^{-\alpha}\in A_1(X)$. Then $0\le \alpha\le \Mu(E)$. 
\end{lemma}

\begin{proof}
The claim holds if $\alpha=0$, and so we may assume that $\alpha>0$. Then $h_E(B(x,R))>0$ for every $x\in E$ and $R>0$, and if $0<r < h_E(B(x,R))\le 2R$, we have
$$
\mu \left( E_r \cap B(x,R) \right) \leq r^\alpha \int_{  B(x,R)} \dist(\cdot,E)^{-\alpha} \ud \mu  \leq [w]_{A_1} r^\alpha \mu \left( B(x,R) \right) \essinf_{B(x,R)} \dist(\cdot,E)^{-\alpha};
$$
where $w=\dist(\cdot,E)^{-\alpha}.$ On the other hand, taking $y\in X$ and $s>0$ so that $h_E(B(x,R)) \geq s \geq \frac{1}{2} h_E(B(x,R))$ and $B(y,s) \subset B(x,R) \setminus E,$ we obtain
$$
 \essinf_{B(x,R)} \dist(\cdot,E)^{-\alpha} \leq \dist(y,E)^{-\alpha} \leq s^{-\alpha} \leq 2^{\alpha} \left(  h_E(B(x,R)) \right)^{-\alpha}.
$$
The two inequalities lead us to $\Mu(E)>\alpha.$ 
\end{proof}

\begin{lemma}\label{muimpliesA1}
Let $E\subset X$ be a nonempty
set 
and assume that $0\le \alpha < \Mu(E)$. Then 
$\dist(\cdot,E)^{-\alpha}\in A_1(X)$.
\end{lemma}

\begin{proof}
We follow the arguments from the proof in the Euclidean setting \cite[Theorem 6.5]{ALMV22}. We first prove the $A_1$-estimate \eqref{d.A1} for balls $B(x,r)$ centered at points $x\in E:$
\begin{equation}\label{e.suff}
\vint_{B(x,r)} \dist(y,E)^{-\alpha}\ud \mu(y)
\le  C_0(C_\mu,\lambda,\alpha,E) \essinf_{y\in B(x,r)}\dist(y,E)^{-\alpha}.
\end{equation}
Let $\lambda>0$ with $\Mu(E)>\lambda>\alpha$,
and let $x\in E$ and $r>0$.  
Observe from inequality $\Mu(E)>0$ that  $0<h_E(B(x,2r))\le 4r$. 
Hence, there is $j_0\in \Z, j_0 \geq -1$ such that 
\[
2^{-j_0} r  <  h_E(B(x,2r)) \le  2^{1-j_0} r.
\]
Define $F_j=\{y\in B(x,r) : \dist(y,E) \le 2^{1-j}r \}$ and $ A_{j}=F_{j}\setminus F_{j+1},$ for $j\ge j_0$. Reasoning as in the proof of \cite[Theorem 6.5]{ALMV22} and using that $\mu(B(x,2r)) \leq C_\mu\, \mu(B(x,r)),$ we obtain a constant $C_1=C_1(E,\lambda,C_\mu)$ such that
\begin{equation}\label{e.level_est}
\begin{split}
 \frac {\mu\left( F_{j} \right)}{\mu \left( B(x,r) \right)} \leq C_\mu  \frac {\mu\left( E_{2^{2-j}r} \cap B(x,2r) \right)}{\mu \left( B(x,2r) \right)}   \leq 
  C_1 2^{-j\lambda}\biggl(\frac {h_E(B(x,2r))}{r}\biggr)^{-\lambda}.
\end{split}
\end{equation}
We then deduce that $\mu( E ) =0$. The union of the sets $A_j$ with $j\ge j_0$ cover $B(x,r)$ up to the zero-measure set 
$\iol E\cap B(x,r)$. Integrating over these annuli $\lbrace A_j\rbrace_{j \geq j_0}$ as in \cite[Theorem 6.5]{ALMV22} and using~\eqref{e.level_est}, we obtain \eqref{e.suff}, for the points $x\in E.$

Now, if $B=B(x,R)$ is an arbitrary ball in $X$, suppose first that $\dist(B,E)<2 \diam(B).$ Let $p\in E$ be so that
$$
\dist(E,B) \leq \dist(p,B) \leq 2\dist(E,B) \leq 4 \diam(B).
$$
It is clear that $d(p,x) \leq \diam(B) + \dist(p,B) \leq 5 \diam(B) \leq 10R,$ implying $B(x,R) \subset B(p, 11 R) \subset B(x, 21 R).$ Because \eqref{e.suff} holds for $B(p,11R),$ we have
\begin{align*}
\vint_{B(x,R)} \dist(\cdot,E)^{-\alpha}\ud \mu & \le C(C_\mu,C_0) \vint_{B(p,11 R)} \dist(\cdot,E)^{-\alpha} \ud \mu \\
& \leq  C(C_\mu,C_0)  \essinf_{y\in B(p,11R)}\dist(y,E)^{-\alpha}   \leq  C(C_\mu,C_0)  \essinf_{y\in B(x,R)}\dist(y,E)^{-\alpha} .
\end{align*}
And in the case where $\dist(B,E) \geq 2 \diam(B),$ the triangle inequality gives $\dist(y,E) \leq 2 \dist(z,E)$ for any $y,z\in E,$ and so the $A_1$ estimate trivially holds over $B.$
\end{proof}

\begin{theorem}\label{t.aikawa_assouad}
Let $\alpha>0$ and let $E\subset X$ be a nonempty 
set. Then $\dist(\cdot,E)^{-\alpha}\in A_1(X)$ if and only if $\Mu(E)>\alpha$. 
\end{theorem}

\begin{proof}
If $\Mu(E) >\alpha,$ then $\dist(\cdot,E)^{-\alpha}\in A_1(X)$ by Lemma \ref{muimpliesA1}. For the reverse implication, we use the self-improvement of $A_1$ weights (see Proposition \ref{propositionselfimprovement}(1)) and Lemma \ref{A1impliesMu}.
\end{proof}

\begin{theorem}\label{t.A_p_char}
Assume that $E$ is weakly porous and $h_E$ is doubling. Let $\alpha\in\R$ and define $w(x)=\dist(x,E)^{-\alpha}$ for every $x\in X$. 
Then 
\begin{enumerate}[label=\textup{(\roman*)}]
\item $w\in A_1(X)$ if and only if $0\le \alpha < \Mu(E)$.
\item $w\in A_p(X)$, for $1<p<\infty$, if and only if 
$$
(1-p)\Mu(E) < \alpha < \Mu(E).
$$
\end{enumerate}
\end{theorem}

\begin{proof} 
By Theorem \ref{maintheoremallspaces}, there exists $\alpha>0$ for which $\dist(\cdot,E)^{-\alpha} \in A_1(X).$ Taking into account the self-improvement properties of Proposition \ref{propositionselfimprovement}, the proof is identical to the one for the corresponding result in $\R^n$ from \cite[Theorem 1.2]{ALMV22}.
\end{proof}

Finally, if $X$ satisfies the annular decay property \eqref{measureannulardecayproperty}, the assumption that $h_E$ is doubling can be done away with in Theorem \ref{t.A_p_char}.

\begin{corollary}\label{t.A_p_charannulardecay}
Suppose that $(X,d,\mu)$ satisfies \eqref{measureannulardecayproperty}, and let $E \subset X$ be weakly porous. Let $\alpha\in\R$ and define $w(x)=\dist(x,E)^{-\alpha}$ for every $x\in X$. 
Then 
\begin{enumerate}[label=\textup{(\roman*)}]
\item $w\in A_1(X)$ if and only if $0\le \alpha < \Mu(E)$.
\item $w\in A_p(X)$, for $1<p<\infty$, if and only if 
$$
(1-p)\Mu(E) < \alpha < \Mu(E).
$$
\end{enumerate}
\end{corollary}
\begin{proof}
Since $(X,d,\mu)$ satisfies \eqref{measureannulardecayproperty}, for every weakly porous set $E$, the maximal $E$-free hole function $h_E$ is doubling by virtue of Lemma \ref{lemmaannulardecayweaklyporousdoubling}. Now, it suffices to apply Theorem \ref{t.A_p_char}.
\end{proof}

\bibliographystyle{abbrv}

\end{document}